\documentclass[11pt,a4paper]{article}
\usepackage{graphicx} % Required for inserting images
\usepackage{amsmath}
\usepackage{amsfonts}
\usepackage{dsfont}
\usepackage{amssymb}
\usepackage{stmaryrd}
\usepackage{bbm}
\usepackage{mathtools} 
\usepackage{enumitem} % To get arabic (1), (2), ... numerotation for enumerate.
\usepackage[most]{tcolorbox}
\usepackage{amsthm} % or thmtools
\usepackage[
  a4paper,           % base paper size
  margin=1in,        
]{geometry}
\usepackage{tikz}
\usetikzlibrary{arrows.meta}
\usepackage{float}
\usepackage[colorlinks=true,linkcolor=blue!50!black,citecolor=blue!50!black,
            urlcolor=blue!50!black]{hyperref}
\usepackage{mathrsfs}

\usepackage{natbib}

\theoremstyle{plain} % italic text, bold header
\newtheorem{theorem}{Theorem}[section]
\newtheorem{lemma}[theorem]{Lemma} % shares counter with theorem

\newtheorem{definition}[theorem]{Definition}

\theoremstyle{remark} % small italic "Remark" label, upright text body

\newcommand{\bdOne}{\mathds{1}}

\newcommand{\bdE}{\mathbb{E}}

\newcommand{\bdN}{\mathbb{N}}

\newcommand{\bdP}{\mathbb{P}}

\newcommand{\bdR}{\mathbb{R}}

\newcommand{\bcA}{\mathcal{A}}

\newcommand{\bcC}{\mathcal{C}}

\newcommand{\bcE}{\mathcal{E}}
\newcommand{\bcF}{\mathcal{F}}

\newcommand{\bcN}{\mathcal{N}}

\newcommand{\bcP}{\mathcal{P}}

\newcommand{\bcX}{\mathcal{X}}

\newcommand{\bfX}{\mathbf{X}}

\newcommand{\brd}{\mathrm{d}} 

\newcommand{\dotp}[2]{\langle #1,\,#2\rangle}
\newcommand{\p}[1]{\left( #1 \right)}
\newcommand{\br}[1]{\left[ #1 \right]}
\newcommand{\grad}[1]{\nabla#1}
\newcommand{\hess}[1]{\nabla^{2}#1}
\newcommand{\norm}[1]{\left\lVert#1\right\rVert}
\newcommand{\discrete}[2]{\llbracket#1,#2\rrbracket} % Discrete intervals
\newcommand{\sumn}{\sum_{i=1}^n} 
\DeclarePairedDelimiter{\abs}{\lvert}{\rvert}
\newcommand{\indep}{\perp\!\!\!\perp}
\newcommand{\Var}[0]{\operatorname{Var}} 

\newcommand{\set}[1]{\left\{#1\right\}} % adaptative size {}

\newcommand{\tr}[0]{\operatorname{Tr}}

\renewcommand{\le}{\leqslant} % overswrite the \le for better looking pdfs
\renewcommand{\ge}{\geqslant}

\makeatletter
\DeclareRobustCommand{\est}{\@ifnextchar\bgroup{\est@witharg}{\est@nobrace}}
\newcommand{\est@witharg}[1]{\hat{#1}_n}
\newcommand{\est@nobrace}[1]{\hat{#1}_n}
\makeatother

\DeclareMathOperator*{\argmin}{arg\,min}
\DeclareMathOperator*{\argmax}{arg\,max}

\makeatletter
\DeclareRobustCommand{\rest}{\@ifnextchar\bgroup{\rest@witharg}{\rest@nobrace}}
\newcommand{\rest@witharg}[1]{\ensuremath{\hat{#1}_{n,\lambda}}}
\def\rest@nobrace#1{\ensuremath{\hat{#1}_{n,\lambda}}}
\makeatother
\newcommand{\ind}[1]{\bdOne_{\set{#1}}}
\newcommand{\eps}[0]{\varepsilon}
\newcommand{\Span}[0]{\operatorname{Span}}
\title{Rapport}
\newcommand{\str}{\sqrt{\tr(\Sigma)}}
\newcommand{\ntr}{\sqrt{\norm{\Sigma}}}
\newcommand{\tga}{\tr(\Sigma)}
\newcommand{\nga}{\norm{\Sigma}}

\author{
  Matthieu Lerasle \footnotemark[1]\\
  \and
  Antoine Roche \footnotemark[1]\\
}
\title{Finite sample behavior of the maximum likelihood estimator in the Poisson model under Gaussian design}

\usepackage{natbib}
\begin{document}

\maketitle

\footnotetext[1]{Department of Statistics, CREST/ENSAE Paris, Palaiseau, France}

\begin{abstract}
    We study the maximum likelihood estimation of the coefficient $\beta_*$ in well-specified Poisson regression; 
    Using tools from empirical process theory and random conic geometry, we show that the probability of existence of the maximum likelihood estimator (MLE) exhibits a sharp phase transition at the threshold $n \gtrsim d$.
    We then determine a minimum threshold $e^{\norm{ \beta_ * } ^ 2 } d$ on the sample size $n$ to guarantee with high probability an excess risk of the asymptotic order $d/n$. 
    We reveal the existence of an intermediate regime in Poisson regression, when $d \lesssim n \lesssim e^{\norm{ \beta_ * } ^ 2 } d$, where the MLE exists but does not achieve the optimal rate $d/n$. 
    We close the gap between the two regimes up to a $d^\eps$ term with $\eps \in (0,1)$ by providing an upper bound on the distance between the MLE and $\beta_*$ whenever $n \gtrsim d^\eps$. 
    Along the way, we provide two generalizations of well-known PAC-Bayes inequalities regarding sub-Gamma random vectors and sub-Gamma random matrices that are of independent interest and that we use extensively to prove the main results of the present paper. 
\end{abstract}

\section{Introduction and problem formulation}
\label{section:section1}

The non-asymptotic study of high-dimensional regression models has become a prominent body of work. In particular, generalized linear models (GLMs) occupy an important place among these methods, \cite{bach2009selfconcordantanalysislogisticregression, bachostrovskiifinite, kuchelmeister2024finite, chardon2026finitesampleperformancemaximumlikelihood, mourtada2022exact}. Among GLMs, Poisson regression \cite{greene2003econometric, Poissonglm} is a foundational tool used to model count data and contingency tables, finding numerous applications across econometrics, engineering, biology, social sciences etc. due to its relatively simple nature and interpretable results \cite{myers2012generalized}. However, while recent breakthroughs have established sharp non-asymptotic bounds for GLMs like logistic regression \cite{chardon2026finitesampleperformancemaximumlikelihood} alongside with some results on M-estimation \cite{bachostrovskiifinite}, Poisson regression remains unexplored. Specifically, the empirical gradient and Hessian of the loss lack a Laplace transform under a standard Gaussian design. Because of this unboundedness, standard empirical process tools and generalized self-concordance arguments fail to provide tight, explicit bounds in the finite-sample regime, making results on non-asymptotic Poisson regression conspicuously sparse in the statistical literature. In this work, we bridge this gap by investigating the finite-sample performance of the maximum likelihood estimator in well-specified Poisson regression under a Gaussian design, providing sharp, non-asymptotic conditions that explicitly quantify the complex interplay between sample size, ambient dimension, and signal strength. This simple model proves to already be adversarial and highlights the core features of the model in the non asymptotic setting. 

\subsection{Problem setting and main questions}

We first define the Poisson model formally. Given $d \ge 1$ a dimension, the Poisson model is the family of conditional distributions on the outcome $y \in \bdN$ given the covariates $x \in \bdR ^d $ defined by:
\begin{equation}
    \bcP_{\text{poisson} } = \set{ p _ \beta, \ \beta \in \bdR ^ d}, \quad \text{where} \quad p_{\beta} ( y \mid x ) = e^{ - e^{ \dotp{\beta}{x} }}\frac{e^{y \dotp{\beta}{x}}}{y!},
\end{equation}
for $(x,y) \in \bdR ^d \times \bdN $ and where $\dotp{\cdot}{\cdot}$ denotes the usual dot product on $\bdR^d$. $(X,Y)$ is said to follow the Poisson model if the distribution of $Y$ given $X$ belongs to $\bcP_{\text{poisson}}$, that is, there exists $\beta _ * \in \bdR^d$ such that the distribution of $Y$ given $X$ has density $p_{\beta _ *}$ with respect to the counting measure. In the statistical setting, while $P$, the distribution of $(X,Y)$ is unknown, one has access to i.i.d. observations $((X_1, Y _ 1), \dots, (X_n, Y _ n ))$ generated from $P$. Using this sample, one can define the MLE as the maximizer of the conditional likelihood or, equivalently, the minimizer of the negative log-likelihood by (assuming existence)
\begin{equation}\label{eq:losses}
    \est \beta =  \argmax _ {\beta \in \bdR ^d } \prod _ {i = 1 } ^n p _ {\beta } ( Y _ i \mid X_i) = \argmin_{\beta \in \bdR ^d } \frac{1}{n} \sumn \br{ e ^{ \dotp{\beta}{X_i} } - Y _ i \dotp{\beta}{X_i} }.
\end{equation}

The two questions that will be of interest to us are the following:
\begin{enumerate}
    \item \textit{Existence:} When does the MLE exist?
    \item \textit{Performance:} When the MLE exists, how accurate is it? 
\end{enumerate}
Let us discuss first why those two questions are of interest to us. 

\quad Regarding the existence, Santos Silva and Tenreyro \cite{SANTOSSILVA2010310} mention an important example in dimension 1 where the MLE does not exist using a distribution of $(X , Y)$ generating the data $((X_1, Y_1), \dots, (X_n, Y_n))$ defined in the following way,
\begin{equation}
	\forall i \in \discrete{1}{n}, \quad
    \begin{cases}
        X_i = 0 \quad \text{if} \quad Y_i > 0 \\
        X_i \le 0 \quad \text{if} \quad Y_i = 0.
    \end{cases}
\end{equation}
In this case, splitting the loss on $\p{ i : y _ i = 0 }$ and $\p{ i : y _ i > 0}$ and picking $ b \in \bdR_+^*$, evaluating the loss at $t b $ and driving $t \to + \infty$ shows the absence of minimizer for the negative log-likelihood \eqref{eq:losses}. More generally, this example highlights the geometric nature of the question of existence in higher dimension, indeed, for any $b \in \bdR^d$, and observations $((X_1, Y_1), \dots, (X_n, Y_n))$, we say that $b$ \textit{Poisson separates} the data if
\begin{equation}
	\forall i \in \discrete{1}{n}, \quad
    \begin{cases}
        X_i^\top b = 0 \quad \text{if} \quad Y _ i > 0 \\
        X_i^\top b \le 0 \quad \text{if} \quad Y _ i = 0.
    \end{cases}
\end{equation}

Assume for simplicity that $X_1, \dots, X_n$ span $\bdR^d$. Then the MLE exists \textit{if and only if} there exists no $b_* \ne 0$ that \textit{Poisson separates} the data. Indeed, if such a $b_* \ne 0$ exists, then the negative log-likelihood in \eqref{eq:losses} stays bounded along the ray $t b _ * $ as $t \to + \infty $. Since a strictly convex function admitting a global minimizer must diverge at infinity, the loss admits no minimizer and the MLE fails to exist. Conversely, if no such $b_*$ exists, the coercivity of $\est L (\beta)$ guarantees existence via a standard compactness argument. Thus, while asymptotically the MLE is well known to exist with probability converging to $1$, in the finite sample case the question is purely of geometrical nature, and is closely related to phase transitions in random convex optimization programs. We refer to the seminal paper on the subject by Amelunxen et al. \cite{amelunxen2014livingedgephasetransitions} for more discussion on the subject along with numerous examples and to \cite{koriyama2025phasetransitionsexistenceunregularized} for an application of some of the results from Amelunxen et al. to the case of M-estimation in specific asymptotics. 

\quad Regarding our second question on performance, evaluating the accuracy of the estimator requires a choice of risk metric. We consider the prediction risk under the Poisson loss, which corresponds to the negative log-likelihood for the Poisson model \cite{McCullagh1989, cameron2013regression}. For $\beta \in \bdR ^d$ and $(x, y) \in \bdR ^d \times \bdN$ we define the Poisson loss $\ell$,
\begin{equation}\label{eq:poissonloss}
    \ell(y, t) = e^t - yt \quad \text{so that} \quad - \log p _ \beta ( y \mid x ) = e^{ \dotp{ \beta }{ x} } - y \dotp{ \beta }{ x } = \ell(y, \dotp{ \beta }{ x } ),
\end{equation}
omitting the $\log y!$ term that does not depend on $\beta$. Note that the Poisson loss is strictly convex in $t$ and is also coercive if $y \ne 0$. The population risk $L$ is then defined as,
\begin{equation}
    L(\beta) = \bdE _ {(X,Y) \sim P} \br{ \ell (Y, \dotp{\beta}{X} )}. 
\end{equation}
Under mild conditions on the support of $X$, a global minimizer of $L$ exists in Poisson regression $\beta _ * \in \argmin _ { \beta \in \bdR ^d } L (\beta ) $, this minimiser matches the definition of $\beta_*$ as the truth when the model is well-specified. Finally, as the distribution $P$ of $(X,Y)$ is unknown, the population risk at $\beta \in \bdR^d$ is estimated by its empirical counterpart $\est L (\beta) $,
\begin{equation}
    \est L (\beta) = \frac{1}{n} \sumn \ell( Y_i, \dotp{\beta}{X_i})= \frac{1}{n} \sumn \br{ e ^{\dotp{ \beta}{X_i } } - Y _ i \dotp{\beta}{X _ i }},
\end{equation}
which coincides with the negative log-likelihood of the model. With these definitions at our disposal, the excess risk of an arbitrary estimator $\est \beta $ is given by $L(\est \beta) - L(\beta_* )$. In this article we aim to establish both non-asymptotic existence conditions and explicit upper bounds on the excess risk with clear dependences in $\norm{ \beta _ *}, d ,n$. Motivating this approach is our aim to contribute toward addressing two current gaps within the non-asymptotic literature. On the one hand, while the finite-sample scaling of the sample size $n$ and ambient dimension $d$ has been intensively documented across a wide variety of regression models, a corresponding non-asymptotic framework for Poisson regression remains surprisingly limited; to the best of our knowledge, a formal finite-sample analysis of the unregularized Poisson MLE has not yet been established.  On the other hand, even within the frameworks of related generalized linear models where finite sample risk bounds have been obtained, explicit tracking of the signal strength $\norm{\beta _ *}$ that is as sharp as possible in its dependence is not as common, the magnitude $\norm{\beta _ *}$ being either normalized entirely or treated as a bounded constant. Consequently, precisely quantifying the structural interactions between the signal strength, the ambient dimension, and the sample size constitutes a very rich and interesting problem. Let us motivate this further by drawing from the recent GLM literature. While not in the finite sample setting, regarding the existence of the maximum likelihood estimator in logistic regression, Candès and Sur \cite{candes2020phasetransitionlogistic} showed that in the "high dimensional asymptotic regime", that is, $d, n \to \infty $ with $d/n \to \gamma \in  (0,1)$, if $X \sim \bcN(0,I_d)$ and the model is well-specified, that is  $\bdP( Y = 1 \mid X ) = \sigma( \dotp{ \beta _ * } { X } ) $ where $\sigma$ is the logistic function, then there exists a nonincreasing function $h : \bdR _ + \to (0,1)$ such that if the dataset consists of $n$ i.i.d. observations $((X_1, Y_1), \dots, (X_n, Y_n))$ from the model, then
\begin{equation}\label{eq:candesfunction}
    \lim _ {n \to \infty } \bdP (\text{MLE exists}) = \begin{cases}
        0 \quad \text{if} \quad \gamma > h ( \norm{ \beta _ * } ) \\
        1 \quad \text{if} \quad \gamma < h (\norm{ \beta _ * }),
    \end{cases}
\end{equation}
hence, in the Gaussian logistic case, the existence of the MLE exhibits a sharp phase transition in the high dimensional asymptotic regime dictated by the signal strength $\norm{ \beta _ *}$. In their work on finite-sample logistic regression, Chardon, Lerasle and Mourtada \cite{chardon2026finitesampleperformancemaximumlikelihood} give, in particular, a fully non-asymptotic result in the well-specified Gaussian case that can be seen as a quantitative version of the result of Candès and Sur, which can be stated informally as follows, let $\norm{ \beta _ * } \ge 3$, then if $n \lesssim \norm{ \beta _ * }d$, the probability that the MLE exists decreases exponentially quickly with the dimension $d$. On the other hand, if $n \gtrsim \norm{ \beta _ * } (d + t)$, then the MLE $\est \beta$ exists with probability at least $1 - e^{-t}$ and satisfies,
\begin{equation}\label{eq:riskboundlogistic}
    L(\est \beta) - L(\beta _ *) \lesssim \frac{d + t}{n},
\end{equation}
where $L$ is the logistic population risk. A first interesting aspect of the result is the compounded effect of the dimension and signal strength on the non-asymptotic probability of existence of the MLE, as illustrated by the factor $\norm{ \beta _ * } d $ in the condition. Indeed, in strong signal regimes, that is, $\norm{ \beta _ * } \gg 1$, many of the labels $Y_i$ have the same sign as $\dotp{\beta_*}{X_i}$ which increases the probability of existence of a nonzero vector linearly separating the dataset, whose existence is a necessary and sufficient condition for the non-existence of the MLE in logistic regression. Hence, in the regime $d \ll n \ll \norm{ \beta _ * } d$, linear separation does not occur deterministically due to the dimension anymore but occurs with high probability due to the signal strength. On the other hand, the dimension $d$ ``increases the degrees of freedom'' for a potential separating vector, thus, in the regime $\norm{ \beta _ *} \vee d \ll n \ll \norm{ \beta _ * }d$, linear separation happens with high probability due to the combination of the signal strength effect and dimension effect, as pointed out in the original paper. This relates intuitively to the fact that the function $h$ in \eqref{eq:candesfunction} is nonincreasing in $\norm{ \beta _ * }$, conveying the idea that if the signal strength is stronger, then the problem must not be ``too high dimensional" for the MLE to exist, which translates to a smaller value of $\gamma$, the limit of $d/n$. Another remarkable aspect of the result in \cite{chardon2026finitesampleperformancemaximumlikelihood} is that the number of observations $n$ needs to be of order $\norm{ \beta _ * } d$ to guarantee the existence of the MLE with high probability in the well-specified Gaussian setting, this order is sharp and is also precisely the order of $n$ necessary to guarantee (with high probability) that the MLE satisfies the asymptotic excess risk of order $d/n$. Indeed, the existence condition is a problem of purely geometric nature, regarding the existence of a separating hyperplane with high probability, or more generally, if two particular random cones intersect. The problem of obtaining bounds on the excess risk however, is closely related to localization of the maximum likelihood estimator and hence, the control of the gradient and Hessian of the model with high probability, which; in the case of the present paper (and also in the case of \cite{chardon2026finitesampleperformancemaximumlikelihood}), reduces mostly to the control of the associated empirical processes. Because of the fundamentally different nature of the two problems, the two thresholds on the sample size $n$ to guarantee either high probability existence or that the MLE satisfies the asymptotic excess risk bound with high probability do not have any incentive to match, which makes the logistic regression more of an exception rather than the rule. In Poisson regression the sample size needed to guarantee the existence of the MLE with high probability does not match the one needed to guarantee an asymptotic-type risk bound of order $d/n$, this causes the existence of an intermediate regime, in which the MLE does exist with high probability but does not satisfy the asymptotic risk bound. To conclude the discussion on classification problems - and motivate further why studying the signal strength in GLMs is interesting - we mention an adjacent case of binary classification. In this setting, Kuchelmeister and Van de Geer \cite{kuchelmeister2024finite} study the Probit model and the importance of signal to noise ratio (SNR) in a non-asymptotic setting. Precisely, i.i.d. variables $(x_i, y_i)_{i \in \set{1, \dots, n}}$ are observed with $y _ i = sign(x_i^\top \beta _ * + \sigma \eps _ i)$ where $x_i$ is a standard Gaussian vector and $\eps _ i $ is a standard gaussian random variable independent of $x_i$. In the paper, $\beta _ *$ is normalized to have norm $1$ leaving $\sigma > 0$ as a free parameter corresponding to the noise level, this leads to the same model as if $\beta _ *$ were a free parameter and $\sigma$ were normalized to $1$, as for identifiability purposes the SNR $1/\sigma$ determines the model. In informal terms, the main result of their paper is that, in this case, the SNR level drives the estimation of both the direction generating the data $\beta _ *$ and the SNR $\tau _ * $ itself by defining two different regimes. 
In the ``large noise" case, where $\sigma \gtrsim \frac{d \log n }{ n }$, they obtain that if $\hat \gamma$ is the logistic regression estimator for $\beta _ * / \sigma$, that is, 
\begin{equation*}
    \hat \gamma = \argmin _ { \gamma \in \bdR ^d } \frac{1}{n} \sumn \log ( 1 + \exp( - y _ i \dotp{ \gamma }{ x _ i } )). 
\end{equation*}
Then,
\begin{equation}\label{eq:direction-vs-norm}
    \norm{ \hat \beta - \beta _ * }_2 \lesssim \sqrt{ \sigma \frac{d \log n }{ n } } \quad \text{and} \quad \abs{ \hat \tau - \tau _ * } \lesssim \sqrt{ \frac{1}{ \sigma ^ 3 } \frac{ d \log n }{ n }},
\end{equation}
where $\hat \beta = \hat \gamma / \norm{ \hat \gamma } $ estimates the direction generating the data and $\hat \tau = \norm{ \hat \gamma } $ estimates the SNR. Moreover, in the ``small noise" regime, i.e. $\sigma \lesssim \frac{d \log n }{n}$ they obtain a fast rate of order $\frac{d \log n }{ n }$ for the classifier $\hat \beta$, which is the same as the noiseless case ($\sigma = 0$), up to log terms, this rate is known to be optimal. The order of the SNR compared to the dimension and sample size thus drives the estimation regime of the MLE for classification. Moreover, as the authors point out, the form of the bounds in \eqref{eq:direction-vs-norm} shows the interplay between classification and SNR estimation, the higher the noise is, the more difficult it is to classify (estimate $\hat \beta$) but the easier it is to estimate the SNR. While insightful as to why the signal strength is a key parameter in maximum likelihood estimation in GLMs, and why we should consider it as a free parameter alongside the dimension and the sample size, the comparison between logistic regression or the probit model with Poisson regression have quick limitations. Indeed, the variable $Y$ in Poisson regression is unbounded and the Poisson model has no clear incidental effects interpretation unlike the logit or probit models. Moreover, Poisson regression differs from the aforementioned examples by the behaviour of the empirical gradient and Hessian of the loss. For any $\beta \in \bdR^d$, the empirical Hessian of the logistic loss at $\beta$ has the form, 
\begin{equation*}
    \hess \est L (\beta ) = \frac{1}{n} \sumn \sigma'( \dotp{ \beta }{ X _ i } ) X _ i X _ i ^\top,
\end{equation*}
where $\sigma (x) = \frac{1}{ 1 + e ^{ - x } }$ is the logistic function. $\sigma' = \sigma ( 1 - \sigma )$ are stochastic weights that are very small with high probability when $\norm{ \beta _ *}$ is ``large". In this sense, when the model is well-specified, the logistic Hessian is shrunk by the signal strength. More broadly, with a less subtle approach, the weights $\sigma'$ can be upper bounded crudely by $1$ and, under some tail assumptions on the design $X$, the empirical Hessian is not particularly heavy tailed in any direction (for instance, it is sub-Gamma in any direction $v \in S ^{d - 1 }$ if we assume a standard Gaussian design). In standard linear regression, the empirical Hessian is given by $\frac{1}{n} \sumn X _ i X _i ^\top$ which does not depend on $\beta$. In Poisson regression, the empirical Hessian of the loss at $\beta$ is given by 
\begin{equation}\label{eq:hessian_poisson}
    \hess \est L (\beta) = \frac{1}{n} \sumn e ^{\dotp{ \beta }{ X _ i }} X _ i X _i ^\top,
\end{equation}
The random variables $e^{\dotp{ \beta }{ X } } X X ^\top$ have no Laplace transform due to these exponential weights when $X$ is taken to be standard gaussian. This lack of control on the moments of the gradient and Hessian is, in our opinion, the main cause as to why the literature on non-asymptotic Poisson regression is sparse. For instance, Bach and Ostrovskii \cite{bachostrovskiifinite} study a large class of M-estimators whose loss function satisfy a self-concordance condition \cite{nesterov1994interior, bach2009selfconcordantanalysislogisticregression}, but assume also that the gradient and Hessian of this loss admits sub-Gaussian tails, this prevents us from applying their result, even if the Poisson loss is self-concordant. Similarly, Spokoiny \cite{spokoiny2012parametric} proposes a general approach for the problem of finite sample estimation, even under a misspecified model, but requires exponential moments conditions on the gradient of the model that are not satisfied in our setting. In the following, and to motivate the proof scheme of the non-asymptotic approach, one should keep in mind how the excess risk is proven asymptotically. Denote $H = \hess L (\beta_*)$ the Hessian of the population loss at the truth, by a Taylor expansion on the empirical risk, the MLE is first localized in $H$-norm controlled by $\| H^{-1}\grad \est L( \beta_ * ) \|^2$ which behaves nicely  in the limit due to the central limit theorem. The excess risk $L(\est \beta) - L(\beta _ *)$ can then be expanded similarly and controlled on the localization radius which yields the classical asymptotic excess risk bound. While very elegant, the CLT approach naturally cannot be applied in the finite-sample setting, in particular, the ``natural" norm of the gradient $\| H^{-1} \grad \est L( \beta_ * ) \|^2$ and empirical Hessian need to be treated differently to be compared to asymptotic quantities, for instance via concentration inequalities to determine when they can be compared to their expectations. An approach for the problem of finite sample estimation, which is the one described by Spokoiny \cite{spokoiny2012parametric} is the following: under some mild regularity conditions on the model, a localization set $\Theta_G$ \eqref{eq:localization_set} is defined in the ``natural" metric induced by the variance of the loss gradient at the population minimizer $G = \bdE \br{ \grad \ell (\beta_*) \grad \ell(\beta _ *) ^\top }$ for some radius $r > 0$ by
\begin{equation}\label{eq:localization_set}
    \Theta_G(r) = \set{ \theta \in \bdR ^d, \norm{ G^{1/2} (\theta - \theta _ *) } \le r}.
\end{equation}
Note that in the case where the model is well-specified, and under some regularity conditions, $G = H = \hess L (\beta _ *)$. The MLE is then shown to belong to $\Theta_G(r)$ with controllable (high) probability by controlling the deviations of the empirical process $\est L (\beta) - \est L(\beta _ *)$ uniformly over this localization set using tools from empirical process theory \cite{talagrand2014upper, boucheron2013concentration}. The second step of the asymptotic approach is then replaced by upper bounding the population risk $L (\beta) - L(\beta _ *)$ uniformly over the localization set by another quadratic form, which allows to transfer the localization of the MLE to a risk bound using a Taylor expansion. Among other methods, we also mention the self-concordance approach used by Bach and Ostrovskii \cite{bachostrovskiifinite, bach2009selfconcordantanalysislogisticregression} that allows one to obtain risk bounds under different sets of hypotheses. 

\subsection{Existing results}

Having motivated the problem, we now provide an overview of the already existing results related to the questions we consider about the MLE in Poisson regression. Once again, we emphasize that while the literature surrounding GLMs in general is well established, specific results about Poisson regression in the finite sample regime are sparse, in particular due to the poor concentration behavior of the quantities of interest in the model, as we pointed out previously. 

\subsubsection{Classical asymptotics}

In the case of classical parametric asymptotics, the distribution $P$ of the data, the signal strength $\norm{ \beta _ *}$ and the dimension $d$ of the problem are fixed while the number of observations $n$ goes to infinity. In this case, the behavior of the MLE $\est \beta$ is well-understood \cite{vanderVaart1998asymptotic}. In particular, under mild regularity conditions on the model, the MLE exists with probability converging to one and is asymptotically normal with asymptotic covariance the inverse Fisher information matrix \cite{vanderVaart1998asymptotic, cam1990asymptotics}. In particular, it is consistent and converges towards $\beta_*$ at rate $1/ \sqrt n$. Furthermore, expanding the risk up to the second Taylor order around $\beta_ * $, combined with the asymptotic normality of the MLE allows to derive the asymptotic distribution of the rescaled excess risk, precisely, 
\begin{equation}\label{eq:cvdisrisk}
    2 n (L (\est \beta) - L (\beta _ * ) ) \xrightarrow[n \to \infty]{(d)} \chi^2(d),
\end{equation}
in the well-specified case, where by $\chi^2(d)$ we denote the chi-squared distribution with $d$ degrees of freedom. This result, known as \textit{Wilks' theorem} implies that the excess risk of the MLE converges to $0$ at rate $1/n$. Combining this with a standard deviation bound for the chi-squared distribution \cite{boucheron2013concentration}, we have for any $\delta > 0$
\begin{equation}\label{eq:hpineq}
    \liminf _ {n \to \infty } \bdP \p{ L(\est \beta ) - L (\beta _ * ) \le \frac{d + 2 \log (1/ \delta)}{n} } \ge 1 - \delta.  
\end{equation}
While the high probability upper bound \eqref{eq:hpineq} is sharp due to the convergence of distribution of the excess risk \eqref{eq:cvdisrisk}, it should be noted that the result is purely asymptotic and that $\norm{ \beta _ *}, d$ are fixed. Moreover, it does not provide any information about how large $n$ should be (in terms of $\norm{ \beta _ * }, d, \delta$) for such an inequality to hold. Note that \eqref{eq:hpineq} justifies the expression ``asymptotic risk bound" we used in the introduction. In the econometrics literature, Poisson regression has been studied intensively in the asymptotic regime, consider the following misspecified Poisson model defined by 
\begin{equation*}
    Y \mid \eps , X \sim \bcP (\exp \p{ \dotp{ \beta _  *}{ X } + \eps } )
\end{equation*}
where $\eps$ is a random variable that captures the specification error of the model to account for explanatory variables of $Y$ that are independent of $X$. In this model, only the vector $X$ of exogenous variables is observed, so that $Y \mid X$ is not necessarily Poisson due to the unobservable $\eps$ random error term. Gourieroux, Monfort and Trognon \cite{gourieroux} showed that the Poisson quasi-maximum likelihood estimator (qMLE) still estimates consistently the true parameter $\beta _ *$ and is still asymptotically normal in a much wider setting provided that the conditional mean is still correctly specified and the errors $\eps$ have a finite second moment. 

\subsubsection{Non-asymptotic setting}

Beyond the high-dimensional asymptotic regime, one may ask for finite-sample guarantees that quantify how the existence and accuracy of the MLE depend jointly on the sample size $n$, the ambient dimension $d$ and the signal strength $\norm{ \beta _ * }$. As mentioned, while the question was recently addressed for logistic regression in \cite{chardon2026finitesampleperformancemaximumlikelihood}, as far as we know, no such guarantees have been obtained for Poisson regression. In the econometrics literature however, some attempts have been made to obtain expression for the bias and mean squared error (MSE) of the MLE. For example \cite{poisson_bias, poisson_bias_mse} apply results from \cite{RILSTONE1996369} to obtain expressions for the bias and the MSE up to terms of order $O(\text{poly}(n^{-1}))$, making the results still somewhat of an asymptotic nature. They then propose unbiased versions of the estimator up to terms of order $O(\text{poly}(n^{-1}))$ and provide numerical simulations to test the performance of those new estimators.

\subsection{Aim of the paper and outline}

Concretely, the aim of the paper can be broken down in three simple questions, the third one emerging naturally from the first two,

\begin{enumerate}
    \item When does the MLE exist?
    \item How large should $n$ be, as a function of the other parameters of the problem for the excess risk to be of the asymptotic order $d/n$ with high probability? That is, for the excess risk to satisfy a bound of the form \eqref{eq:hpineq}. 
\end{enumerate}

It turns out that, informally, the MLE in the Poisson model exists with high probability whenever $n \gtrsim d$, while the asymptotic risk bound is attained for $n \gtrsim T(\norm{\beta_*},d)$ much larger, $T$ denoting a threshold depending on $\norm{\beta_*},d$. In consequence, there exists an ``intermediate" regime $d \lesssim n \lesssim T(\norm{\beta_*},d)$ where the MLE exists with high probability, yet, the excess risk does not satisfy the asymptotic risk bound of order $d/n$. This leads to the third question, 

\begin{enumerate}
    \setcounter{enumi}{2}
    \item What can we say about the MLE in the intermediate regime? 
\end{enumerate}

The paper is organized as follows, Section \ref{section:section2} presents the three main theorems of the paper, Theorem \ref{th:mainth_existence} answers the question about existence of the MLE, Theorem \ref{th:mainth_riskbound} answers the second question regarding the sample size needed to guarantee an upper bound on the excess risk of the asymptotic order $d/n$. Finally, Theorem \ref{th:mainth_riskbound_smalln} provides guarantees on the MLE in the intermediate regime, answering the third and last question. Section \ref{section:proofschemes} presents the main lemmas used to prove the results of the present paper, jointly with a proof scheme for each one, together with two new PAC-Bayes type inequalities (Section \ref{section:pacbayes}) that we use extensively in the proofs to control both the empirical gradient and Hessian and that generalize some already well-known inequalities. In particular, the localization lemma is stated and proven. This lemma reduces the proof of Theorems \ref{th:mainth_riskbound} and \ref{th:mainth_riskbound_smalln} in two main components: a high probability upper bound on the norm of the empirical gradient at $\beta _ *$ in the appropriate geometry, and a high probability lower bound on the empirical Hessian in a neighborhood of the true parameter $\beta _ *$, which corresponds to the localization set. Those two main ingredients structure the proofs of the risk bounds, as, in both cases, the empirical Hessian is first bounded from below uniformly over the localization set, after that, the empirical Hessian is upper bounded at $\beta _ *$ (and not uniformly this time), this upper bound on the empirical Hessian is then used to control the norm of the empirical gradient at $\beta _ *$. Those results, combined with a deterministic upper bound on the true Hessian, allow the convex localization lemma to conclude. Section \ref{discussion} closes the main body by discussing the two assumptions under which our results are obtained, locating where each of them enters the proofs and to what extent they can be relaxed. Appendix \ref{app:lemmas} provides the proofs for the main lemmas stated in Section \ref{section:proofschemes}, while Appendix \ref{app:theorems} provides the proofs for the 3 main theorems of the paper, by applying all the previously shown results. Appendix \ref{appendix:pacbayes} recalls the generic PAC-Bayes inequality and provides the proofs of the two PAC-Bayes inequalities, and Appendix \ref{appendix:technical} gathers additional technical lemmas and their proofs, which are not central for the understanding of the main ideas.

\subsection{Notations}

Throughout the paper, $n$ and $d$ are positive integers that denote respectively the sample size and the dimension of the ambient space. We define $B = \max(2, \norm{ \beta _ * })$ the signal strength, where $\norm{ \cdot }$ is the standard Euclidean norm on $\bdR^d$, $\dotp{ \cdot }{ \cdot }$ denotes the usual inner product on $\bdR^d$. For $x \in \bdR ^d$ and $r \ge 0$, we denote by $B(x, r)$ the Euclidean ball centered at $x$ with radius $r$, $S^{d - 1}$ denotes the unit sphere in $\bdR^d$, i.e. the set of elements $u \in \bdR^d$ with unit norm. When comparing two quantities $a,b$, we denote $a \gtrsim b$ when there exists a universal constant $C$ such that $a \ge Cb$. We write $a \asymp b$ when $a \gtrsim b \gtrsim a$. The relation $\succeq$ denotes the Loewner order, i.e. the partial order on the set of (symmetric) positive semidefinite matrices such that $B \succeq A$ if $B - A$ is positive semidefinite. In the same fashion, we denote $B \succ A$ if $B - A$ is positive definite. For a (symmetric) positive semidefinite matrix $A$, we denote its associated semi-norm by $\norm{ \cdot } _ A ^2$ defined for $x \in \bdR ^d$ by $\norm{ x } _ A ^2 = \dotp{ A x }{ x } = \norm{ A^{1/2} x}^2$. For $x \in \bdR ^d$ and $r \ge 0$, we define respectively the ball $B_A(x, r)$ and the sphere $S_A(x,r)$ induced by the semi-norm associated to the matrix $A \succeq 0$. The operator norm of a matrix $A$ is denoted by $\norm{A}$ and its trace by $\tr(A)$. For two matrices $A, B$ of compatible dimension, we define $\dotp{A}{B}_F = \tr(A^\top B)$ their Frobenius inner product, where $A^\top$ denotes the transpose of $A$. If $f : \bdR ^d \to \bdR$ is a twice continuously differentiable function, we let $\grad f(x) \in \bdR ^d$ and $\hess f(x) \in \bdR ^{d \times d}$ be respectively its gradient and Hessian at $x \in \bdR^d$. For two real numbers $a,b$ we denote $a \vee b = \max(a,b)$ the maximum between these two and $a \wedge b = \min(a,b)$ the minimum between these two. Throughout the paper, for $k$ any integer, $C_k$ and $c_k$ denote numerical constants whose value does not depend on any parameter of the problem, their exact value can be found in the proof of the corresponding result. We say that the model is well-specified if the distribution of $Y$ given $X$ belongs to $\bcP_{\text{poisson}}$, that is, there exists $\beta _ * \in \bdR ^d$ such that $Y$ given $X$ follows a Poisson distribution with parameter $\lambda = \exp( \dotp{ \beta _ * }{ X } )$, otherwise, we say that the model is misspecified. 

\section{Main results}
\label{section:section2}

In this section, we provide the precise statements of our main results on Poisson regression. To make this section as clear as possible, we do not provide the exact constants here, however, they can all be recovered from the proofs of each particular result. 

\quad \textbf{Assumptions:} In all the three main theorems, we always assume that the design $X \sim \bcN(0, I_d)$ is Gaussian and that the model is well-specified. Those assumptions, and potential relaxations will be discussed in Section \ref{discussion}. Note that, due to the absence of regularization and by invariance under rotation of the Gaussian distribution, setting the identity matrix as the covariance matrix is without loss of generality, hence, the results encompass the case where $X \sim \bcN(0, \Sigma)$ for an arbitrary covariance matrix $\Sigma \succ 0$. 

\quad Theorem \ref{th:mainth_existence} below provides a sharp condition (up to universal constant factors) on the sample size $n$ required for the maximum likelihood estimator to exist with high probability. Its proof can be found in Appendix \ref{app:theorems}. 

\begin{theorem}[Non asymptotic high probability existence condition for the maximum likelihood estimator]\label{th:mainth_existence} When $n < d$, the maximum likelihood estimator does not exist almost surely. Moreover, there exists a universal constant $C_1$ such that for any $z > 0$, whenever
\begin{equation}
    n \ge C_1 (d \vee z), 
\end{equation}
the maximum likelihood estimator exists with probability at least $ 1 - 6\exp(-z)$. 
\end{theorem}

It follows from Theorem \ref{th:mainth_existence} that the condition $n \gtrsim d$ is both necessary and sufficient to guarantee that the MLE exists with high probability. Several aspects of this result are interesting to comment on. Comparing this to the analogous result in Gaussian logistic regression, both models exhibit a sharp phase transition regarding the probability of existence of the MLE, at thresholds $d$ and $Bd$ respectively, omitting constants. The case of Poisson regression is remarkable in the sense that the order of the existence threshold is completely independent of the signal strength $B$, unlike logistic regression. This independence leads to a transition in the existence of the MLE that is even steeper for the Poisson model. Indeed, the threshold for existence coincides exactly (up to constants) with the threshold where linear separation no longer occurs deterministically. This results in a sudden shift from almost sure non-existence to existence with controllable high probability. In comparison, while the probability of non-existence of the MLE in Gaussian logistic regression decreases exponentially fast in the dimension for sample sizes below the threshold $Bd$ \cite[Theorem 2]{chardon2026finitesampleperformancemaximumlikelihood}, it does not exhibit this sudden ``drop to zero" behavior.

\quad While the signal strength has no effect on the order of the existence threshold for the MLE in Poisson regression, its influence remains present in more subtle ways. Informally, in Poisson regression, the sample size needed to guarantee existence with high probability is driven by the number of zero-count observations $\abs{Z}$. This dependence arises because the Poisson loss, defined in \eqref{eq:poissonloss} is coercive only for $y>0$; otherwise, it is simply increasing. As pointed out by Koriyama and Bellec in their paper \cite{koriyama2025phasetransitionsexistenceunregularized}, the phase transition phenomenon directly stems from this lack of coercivity. If the empirical loss lacks a sufficient number of coercive terms, it becomes possible to find a non-zero direction $b_*$ that is orthogonal to all covariates $x_i$ associated with $y_i > 0$, then, moving the parameter along the ray $t b_*$ bypasses the coercivity constraint and the loss becomes bounded from below but infinitely decreasing or flat, effectively preventing the existence of a minimizer. As the sample size increases, the probability that a vector $b_*$ that is orthogonal to all coercive datapoints $x_i$ with $y _ i > 0$ while still satisfying the negativity constraint $b_* ^\top x _ i \le 0$ for the $x_i$ with $y_i = 0$ exists decreases rapidly. Indeed, up to constants, and conditionally on $\abs{Z}$, the condition for existence is given by:
\begin{equation}\label{expr:cond_exis}
    n \gtrsim d + \abs{Z}.
\end{equation}
Expression \eqref{expr:cond_exis} can be interpreted as requiring the sample size $n$ to be large enough to overcome the dimension effect, which prevents existence deterministically by linearly separating the data, and as requiring that enough samples are present to overcome the ``lack of coercivity" of the loss. This discussion should now make the effects of the signal strength on the problem of existence more apparent. Indeed, $\abs{Z}$ is a binomial random variable with parameters $(n,p(b))$, where $p(b)$ is given by $\mathbb{E}\br{\mathbb{P}(Y=0|X)} = \mathbb{E}\br{\exp(-\exp(bg))}$ with $b=\norm{\beta _ *}$, $g \sim \bcN(0,1)$. Notice that $p(b)$ is strictly increasing in $b$, with $p(0)=1/e$ and $p(b)\rightarrow1/2$ as $b\rightarrow\infty$. As $\abs{Z}$ is binomial, it exhibits rapid concentration around $n p(b) \asymp n$. Hence, unlike in logistic regression where $B$ contributed linearly to the existence threshold, in Poisson regression, the contribution is only of constant order. 

\begin{theorem}[Behavior of the MLE in the large $n$ regime]\label{th:mainth_riskbound} There exists numerical constants $C_{2}$ and $C_{3}$ such that, whenever there exists $\eps \in (0,1)$ satisfying
\begin{equation}\label{cond:largen}
    n \ge C_2 \exp \p{ \frac{B ^2 }{ 2 } (1 + \varepsilon ^{- 1 }) } \p{ (d + z )e^{4 B \sqrt{ z } } }^{1 + \varepsilon}, 
\end{equation}
the excess risk satisfies
\begin{equation}
    L( \est \beta) - L (\beta _ * ) \le C_3 \frac{ d + z }{ n },
\end{equation}
with probability at least $1 - 6\exp(-z)$.
\end{theorem}

Note that the restriction of $\eps$ to $(0,1)$ is to ensure a simpler sample size condition, $\eps > 0$ can be chosen arbitrarily if $n \ge e^{B ^2}(d+z)$ is also assumed, restricting $\eps$ serves to automatically guarantee this condition. 

\quad The main result of Theorem \ref{th:mainth_riskbound} is that whenever the sample size condition \eqref{cond:largen} is satisfied, the MLE achieves the same risk guarantees (up to a constant) as in the asymptotic case where $B$ and $d$ are fixed. A feature hinted by this condition is a form of tradeoff between the signal strength $B$ and the dimension $d$. Specifically, the sample size requirements can be relaxed with respect to $d$ if we accept a worse dependence on $B$, and vice versa. 

\quad A second notable feature of the sample size condition is its large dependence on $B$. To understand why this dependence is structural rather than a mere artifact of the proof, it is helpful to examine the proof strategy. To localize the MLE and bound the excess risk, we must establish a high-probability lower bound on the empirical Hessian that holds uniformly over the localization set (see Section \ref{section:proofschemes} for a detailed proof scheme). In particular, we seek to bound the empirical Hessian from below by the matrix $H = e^{B ^2 /2}(B ^2 u _ * u _ * ^\top + I_d)$ where $u_* = \beta _ * / \norm{\beta _ *}$ which serves as a proxy for the true Hessian at $\beta_{*}$. Due to the form of the empirical matrix (given in \eqref{eq:hessian_poisson}), obtaining enough curvature ensuring $\hess \est L (\beta_*) \succeq H$ requires at least one exponential weight to be of order $e^{B^{2}/2}$. This implies we need an observation where $\dotp {u_*}{X_i} \asymp B$ which occurs with probability smaller than $e^{-B^2/2}$. Consequently, to observe at least one such sample, we intuitively need $n \gtrsim e^{B^{2}/2}$. Controlling the Hessian from above requires the (even stronger) condition $n\gtrsim e^{B^2}$. This exponential term is therefore structurally necessary to guarantee the localization of the MLE, a way to understand this scale is to interpret it as the sample size needed for the law of large numbers to take effect. Due to the exponential weights being so heavy tailed, two competing effects occur, on one hand, the empirical Hessian is driven by the rare events $ \dotp{ u _ * }{X } \asymp B$, requiring the large sample size $n \gtrsim e^{B^ 2 / 2}$ to guarantee with high probability that enough of those rare events are observed to get the correct scale for the lower bound. On the other hand, to get the correct scale from for the upper bound, the sample size has to be large enough $n \gtrsim e^{B ^2}$ to smooth out the potential very large outliers that can strongly inflate the empirical Hessian from above. 

\begin{theorem}[Behavior of the MLE in the intermediate regime]\label{th:mainth_riskbound_smalln} Let $\eps \in (0,1)$ and define $\delta = \frac{2 \eps}{ \eps + 2} \le 1$. There exists numerical constants $C_4, K_\eps$, with $K_\eps$ depending only on $\eps$ such that, whenever 
    \begin{equation}\label{cond:intermediate}
        n \ge K _ \eps d^{1 + \eps} \exp(8z),
    \end{equation}
    then
    \begin{equation}\label{eq:loc_bound_smalln}
        \norm{ \est \beta - \beta _ * } \le C_4 e^{ B ^2/ 2 \delta } \sqrt{\frac{ d + z }{ n }},
    \end{equation}
    with probability at least $1 - 7\exp(-z)$. Moreover, the excess risk satisfies a bound of the form:
    \begin{equation}
        L (\est \beta) - L (\beta _ * ) \le f(B, \delta) \frac{d + z}{n}.
    \end{equation}
\end{theorem}

The exact expressions for $f$ and $K_\eps$ can be found in \eqref{expr:f_expression} and Section \ref{def:k_eps} respectively.

\quad Theorem \ref{th:mainth_riskbound_smalln} aims to bridge the sample size gap by providing formal statistical guarantees for the MLE in the intermediate regime, specifically for $d^{1+\epsilon} \lesssim n$ up to the threshold established in Condition $\eqref{cond:largen}$, as in the general setting where the signal strength $B$ is allowed to be of arbitrary order compared to $d, n$, the gap between the existence condition and condition \eqref{cond:largen} can be arbitrarily large. Conversely, if $B$ is treated as a bounded constant the theorem recovers the same asymptotic rate conclusions as Theorem \ref{th:mainth_riskbound}. 

\quad For smaller sample sizes, the random fluctuations of the empirical gradient and Hessian can be very large, and are exacerbated by the magnitude of $B$. This prevents a precise localization of the MLE. Consequently, this loss of precision directly translates into an amplified dependence on $B$ within the localization bound \eqref{eq:loc_bound_smalln} and the subsequent excess risk bound.

\quad Finally, the additional $d^{\eps}$ dependence in the thresholds of Theorem \ref{th:mainth_riskbound} and Theorem \ref{th:mainth_riskbound_smalln} introduces complications. If $B$ is treated as a constant, the leading exponential factor in \eqref{cond:largen} is absorbed into the universal constant. However, this still leaves a gap in the regime $d\lesssim n\lesssim d^{1+\eps}$, where the MLE exists with high probability but lacks formal performance guarantees. Despite our best efforts, we were unable to fully close this gap, and the $\eps$ term remains. Establishing a final risk guarantee that is fully linear in $d$ is an interesting open problem, and solving it would provide a satisfactory closure to the study of Poisson regression under a Gaussian design.

\section{Proof scheme and main lemmas}

\label{section:proofschemes}

As several parts of the proofs involve technical arguments, we begin by outlining the overall proof strategy for Theorem \ref{th:mainth_riskbound} and Theorem \ref{th:mainth_riskbound_smalln} emphasizing the main ideas. We then state the main lemmas on which the proofs of the main results rely; the aim of this section being to guide the reader through what is really central to the arguments, we postpone the proof of the lemmas to Appendix \ref{app:lemmas}.

\subsection{Convex localization}

The following lemma, which is very classical \cite{chardon2026finitesampleperformancemaximumlikelihood, spokoiny2012parametric, geoffrey2020robust}, is at the core of the paper. It is used to prove both existence of the MLE and an upper bound on its risk. The approach is based on a convex localization argument and is fully deterministic and generic, in the sense that its use is not restricted to the Poisson model. The only conditions required for the results of the lemma to hold, aside from the assumptions of the lemma itself, are that $\est L$ is convex, that $\est L$ and $L$ are twice continuously differentiable and that $\beta_*$ is a global minimizer of $L$.

\begin{lemma}\label{lemma:localization}
    Assume that there exists a positive-definite matrix $H \in \bdR ^{ d \times d}$ and real numbers $r_0, c_0, c_1, \nu$ such that the following conditions hold: 
    \begin{itemize}
        \item For every $\beta \in \bdR ^d $ such that $\norm{ \beta - \beta _ * } _ {H } \le r_0$, one has $\hess \est L (\beta  ) \succeq c_ 0 H$; 
        \item $\norm{ \grad \est L (\beta _ * ) } _ { H ^{- 1 }} \le \nu$;
        \item For every $\beta \in \bdR ^d $ such that $\norm{ \beta - \beta _ * } _ {H } \le r_0$, one has $\hess L (\beta ) \preceq c_ 1 H$.
    \end{itemize}
    If $\nu < c_0 r_0 / 2$ then the empirical risk $\est L $ admits a unique minimizer $\est \beta$ which satisfies 
    \begin{equation}\label{eq:riskbound_and_loc}
        \norm{ \est \beta - \beta _ * } _ H \le \frac{2 \nu }{ c _ 0 }, \quad L(\est \beta ) - L (\beta _ * ) \le \frac{2 \nu ^2 c _ 1 }{ c _ 0 ^2 }. 
    \end{equation}
\end{lemma}

\begin{proof}[Proof of Lemma \ref{lemma:localization}]
    Take $\beta \in S _ H (\beta _ *, r)$, we have 
    \begin{align}
        \est L (\beta ) - \est L (\beta _ * ) &\ge \dotp{ \grad \est L (\beta _ * ) }{ \beta - \beta _ * } + \frac{1}{2} \inf _ { \beta ' \in B_H( \beta _ * , r) } \norm{ \beta - \beta _ * } _ {\hess \est L (\beta ' ) } ^ 2 \notag \\
        & \ge - \norm{ \grad \est L (\beta _ * ) } _ { H ^{ - 1 } } \norm{ \beta - \beta _ * }_ H  + \frac{ c_ 0 }{ 2 } \norm{ \beta - \beta _ * }_H ^2 \notag \\
        & \ge - \nu r + \frac{c_ 0 r^2}{2} > 0 \quad \text{ if $ \nu < c_0 r / 2 $}. \label{eq:positive_lb}
    \end{align}
    Now, for any $\beta' \in \bdR ^d$ such that $\norm{ \beta' - \beta _ * }_H = r ' \ge r$, the parameter $\beta = (1 - t) \beta _ *+ t \beta'$ with $t = r/r' \in (0,1]$ satisfies $\norm{ \beta - \beta _ *}_H = r$, hence by convexity of $\est L$, 
    \begin{equation*}
        (1-t) \est L ( \beta _ * ) + t \est L (\beta ') \ge \est L ((1-t) \beta_* + t \beta ') = \est L (\beta) > \est L (\beta _ *),
    \end{equation*}
    by \eqref{eq:positive_lb}. This simplifies to $\est L (\beta') > \est L (\beta_*)$, we deduce, $\inf _ {\beta \in \bdR^d} \est L (\beta) = \inf_{ \beta \in B_H(\beta _ *, r)} \est L (\beta)$. By continuity of $\est L$, a compactness argument ensures the existence of the MLE. Moreover, due to the first assumption, $\est L$ is strictly convex on the set $\set{\beta \in \bdR ^d, \|\beta - \beta _ * \| _ H \le r_0}$ thus admits a unique minimizer $\est \beta$ satisfying $\| \est \beta - \beta _ * \| _ H \le r$. Since the result holds for any $r \in (2 \nu/ c _ 0, r_0)$, we deduce $\| \est \beta - \beta _ * \| _ H \le 2 \nu / c _ 0$. Finally, the excess risk bound \eqref{eq:riskbound_and_loc} is deduced from the localization as for any $\beta \in \set{\beta \in \bdR ^d, \|\beta - \beta _ * \| _ H \le r_0}$, $L(\beta) - L(\beta _ *) \le \frac{c _ 1}{2} \|\beta - \beta _ * \| ^2 _ H$ as $\grad L (\beta _ *) = 0$ and $\hess L (\beta) \preceq c_1 H$ over this domain. 
\end{proof}

Lemma \ref{lemma:localization} is very interesting as it reduces the question of existence of the MLE and of obtaining an upper bound on its excess risk into two concrete steps. While there is no constraint on the matrix $H$ to be used in principle, a natural choice to obtain sharp results is to take this matrix such that $\hess L (\beta _ *) \asymp H$ i.e. equivalent up to constant factors to the Hessian at the population minimizer. This is the approach for the proof of Theorem \ref{th:mainth_riskbound} (see the definition of $H$ given in \eqref{def:definition_matrix_h} and the expression of $\hess L (\beta)$ given by Lemma \ref{lemma:form_of_hessian}). This choice, combined with purely algebraic inequalities allows to determine the largest possible radius $r_0$ such that the inequalities of Lemma \ref{lemma:localization} could be expected to hold. This determines the localization set (or a set that contains it) by, 
\begin{equation*}
    \Theta = \set{ \beta \in \bdR^d, \norm{ \beta - \beta _ * }_H \le r_0}. 
\end{equation*}

It remains to obtain:
\begin{itemize}
    \item A high probability upper bound on the $H^{-1}$-norm $\| \grad \est L(\beta _ *) \| _ { H ^{ - 1 }}$ of the empirical gradient at $\beta _ *$;
    \item A high probability lower bound $\hess \est L (\beta) \succeq c_0 H$ on the Hessian at $\beta$ that holds uniformly over all $ \beta \in \Theta$ (or a subset thereof). 
\end{itemize}

The condition for the existence of the MLE is then given by $\nu < c _ 0 r / 2$, due to a continuity and compactness argument on the localization set. Moreover, the risk bound and localization of the MLE are given by \eqref{eq:riskbound_and_loc}.

\quad This convex localization lemma gives a strong structure to the proof of Theorem \ref{th:mainth_riskbound}, where this exact approach is used. Indeed, Lemma \ref{mainlem:hessian_lowerbound} provides the uniform lower bound on the empirical Hessian, while Lemma \ref{mainlem:gradient_bound} provides the upper bound on the $H^{-1}$-norm of the gradient at $\beta _ *$. This last bound relies crucially on Lemma \ref{mainlem:hessian_upperbound}, which, while not directly necessary for the localization lemma, is used to control from above the variance of the empirical gradient at $\beta_*$, given by the empirical Hessian at that same point. The combination of these results yields the desired deviation bound on the gradient. The localization lemma then concludes, the first deterministic upper bound on the Hessian being given by Lemma \ref{lemma:deterministic_upperbound}. 

\quad To prove Theorem \ref{th:mainth_riskbound_smalln}, while close in spirit, the approach is slightly modified as in the intermediate regime with $n \gtrsim d^{1 + \eps}$ but still not large enough to satisfy condition \eqref{cond:largen}, a uniform lower bound on the empirical Hessian of order $H$ can no longer be guaranteed with high probability; an upper bound on the empirical Hessian of order $H$ cannot be guaranteed anymore either. Instead, the empirical Hessian is compared from below to the identity matrix over all $\bdR^d$ (Lemma \ref{mainlem:hessian_lowerbound_small}), which corresponds to a much more modest lower bound, however this bound now holds uniformly over any ball of any arbitrary radius. The gradient is then controlled in a much stronger $H_0^{-1}$ norm, with $H_0^{-1} \preceq H^{-1}$ to obtain the same order for its upper bound (Lemma \ref{mainlem:gradient_bound_lown}). This is due to the deviations being potentially much larger in this smaller sample size regime. Combining those results, we localize the MLE in a much larger ball, the bound on the excess risk can then be deduced by an elementary deterministic upper bound on the true Hessian uniformly over this new localization set. 

\subsection{Lower bounds on the empirical Hessian}

This section provides the first component of the localization lemma, namely, a high probability lower bound on the Hessian of the empirical risk,
\begin{equation}\label{empirical_hessian_recall}
    \hess \est L (\beta) = \frac{1}{n} \sumn e^{ \dotp{ \beta }{ X _ i } } X_ i X _ i ^\top,
\end{equation}
uniformly over a neighborhood $\Theta$ of $\beta _ *$. It follows from Lemma \ref{lemma:deterministic_upperbound} that an ideal guarantee would be of the form: for $n$ large enough, 
\begin{equation*}
    \hess \est L (\beta) \succeq c_0 H \quad \text{for any $\beta \in B_H(\beta _ *, r e ^ {B ^2 / 4})$},
\end{equation*}
with probability at least $1 - \exp(-z)$ and $r$ of constant order. 

\quad As mentioned previously in this section, this guarantee is only possible in the large sample size regime, given by \eqref{cond:largen}, which is the needed scale for the law of large numbers to take effect, as the empirical Hessian is strongly driven by the rare events $\dotp{u _ * }{ X } \asymp B$. 

\quad When $n$ is of order potentially much smaller than the threshold given by \eqref{cond:largen}, we instead bound the empirical Hessian from below by the much smaller $I_d$, while much milder, this lower bound holds uniformly over $\bdR^d$ with high probability, and not only over the localization set $\Theta$. This important fact allows us to still localize the MLE. 

\quad We start by giving the corresponding lower bound on the empirical Hessian in the large $n$ regime, then provide the analog lower bound in the intermediate sample size regime. In both cases, we discuss informally the main ideas of the proof. We refer to Appendix \ref{app:lemmas} for the full proofs of these results. 

\subsubsection{Large sample size regime}

Recall the definition of $H$ 
\begin{equation}\label{def:definition_matrix_h}
    H = e^{ B ^2 / 2 } ( B ^2 u_* u_*^\top + I _ d).
\end{equation}

\begin{lemma}\label{mainlem:hessian_lowerbound}
    There exists numerical constants $C_5, C_6$ such that with probability at least $1 - \exp(-z)$, for any $\beta \in B _ H ( \beta _ *, \frac{1}{2} \exp(B ^2 / 4))$
    \begin{equation}
        \hess \est L (\beta ) \succeq C_6 H,
    \end{equation}    
    whenever $n \ge C_5 e^{B ^2 / 2}B^2(Bd + z)$, with $C_5 \le 10^{30}, C_6 > 3 \times 10 ^{- 1 0 }$. 
\end{lemma}

The proof of this result relies on the reduction of the problem to the control of an empirical process from below, thanks to Lemma \ref{lemma:empirical_process_reduction}. We denote this process $\est P$ to ease the reading. We then further decompose $\est P = E + \tilde P _ n$, where $E$ is a deterministic term, capturing the expectation of the process, while $\tilde P _ n $ is the centered version of the process, capturing its fluctuations around its mean (the exact expressions for those two terms are given in equation \eqref{eq:first_term} and equation \eqref{eq:second_term} respectively). As shown by Lemma \ref{lemma:bound1}, the expectation term $E$ has the correct order in $H$, consequently, the core of the proof revolves around controlling the fluctuations of the process around $E$. Precisely, we want to show that the size of those deviations are negligible compared to $E$, provided that the sample size $n$ is large enough. This result is given by Lemma \ref{lemma:emp_bound}, which shows using Vapnik-Chervonenkis (VC) arguments that the deviations of $\tilde P _ n$ are indeed negligible compared to $E$, provided that the sample size is of the necessary order for the law of large numbers to apply, as discussed in Section \ref{section:section2}. We now state the analog result in the intermediate sample size regime. 

\subsubsection{Intermediate sample size regime}

\begin{lemma}\label{mainlem:hessian_lowerbound_small}
    There exists numerical constants $C_7, C_8$ such that with probability at least $1 - 2\exp(-z)$, for all $\beta \in \bdR^d$
    \begin{equation}
        \hess \est L (\beta ) \succeq C_7 I_d,
    \end{equation}    
    whenever $n \ge C_8 (d \vee z)$, with $C_7 > \frac{1}{100}, C_8 \le 3.7\times 10^6$.
\end{lemma}

The proof of this result is more straightforward than the previous one, the intuitive idea behind being that due to the form of the empirical Hessian \eqref{empirical_hessian_recall}, if we are willing to forego dependence in $B$, we can expect $\frac{1}{n} \sumn X _ i X _ i ^\top$ to concentrate relatively quickly around $I_d$. In particular, as the lower bound is agnostic in $B$, due to the exponential terms being lower bounded by a term of constant order, less subtle control is required: we do not have to slice intervals in disjoint $I_k$'s as in the proof of Lemma \ref{mainlem:hessian_lowerbound}. While omitting the $B$-dependence in the empirical Hessian lower bound might appear to sacrifice sharpness, it is appropriate for the intermediate regime; indeed, in this regime, the exponential weights are of order $\exp(O(B))$ with high probability, rather than the $\exp(O(B^2))$ weights (observed when $\langle u, X \rangle \asymp B$) that drive the large-sample case. Clearly, weights of order $\exp(O(B))$ then remain negligible compared to the $\exp(O(B^2))$ terms. 

\quad The approach is formalized once again by introducing a (simpler) empirical process, associated with another VC class of complexity $O(d)$ and by applying Lemma 11 from \cite{chardon2026finitesampleperformancemaximumlikelihood}. One should also keep in mind that giving up on the dependence in $B$ for this lower bound allows for it to hold over the entire space $\bdR^d$, unlike in the approach of Lemma \ref{mainlem:hessian_lowerbound}, where $\beta$ was restricted in the set $\Theta$ such that $\hess L (\beta) \asymp \hess L (\beta _ *)$. 

\subsection{PAC-Bayesian inequalities}

\label{section:pacbayes}

PAC-Bayes type inequalities are a large class of inequalities which allow to control a ``smoothed" version of a process of interest in terms of a Laplace transform term and a divergence term. The use of this method in non-asymptotic statistics was pioneered by Catoni \cite{catoni2007pac, catoni2017dimension} and co-authors, who popularized results from McAllester \cite{mcallester2003pac}. It has since found several applications in the non asymptotic study of random matrices \cite{oliveira2016lower, mourtada2022exact, chardon2026finitesampleperformancemaximumlikelihood} and random uniform quantities such as norms due to the uniform control the approach offers (see Appendix \ref{appendix:generic_pacbayes}). Precisely, this section provides the two PAC-Bayesian type inequalities that we use to control both the norm of the empirical gradient and the Hessian in the remainder of this section. The first inequality, given by Theorem \ref{th:pac_bayes}, extends the result from Gupta et al. \cite{gupta2023high} by allowing an arbitrary localization semi-norm describing the ``closeness" to zero of the vector in the Laplace transform. The second inequality, given by Theorem \ref{th:pac_bayes_mat}, extends a result from Zhivotovskiy \cite{zhivotovskiy2024dimension} concerning bounds of the sum of independent matrices. We recall and prove the ``generic" PAC-Bayes inequality in Appendix \ref{appendix:generic_pacbayes} for convenience, as the proofs of the inequalities for sub-Gamma random vectors and matrices both rely on an application of this idea. For a longer and more in depth expository paper on the subject, we refer to the following monograph by Catoni \cite{catoni2007pac}. 

\subsubsection{PAC-Bayesian inequality for sub-Gamma random vectors}

For $E$ a linear space, we recall that a mapping $\norm{ \cdot } _ \star : E \to \bdR_+$ is a seminorm over $E$ if it satisfies the two following properties,
\begin{itemize}
    \item Homogeneity: For any $x \in E$ and any scalar $\lambda$, $\norm{ \lambda x }_\star = \abs { \lambda } \norm{ x } _ \star$. 
    \item Sub-additivity/Triangle inequality: For any $x,y\in E$, $\norm{ x + y } _ \star \le \norm{ x } _\star + \norm{ y } _ \star$. 
\end{itemize}
Unlike a norm, we do not require for a seminorm to satisfy $\norm{ x } _ \star = 0 \implies x = 0$. 

\begin{definition}{$(\star, \Sigma)$ sub-Gamma random vector.}\label{def:subgammarv}
    Let $\norm{ \cdot } _ \star $ be an arbitrary seminorm on $\bdR^d$ and $\Sigma \succ 0$ be an arbitrary positive definite matrix. We say that a random vector $X$ with $\bdE \br{X} = 0$ is $(\star, \Sigma)$ sub-Gamma if 
    \begin{equation}
        \forall u \in \bdR ^d, \norm{ u } _ \star \le 1, \quad \bdE \br{ \exp( \dotp{ u }{ X } ) } \le \exp \p{ \frac{\norm{ u } _ \Sigma ^2 }{2} }.
    \end{equation}
\end{definition}
To state the result, we introduce $\norm{ \cdot } _ d $ the dual map of $\norm{ \cdot } _ \star$, that is, for any $ u \in \bdR ^d $, 
\begin{equation}
    \norm{ u } _ d = \sup _ {t : \norm{t } _ \star \le 1} \dotp{ u }{ t },
\end{equation}
note that $\norm{ \cdot }_d$ is a norm if and only if $\norm{ \cdot } _ \star$ is a norm and that when $\norm{ \cdot } _ \star$ is a seminorm, $\norm{ \cdot } _ d $ is infinite on the set $\set{ x \in E, x \ne 0, \norm{ x } _ \star = 0}$. For $g \sim \bcN(0 , I _ d )$, we define the Gaussian width of the dual ball,
\begin{equation}
    w _d = \bdE \br{ \sup _ {t : \norm{t } _ d \le 1} \dotp{ g }{ t } } = \bdE \br{ \norm{ g } _ \star },
\end{equation}
and its diameter 
\begin{equation}
    \Delta _ d = \sup _{ t : \norm{ t } _ d  \le 1 } \norm{ t } = \sup _ { u \in S ^ { d - 1 } } \norm{ u } _ \star.
\end{equation}

\begin{theorem}{Pac-Bayesian inequality for sub-Gamma random vectors.}\label{th:pac_bayes} For any $(\star, \Sigma)$ vector $X$ defined in \ref{def:subgammarv}, for any $z > 0$, with probability at least $1 - \exp( - z)$, it holds
\begin{equation}
    \norm{ X } \le 14 \sqrt{ \tr( \Sigma ) } + 13 \sqrt{ \norm{ \Sigma } (\log 2 + z ) } + 19 w _ d \sqrt{ \log 2 + z } + 22 \Delta _d ( \log 2 + z ) 
\end{equation}
    
\end{theorem}

The proof of Theorem \ref{th:pac_bayes} is given in Appendix \ref{proof:pac_bayes_vec}. 

\subsubsection{PAC-Bayesian inequality for sub-Gamma random matrices}

\begin{definition}{Sub-Gamma random matrix.}\label{def:subgammarm}
    Let $\mathbf{X}$ be a $n \times p$ random matrix with $\bdE \br{ \mathbf X} = 0$, we say that $\mathbf X $ is $(\sigma^2, b)$ sub-Gamma if 
    \begin{equation}
        \forall (u, v ) \in \bdR ^n \times \bdR ^p : \norm{ u } \vee \norm{ v } \le 1, \forall \abs{s} < \frac{1}{b}, \quad \bdE \br{ \exp( s \dotp{ \mathbf X }{ uv^\top }_F ) } \le \exp \p{ s^2 \sigma ^2 }.
    \end{equation}
\end{definition}

Let also $\Gamma_ U$ and $\Gamma _ V$ be two symmetric positive definite matrices and let $\mathcal E _ U, \mathcal E _ V $ be their associated ellipsoids, that is, 
\begin{equation}
    \mathcal E _ U = \set{ u \in \bdR ^n, \norm{ u } _ {\Gamma_U ^{-1}} \le 1 }, \quad \mathcal E _ V = \set{ v \in \bdR ^ p, \norm{ v } _ {\Gamma_V ^{-1}} \le 1 }. 
\end{equation}

\begin{theorem}{Pac-Bayesian inequality for sub-Gamma random matrices.}\label{th:pac_bayes_mat} Let $\mathbf X$ be a $(\sigma ^2, b)$ sub-Gamma random matrix. Then, for any $z > 0$, 
\begin{equation}
    \sup _ { u \in \bcE_ U, v \in \bcE _ V } u ^\top \mathbf X v = \sup _ { u \in \bcE_ U, v \in \bcE _ V } \dotp{ \mathbf X  }{ u v ^\top } _ F \le 4 \sqrt{ 2 \norm{ \Gamma _ U } \norm{ \Gamma _ V } } \p{ \sigma \bcC (U, V , z ) \vee b \bcC (U , V , z ) ^2 },
\end{equation}
with probability at least $ 1 - \exp ( - z)$, where $r (\Gamma _ U ) = \tr( \Gamma _ U) / \norm{ \Gamma _ U }$ , $r (\Gamma _ V) = \tr( \Gamma _ V) / \norm{ \Gamma _ V }$ and,
\begin{equation}
    \bcC ( U , V , z ) = \sqrt{r ( \Gamma _ U ) + r (\Gamma _ V ) + 2(2 \log 2 + z) }.
\end{equation}
    
\end{theorem}

The proof of Theorem \ref{th:pac_bayes_mat} is given in Appendix \ref{proof:pac_bayes_mat}. 

\paragraph{Application, control over a linear subspace.}

To simplify assume that $\bfX$ is of size $d \times d $ for some $d \ge 1$. This corollary can be useful for several reasons, the first one, which we use in our case is that since the random matrix $\bfX$ needs to be centered, its expectation $M = \bdE \br{ \bfX} \in \bdR^{d \times d}$ may not behave similarly over the entire space. Indeed, in our case $v^\top M v = e^{B ^2 / 2}$ over $\set{u _ *} ^\perp$ and $(B^2 + 1)e^{B ^2 / 2}$ over $\bdR u _ *$. The second reason is that we may only have a sub-Gamma control over a strict subspace $E \subsetneq \bdR^d$, i.e., we know how to upper bound the Laplace transform of $\dotp{ \bfX}{uv^\top}_F$ only for $(u,v) \in E \times E$. 

Let $E \subsetneq \bdR^d$ and assume we have,
\begin{equation*}
    \forall (u,v) \in E \times E, \forall \abs{s } < \frac{1}{b}, \quad \bdE \br{ \exp \p{ s \dotp{ \bfX}{ u v ^\top } _ F } } \le \exp(s^2\sigma ^2). 
\end{equation*}
If we denote by $P$ the orthogonal projection on $E$, since $\dotp{\bfX}{Puv^\top P ^\top } _ F = \dotp{ P ^\top \bfX P }{ u v ^\top } _ F$, this is equivalent to 
\begin{equation*}
    \forall (u,v) \in \bdR ^d \times \bdR ^d : \norm{ u } \vee \norm{ v } \le 1, \forall \abs{s } < \frac{1}{b}, \quad \bdE \br{ \exp \p{ s \dotp{ P^\top \bfX P}{ u v ^\top } _ F } } \le \exp(s^2\sigma ^2). 
\end{equation*}
Hence, $P ^\top \bfX P$ is $(\sigma ^2, b)$ sub-Gamma in the usual sense and Theorem \ref{th:pac_bayes_mat} can be applied to give (using the same notations), 
\begin{align*}
    \sup _ { u , v \in \bcE_U \times \bcE_V } u ^\top P ^\top \bfX P v &= \sup _{(u,v) \in E \cap \bcE_U \times E \cap \bcE_V} u^\top \bfX v \le  4 \sqrt{ 2 \norm{ \Gamma _ U } \norm{ \Gamma _ V } } \p{ \sigma \bcC (U, V , z ) \vee b \bcC (U , V , z ) ^2 }.
\end{align*}

\subsection{Upper bounds on the empirical Hessian}

This section provides the fundamental controls of the empirical Hessian from above, which, while not directly necessary for the localization lemma, are central in obtaining deviation bounds for $\|\grad \est L (\beta _ * ) \| _ { H ^{-1}}$. In both proofs, the particular structure of the matrix $H$ (having two eigenspaces $\bdR u _ *$ and $\set{u _ *}^\perp$), reduces the problem of showing
\begin{equation*}
    \hess \est L ( \beta _ *) \preceq H,
\end{equation*}
to proving
\begin{equation*}
    u_ * ^\top \hess \est L (\beta _ * ) u _ * \le (B^2 + 1)e^{B ^2 / 2 } \quad \text{and} \quad \sup_{ v \in S ^{d - 1 } \cap \set{u_*}^\perp  } v ^\top \hess \est L (\beta _ * ) v \le e^{B ^2 / 2}.
\end{equation*}
This fact structures both proofs, we start by controlling the empirical Hessian in the direction $u _ *$ and then control it over the orthogonal subspace of $u _ *$. 

\quad It is worth mentioning that, due to the Gaussian design, the quantity $u _ * ^\top \hess \est L (\beta _ *) u _ *$ is a sum of random variables where the Gaussian in the exponential weight is the same as the squared Gaussian $\dotp{u _ *}{ X }^2$ that multiplies it. On the other hand, in any other direction $v \perp u _ *$, the exponential weight becomes independent of $\dotp{v}{X} ^2$. Moreover, the dependence in $d$ of the problem comes solely from the second upper bound, as a uniform control over $\set{u_*}^\perp$, a subspace of dimension $d -1$ is needed.

\quad In both proofs, due to the random variables having very explosive moments, truncation is used to control the magnitude of the exponential weights. Due to the sub-Gaussian tails of $\dotp{u_*}{X_i}$, the event $( \max _ i \dotp{ u _ * }{ X _ i } \le \sqrt{ 2 \log n + 2z})$ occurs with probability at least $1 - e^{-z}$. Bernstein's inequality is then used to control the empirical Hessian from above in the direction $u_*$. To obtain the control over the subspace $\set{u_*}^\perp$, the PAC-Bayes inequality for sub-Gamma random matrices (Theorem \ref{th:pac_bayes_mat}), is used, truncation and independence between $\dotp{u _ *}{X}$ and $\dotp{v}{X}$ being used to control the moments of the matrix. 

\quad Another potential approach to control the empirical Hessian uniformly over $\set{u_*}^\perp$ without using the PAC-Bayes machinery is to exploit independence of the exponential weights and the $\dotp{v}{X_i}^2$ for any $v \perp u _ *$. Indeed, for any $v \in S ^{d-1} \cap \set{u_*} ^\perp$,
\begin{equation*}
    v ^\top \hess \est L (\beta _ * ) v = \frac{1}{n} \sumn e^{ \norm{ \beta _ * } \dotp{ u _ * }{ X _ i } } \dotp{ v }{ X _ i } ^2. 
\end{equation*}
One can then condition on the exponential weights and apply Bernstein's inequality conditionally, the $\dotp{v}{X _ i } ^2$ being sub-Gamma as they follow a chi-squared distribution with parameter $1$. To remove the conditioning however, one must control the conditional variance; using truncation and Bernstein's inequality again then yields a much worse dependence in $B$ for the threshold. 

\subsubsection{Large sample size regime}

\begin{lemma}\label{mainlem:hessian_upperbound}
    Let $\eps \in (0,1)$, there exists a numerical constant $C_9$ such that with probability at least $1 - 3 \exp(-z)$,
    \begin{equation}
        \hess \est L (\beta_* ) \preceq C_9 H,
    \end{equation}    
    whenever $n \ge  e^{(\frac{B ^2 }{ 2 }  (1 + \varepsilon ^{ - 1 } ) )} \p{ (d + z )e^{4B \sqrt{ z } } }^{1 + \varepsilon}$, with $C_9 \le 116$.
\end{lemma}

Note that the condition $\eps \in (0,1)$ is to simplify the sample size constraint, $\eps$ can take any strictly positive value if $n \gtrsim e ^{B ^2}(d+z)$ is also satisfied. 

\subsubsection{Intermediate sample size regime}

\begin{lemma}\label{mainlem:hessian_upperbound_low}
    Let $ \eps \in ( 0,1)$ and $\delta = \frac{2 \eps}{ 2 + \eps} \le 1$, there exists numerical constants $C_{10}$ and $K_\eps$ which depends only on $\eps$ such that with probability at least $1 - 3 \exp(-z)$,
    \begin{equation}
        \hess \est L (\beta_* ) \preceq C_{10} e^{B ^2 / \delta} I _d,
    \end{equation}    
    whenever $n \ge  K_\eps d ^{1 + \eps} \exp(8z)$, with $C_{10} \le 116$.
\end{lemma}

The exact expression for $K _ \eps $ is given in the proof of this lemma in Appendix \ref{app:lemmas}. 

\subsection{Upper bounds on the empirical gradient}

This section provides the second component of the localization lemma, namely, a high probability upper bound on the norm of the empirical gradient at $\beta _ *$:

\begin{equation*}
    \grad \est L (\beta _ * ) = \frac{1}{n} \sumn \p{ e ^{ \dotp{ \beta _ * }{ X _ i }} - Y _ i } X _ i.
\end{equation*}

\quad This bound is established in two different geometries, corresponding to the large sample size regime and the intermediate sample size regime. While slightly technical, the proofs of both lemmas follow the same straightforward approach: since $Y$ given $X$ is Poisson, the PAC-Bayes inequality for sub-Gamma random vectors (Theorem \ref{th:pac_bayes}) readily applies conditionally on $\bcX = (X_1, \dots, X_n)$. To obtain a non-conditional upper bound, one must control the various terms involved in the PAC-Bayes inequality with high probability. The primary difficulty in this step is controlling the conditional variance of the empirical gradient, which is governed by the empirical Hessian at $\beta_*$. This control is provided by Lemma \ref{mainlem:hessian_upperbound} and Lemma \ref{mainlem:hessian_upperbound_low}, which immediately yield high-probability upper bounds for the two key quantities,
\begin{equation*}
    \tr( M ^{- 1 / 2} \hess \est L (\beta _ *) M ^{-1/2}) \quad \text{and} \quad \norm{ M ^{- 1 / 2} \hess \est L (\beta _ *) M ^{-1/2}},
\end{equation*}
where $M \in \set{H, H _ 0}$ defines the geometry in which the empirical gradient is controlled, depending on the sample size regime. The remaining quantities are relatively straightforward to bound; the detailed calculations for those are deferred to Appendix \ref{appendix:calculations}. We now provide the two high probability upper bounds on the empirical gradient, in both the $H$ geometry, for the large sample size regime, and the $H_0$ geometry, for the intermediate sample size regime. 

\subsubsection{Large sample size regime}

\begin{lemma}\label{mainlem:gradient_bound} Let $\eps \in (0,1)$, there exists numerical constants $C_{11}, C_{12}$ such that, 
\begin{equation}
    \norm{ \grad \est L (\beta _ * ) } _ {H ^{ - 1 } } \le C_{11} \sqrt {\frac{d + z}{n} }
\end{equation}
with probability at least $1 - 5\exp(-z)$, whenever $n \ge C_{12}  e^{(\frac{B ^2 }{ 2 }  (1 + \varepsilon ^{ - 1 } ) )} \p{ (d + z )e^{4 B \sqrt{ z } } }^{1 + \varepsilon}$, with $C_{11} \le 550, C_{12} \le 300$.
\end{lemma}

We mention once again that the restriction of $\eps$ to $(0,1)$ is to simplify the condition on the sample size, ensuring that $n \gtrsim e^{B^2}(d+z)$ to satisfy the sample size condition for Lemma \ref{mainlem:hessian_upperbound} to hold. 

\subsubsection{Intermediate sample size regime}

For any $\eps \in (0,1)$, let $\delta = \frac{2 \eps}{ 2 + \eps}$ and define the matrix $H _ 0$ by, 
\begin{equation}
    H _ 0 = e^{B ^2 / \delta} I _d. 
\end{equation}

\begin{lemma}\label{mainlem:gradient_bound_lown} Let $\eps \in (0,1)$ and $\delta = \frac{2 \eps}{ 2 + \eps} \le 1$, there exists numerical constants $C_{13}$ and $K_\eps$ which depends only on $\eps$ such that, 
\begin{equation}
    \norm{ \grad \est L (\beta _ * ) } _ {H_0 ^{ - 1 } } \le C_{13} \sqrt {\frac{d + z}{n} }
\end{equation}
with probability at least $1 - 5\exp(-z)$, whenever $n \ge K_\eps d ^{ 1 + \eps } \exp(8z)$, with $ C _ {13} \le 550$.
\end{lemma}

\section{Discussion of the paper's assumptions}\label{discussion}

As we mentioned at the beginning of Section \ref{section:section2}, all the results of the present paper are obtained under two assumptions: an isotropic Gaussian design and a well-specified model. Discussing what is necessary and what can be relaxed is an interesting objective, as it both suggests generalizations of the present results and highlights the limitations of the approaches we use. The aim of this section is to locate precisely where each assumption enters the proofs, and to separate what is a technical convenience from what is structural. The conclusion, in short, is that the risk bounds of Theorem \ref{th:mainth_riskbound} and Theorem \ref{th:mainth_riskbound_smalln} are more robust than the existence result of Theorem \ref{th:mainth_existence}, and that the well-specification assumption is by far the least essential of the two.

\quad Let us first dispose of the isotropy assumption, which is not an assumption at all. For any invertible covariance $\Sigma \succ 0$, writing $\dotp{\beta}{X} = \dotp{\Sigma^{1/2} \beta}{\Sigma^{-1/2} X}$ shows that the model with design $X$ and parameter $\beta$ is the model with normalized design $\Sigma^{-1/2}X$ and reparametrized parameter $\Sigma^{1/2}\beta$. All the results of the paper therefore apply verbatim to any design of the form $X = \Sigma^{1/2} Z \sim \bcN(0, \Sigma)$ provided $\Sigma \succ 0$, in this case, the signal strength is read as $B = \max(2, \norm{ \Sigma ^{1 /2 } \beta _ * })$. 

\subsection{The design assumption}

\quad Before discussing the technicalities of what can be generalized and what cannot, it is helpful to ask what an ideal generalization would look like. The natural candidate is an isotropic sub-Gaussian design, that is, $\bdE \br{ X X ^\top } = I _ d$ together with
\begin{equation}\label{ass:subgaussian}
    \forall u \in S ^{d - 1 }, \quad \bdE \br{ \exp \p{ \dotp{u}{X} } } \le \exp \p{ \frac{\norm{u} ^2 }{ 2 } }.
\end{equation}
It is worth stressing that, for Poisson regression, \eqref{ass:subgaussian} is not only natural but essentially maximal, in a sense that has no counterpart in logistic or linear regression. Indeed, the population risk itself,
\begin{equation*}
    L (\beta ) = \bdE \br{ e ^{ \dotp{ \beta }{ X } } } - \bdE \br{ Y \dotp{ \beta }{ X } },
\end{equation*}
is finite for every $\beta \in \bdR^d$ only if the marginal moment generating function $\psi _ u ( t ) = \bdE \br{ e ^{ t \dotp{u}{X} } }$ is finite for every $t \in \bdR$ and every direction $u$. If the design has only sub-Gamma marginals, say, then $\psi _ u (t) = + \infty$ beyond some finite $t$, the population risk is thus identically $+\infty$ outside a ball, and the estimation problem is simply not well-posed for signal strengths above a threshold determined by the design. As one of the aims of the present paper is precisely to let $B$ be of arbitrary order, very light tails are forced upon us by the model itself; the Gaussian design is thus not merely a convenient choice, but is (almost) ``maximal" in the sense we just defined. 

\quad We now turn to the proofs. The architecture of the argument is entirely distribution-free: the convex localization Lemma \ref{lemma:localization} is deterministic, the empirical process arguments of Lemma \ref{lemma:emp_bound} and Lemma \ref{lemma:emp_processlb} rely only on Vapnik--Chervonenkis bounds and on Talagrand's inequality, and the two PAC-Bayesian inequalities of Section \ref{section:pacbayes} are stated for arbitrary sub-Gamma vectors and matrices. Gaussianity is used only in a small number of explicit computations, which we list here.

\quad \textbf{(D1) The closed form of the population Hessian.} Lemma \ref{lemma:form_of_hessian} gives $\hess L ( \beta ) = \psi ( \norm{\beta} ) ( I _ d + \norm{\beta} ^2 u u ^\top )$ with $\psi (t) = e ^{t ^2 / 2 }$, and this is what dictates the choice of the proxy matrix $H$ in \eqref{def:definition_matrix_h} and the deterministic comparison of Lemma \ref{lemma:deterministic_upperbound}. For a general isotropic design one would instead take $H \asymp \hess L (\beta _ *)$ directly (up to a convenient constant); the statements would then be relative to the true population Hessian which are not necessarily explicit in $B$, or are not as sharp, this is a potential loss of information but not an absolute obstruction. The substantive content of Lemma \ref{lemma:deterministic_upperbound}, namely that $\hess L ( \beta) \asymp \hess L ( \beta _ *)$ over a ball of $H$-radius of order $e^{B^2/4}$ (root of the minimal eigenvalue of $H$) is reminiscent of a generalized self-concordance statement in the sense of \cite{nesterov1994interior, bach2009selfconcordantanalysislogisticregression}. This Lemma would have to be re-derived for a different population Hessian. It is interesting to mention that in this step the tail behavior of the design is felt quite directly, since we compare $\psi$ at nearby arguments.

\quad \textbf{(D2) About more general classes of distributions.} The scale $e^{B^2/2}$ that appears throughout is nothing but $\psi (B) = \bdE [ e ^{ B \dotp{u}{X}}]$, and the rare event $\dotp{u _ *}{X} \asymp B$ identified in the discussion following Theorem \ref{th:mainth_riskbound} is the corresponding tilting point, we may reasonably hope to establish similar deviation results under some regularity assumptions on the design. The second ingredient, Lemma \ref{lemma:probabound}, is more delicate: it uses that $\dotp{u}{X}$ and $\dotp{w}{X}$ are \emph{independent} for $u \perp w$, which is typical of the Gaussian among isotropic distributions. For a general (isotropic) design these two variables are merely uncorrelated, and the conclusion of Lemma \ref{lemma:probabound} would have to be assumed in the form of an inequality condition on the joint distributions, namely that there exists $c > 0$ such that
\begin{equation}\label{ass:ineq_marginals}
    \forall u, v \in S ^{d-1}, \ \forall k, \quad \bdP \p{ \dotp{u}{X} \in I _ k , \ \abs{ \dotp{v}{X}} > \eta / 4 } \ge c \, \bdP \p{ \dotp{u}{X} \in I _ k },
\end{equation}
with $\eta = \max(1, B \abs{\dotp{u}{v}})$ as in \eqref{def:slicings}. Lemma \ref{lemma:small_lemma_lbproba}, which plays the same role in the intermediate regime, is a simple instance of such a condition. Finding classes of distributions for which such an inequality may hold is an interesting problem, we refer to \cite{chardon2026finitesampleperformancemaximumlikelihood} for a similar discussion in the setting of logistic regression, they establish a class of ``regular" distributions for which their results can be generalized.  

\quad \textbf{(D3) Independence across the splitting $\bdR u _ * \oplus \set{u_*}^\perp$.} The moment computations in the proofs of Lemma \ref{lem:orthospace} and Lemma \ref{lem:orthospace_smalln} use that $\dotp{u_*}{X}$ is independent of $(\dotp{u}{X}, \dotp{v}{X})$ for $u, v \perp u _ *$, which allows the exponential weight to be decoupled from the quadratic factor. This is a convenience rather than a necessity: applying the Cauchy--Schwarz inequality to the whole product instead yields the same bound with $\psi(2B) = e^{2B^2}$ replaced by $\psi(4B)^{1/2} = e^{4B^2}$. The conclusions therefore still hold, with a worse exponent in $B$ in the sample size conditions \eqref{cond:largen} and \eqref{cond:intermediate}. Note also that the truncation event $\Omega _ b$ used in these proofs requires only sub-Gaussian marginals, hence \eqref{ass:subgaussian}, and nothing more.

\quad \textbf{Existence.} The situation is markedly different for Theorem \ref{th:mainth_existence}. Gaussianity is used there twice, and in two different ways. First, the rotational invariance of $\bcN(0,I_d)$ allows us to assume that $Y_i$ only depends on the first coordinate $X_{i1}$, since for a rotation $R$ mapping $u_*$ to $e_1$ one has $\dotp{\beta_*}{X_i} = \norm{ \beta _ * } (RX_i)_1$ with $RX_i$ having the same law as $X_i$. The independence of the coordinates then gives $(Y_i, X_{i1}) \indep (X_{i2}, \dots, X_{id})$, so that the conditioning is legitimate. Second, the approximate kinematic formula of \cite{amelunxen2014livingedgephasetransitions} requires the random subspace $\Span(\bfX e _2, \dots, \bfX e _ d)$ to be uniformly distributed among among $d - 1$-dimensional subspaces of $\bdR^n$, conditionnally on $(\bfX e _1, Y)$. This is immediate for a standard Gaussian design and we do not attempt to determine the minimal conditions under which the property holds. Note however that due to the rotational invariance being central to the arguments, a proof for a general isotropic sub-Gaussian design would likely require a completely different machinery. 

\subsection{The well-specification assumption}

\quad We now argue that the well-specification assumption is much less rigid, and that the natural setting for the risk bounds is the one where only the conditional \emph{mean} is correctly specified. Indeed, the response $Y$ appears nowhere in Lemma \ref{mainlem:hessian_lowerbound}, Lemma \ref{mainlem:hessian_lowerbound_small}, Lemma \ref{mainlem:hessian_upperbound} and Lemma \ref{mainlem:hessian_upperbound_low}. The empirical Hessian \eqref{eq:hessian_poisson}, its proxies $H$ and $H_0$, and the deterministic comparison of Lemma \ref{lemma:deterministic_upperbound} are all functions of the design and of $\beta _ *$ alone. Consequently, the model assumption enters the proofs of Theorem \ref{th:mainth_riskbound} and Theorem \ref{th:mainth_riskbound_smalln} at exactly one point, the control of the empirical gradient
\begin{equation*}
    \grad \est L (\beta _ * ) = \frac{1}{n} \sumn \p{ e ^{ \dotp{ \beta _ * }{ X _ i }} - Y _ i } X _ i,
\end{equation*}
in Lemma \ref{mainlem:gradient_bound} and Lemma \ref{mainlem:gradient_bound_lown}, and only two properties of the conditional law of $Y$ given $X$ are used there. The first is that the conditional mean is well specified,
\begin{equation}\label{ass:mean}
    \bdE \br{ Y \mid X } = \exp \p{ \dotp{ \beta _ * }{ X } } \quad \text{almost surely},
\end{equation}
which is what makes $\grad L (\beta _ *) = 0$, hence  $\beta _ *$ the population minimizer, and the empirical gradient centered. The second is the conditional Laplace transform bound $\bdE \br{e ^{ s ( Y - \lambda ) } } \le e ^{ \lambda s ^2 }$, valid for $\abs s \le 1$ when $Y \sim \bcP(\lambda)$, which we may replace by the assumption that $Y$ is conditionally sub-Gamma with a variance proxy proportional to the conditional mean: there exist $c_1, c_2 > 0$ such that
\begin{equation}\label{ass:subgamma}
    \forall \abs{s} \le \frac{1}{c _ 1}, \quad \bdE \br{ \exp \p{ s \p{ Y - \bdE \br{ Y \mid X } } } \, \middle| \, X } \le \exp \p{ c _ 2 s ^2 e ^{ \dotp{ \beta _ * }{ X } } }.
\end{equation}
Under \eqref{ass:mean} and \eqref{ass:subgamma}, the proofs of Lemma \ref{mainlem:gradient_bound} and Lemma \ref{mainlem:gradient_bound_lown} go through unchanged with the conditional covariance proxy $\Sigma = \frac{2 c_2}{n} M ^{-1/2} \hess \est L (\beta _ *) M^{-1/2}$, and one obtains the same bounds inflated by $\sqrt{c_2}$. Propagating this through Lemma \ref{lemma:localization} shows that Theorem \ref{th:mainth_riskbound} and Theorem \ref{th:mainth_riskbound_smalln} hold with the excess risk bounds and the sample size conditions multiplied by a factor depending on $c_1, c_2$ only, leaving rate unchanged.

\quad To discuss Assumption \eqref{ass:subgamma}, it is instructive to see what happens for the misspecified model $Y \mid \eps, X \sim \bcP ( \exp ( \dotp{ \beta _ *}{X} + \eps))$ discussed in Section \ref{section:section1}, with $\eps \indep X$ normalized so that $\bdE [e ^\eps ] = 1$. Assumption \eqref{ass:mean} then holds, but the conditional variance $\Var(Y \mid X)$ is $e^{\dotp{\beta_*}{X}} + \operatorname{Var}(e^\eps) e^{2 \dotp{\beta_*}{X}}$, so the correct variance proxy has a quadratic term. Moreover, the conditional covariance of the gradient $\bdE \br{ \grad \est L (\beta _ *) \grad \est L (\beta _ *)^\top \mid X_1, \dots X_n}$ is a combination of $\hess \est L (\beta _ *)$ and of
\begin{equation*}
    \frac{1}{n} \sumn e^{2 \dotp{ \beta _ * }{ X _ i }} X _ i X _ i ^\top = \hess \est L ( 2 \beta _ *).
\end{equation*}
Since the Hessian upper bounds of Lemma \ref{mainlem:hessian_upperbound} and Lemma \ref{mainlem:hessian_upperbound_low} are statements about the design alone, they may be applied at $2 \beta _ *$ instead of $\beta _ *$, at the cost of replacing $B$ by $2B$ throughout. One would then recover the same conclusions with the sample size conditions \eqref{cond:largen} and \eqref{cond:intermediate} holding for $2B$, that is, with $e^{B^2/2}$ replaced by $e^{2B^2}$. This is a substantial degradation, but it is a quantitative one, and it makes the price of multiplicative overdispersion explicit.

\quad Finally, the existence result is also largely insensitive to well-specification, for a different reason: the characterization of Lemma \ref{thm:koriyama} is a statement about the loss and the observed data, not about the model. Inspecting the proof of Theorem \ref{th:mainth_existence}, only two facts about the conditional law of $Y$ are used, namely that $Y$ is conditionally independent of $X$ given $\dotp{\beta _ *}{X}$, which preserves the reduction to the first coordinate and holds in particular for the model above, and that $p = \bdP ( Y = 0 )$ is bounded away from $1$. In the well-specified case we obtained $p \le (e+1)/(2e)$, but any bound of the form $p \le 1 - c$ would do, with $C_1$ depending on $c$.

\quad To summarize, the risk bounds rest on weaker foundations than their statements suggest: on the model side, correct specification of the conditional mean together with conditional sub-Gamma behavior of the response, and on the design side, isotropy, sub-Gaussian marginals, an inequality condition on the marginals of the type \eqref{ass:ineq_marginals}, and a self-concordance-type comparison of the population Hessian over the localization set. Explicitness in $B$ would be lost (or at least degraded), and the exponent of $B$ in the sample size conditions would also increase, but the phenomenology, and in particular the intermediate regime, would persist. The existence result, by contrast, is tied to rotational invariance through the conic geometry it relies on, and extending it beyond the Gaussian design appears to require a genuinely different argument. 

\bibliographystyle{plain}
\bibliography{writing_bib(2)}

\newpage

\appendix

\section{Proof of the main lemmas}
\label{app:lemmas}

\subsection{Proof of the lower bounds on the empirical Hessian}

\subsubsection{Large sample size regime}

\begin{proof}[Proof of Lemma \ref{mainlem:hessian_lowerbound}]
This is at the heart of the problem in this regime, we decompose the proof of this result in several Lemmas to improve readability. We want to show that for any $v \in S ^{d -1}$, for any $ \beta \in B_H( \beta _ *, \frac{1}{2} \exp \p{ B^2/4 })$, with high probability, 
\begin{equation}\label{eq:lowerbound}
    \frac{1}{n} \sumn \exp \p{ \dotp{ \beta }{ X_i} } \dotp{v}{X_i} ^ 2 \gtrsim \exp \p{ \frac{B^2}{2} } \p{ 1 + B ^2 \dotp{u_*}{v} ^ 2 }.
\end{equation}

Recall that $u = \beta / \norm{\beta}$ for any $\beta \in B _ H ( \beta _ *, \frac{1}{2} \exp(B ^2 / 4 ))$. For $v \in S ^{ d - 1 }$, we introduce the following notations, 
\begin{equation}\label{def:slicings}
    \eta = \max (1, B \abs{ \dotp{u}{v} }), \qquad I_k = \br{ B - \frac{k+1}{B}, B - \frac{k }{ B }},
\end{equation}
and $\gamma \in [1/2, 1]$ such that $k _ 0 = \gamma B \in \bdN^*$. Finally, we define,
\begin{equation}
    \eta _ * = \max(5, B\abs{ \dotp{ u _ * }{ v } }).
\end{equation}
We first need the following lemma, proven in Appendix \ref{proof:emp_process_reduction}, which reduces the problem of showing the desired uniform lower bound on the empirical Hessian to controlling the deviations of an empirical process of indicator functions around their mean.
\begin{lemma}{Reduction to an empirical process.}\label{lemma:empirical_process_reduction}
    There exists a numerical constant $C_{14}$ such that, for any $ \beta \in B _ H ( \beta _ *, \frac{1}{2} \exp(B ^2 / 4 ))$ and any $v \in S^{d-1}$, 
    \begin{equation}
        v^\top \hess \est L (\beta ) v \ge C_{14} \frac{\eta_*^2}{n} \sum _ { k = 0 } ^{k _ 0 - 1 } \sumn \exp ( B ^ 2 - k ) \bdOne _ { \set{ \dotp{u}{X_i } \in I _ k , \ \abs {\dotp{v}{X _ i } } > \eta/ 4  }},
    \end{equation}
    with $C_{14} \ge 1/1600e^5$.
\end{lemma}
Define $p_{k,u,v} = \bdP \p{ \dotp{u}{X_i } \in I_k ,\  \abs{ \dotp{v}{X _ i }} > \eta / 4 }$, the empirical process given by Lemma \ref{lemma:empirical_process_reduction} can be written as the sum of the two following terms,
\begin{equation}\label{eq:first_term}
    \frac{\eta _ * ^ 2}{1600e^5} \sum _ { k = 0 } ^ { k _ 0 - 1 } \exp \p{ B ^2 - k } p _ { k , u , v },
\end{equation}
and 
\begin{equation}\label{eq:second_term}
    \frac{\eta _ * ^ 2 }{1600 e ^5 n } \sum _ { k = 0 } ^ { k _ 0 - 1 } \sumn \exp \p{ B ^2 - k } \p{ \bdOne _ { \set{ \dotp{u}{X_i } \in I _ k , \ \abs {\dotp{v}{X _ i } } > \eta/ 4  }} - p _ {k , u ,v } }.
\end{equation}
We will control each of these terms individually.

\textbf{First step: Control of \eqref{eq:first_term}}

The following lemma, proven in Appendix \ref{proof:proof_det_lb_hess} shows that the deterministic part of the empirical process is already lower bounded by a quantity of the desired order,
\begin{lemma}\label{lemma:bound1}
    Let $k_0, p_{k,u,v}$ and $\eta_*$ be as previously defined. There exists a numerical constant $C_{15}$ such that 
    \begin{equation*}
        \sum _ { k = 0 } ^{ k _ 0 - 1 } \exp( B ^ 2 - k ) \eta _ * ^2 p _ {k , u , v } \ge C_{15} v ^ \top H v, 
    \end{equation*}
    with $C_{15} > \frac{3}{20000}$.
\end{lemma}

\textbf{Second step : Control of \eqref{eq:second_term}}

To get the desired lower bound of order $v^\top Hv$ for $v^\top \hess \est L (\beta) v$ uniformly over the $H$-ball now, it suffices to show that for a large enough sample size $n$, the random fluctuations of the process around the mean are negligible, hence, the goal of the section is to prove the following result. 

\begin{lemma}\label{lemma:emp_bound}

    For any $ z > 0 $ and any numerical constant $\theta > 0$, there exists a numerical constant $C_{16}$ such that, by defining $\rho = \max(9, 9/\theta^2)$, if $n \ge  C_{16} \rho \exp(B^2 / 2)B^2(Bd + z)$ then with probability at least $1 - \exp(-z)$, simultaneously for all $k \in \discrete{0}{k_0 - 1 }$
    \begin{equation}
        \sup_ {u, v \in S ^ { d - 1} } \abs*{ \frac{1}{n} \sumn ( \bdOne _ { \set{ \dotp{u}{X_i} \in I _ k, \ \abs{ \dotp{v}{ X _ i }} > \eta / 4 }    }  - p _ {k , u , v } ) } \le \theta \bdP (g \in I _ k ) 
    \end{equation}
    with $C_{16} \le 1.4 \times 10^8$.
\end{lemma}

\begin{proof}[Proof of Lemma \ref{lemma:emp_bound}] 
We first apply Bousquet's version of Talagrand's inequality for empirical processes \cite{bousquet2002bennettconcentration}. As all random variables $\bdOne _ { \set{ \dotp{u}{X_i} \in I _ k, \ \abs{ \dotp{v}{ X _ i }} > \eta / 4 }}$ have variance less than $\bdP(g \in I_k)$, with probability at least $1 - \exp(-z)$, 
\begin{equation*}
    \sup_ {u, v \in S ^ { d - 1} } \abs*{ \frac{1}{n} \sumn ( \bdOne _ { \set{ \dotp{u}{X_i} \in I _ k, \ \abs{ \dotp{v}{ X _ i }} > \eta / 4 }   } -  p _ {k , u , v } ) } \le 2E + \sqrt{ \frac{2z}{n} \bdP (g \in I_k) } + \frac{2z}{n}
\end{equation*}
where 
\begin{equation}\label{def:expectation_term}
    E = \bdE \br { \sup_ {u, v \in S ^ { d - 1} } \abs*{ \frac{1}{n} \sumn ( \bdOne _ { \set{ \dotp{u}{X_i} \in I _ k, \ \abs{ \dotp{v}{ X _ i }} > \eta / 4 }    } - p _ {k , u , v } ) } }.
\end{equation}
For $a, b \in \bdR, \nu \in \bdR_+$ fixed real numbers, we define the following class of indicator functions 
\begin{equation*}
    \bcF = \set{ x \in \bdR^d \mapsto \bdOne _ { \set{ \dotp{u}{x} \in [a,b], \ \abs{ \dotp{v}{ x }} > \nu } }} \cup \set{0},
\end{equation*}
$\bcF$ describes the intersection of 4 half spaces of $\bdR ^d$, with the zero function added to simplify the control of the absolute value. To bound the expectation term \eqref{def:expectation_term} we use \cite[Theorem 1.1]{vanderVaart2009vcdimensions}, which gives that $\bcF$, as an intersection of $4$ VC-classes, each with VC-dimension $d+1$, has VC-dimension upper bounded by $c_1d$ for some numerical constant $c_1 \le 53$. Now, by \cite[Theorem 13.7]{boucheron2013concentration}, as all functions in $\bcF$ have variance bounded above by $\bdP ( g \in I _k )$, whenever  
\begin{equation*}
    n \ge \rho \frac{24^2 c_1}{10} d \frac{\log(16 e ^4  / \bdP (g \in I _k ))}{ \bdP ( g \in I _ k ) }, 
\end{equation*} 
we have
\begin{equation}\label{def:ctilde}
    E \le c_2 \sqrt{ \frac{d \bdP ( g \in I _ k ) \log( 16e^4 / \bdP (g \in I _ k )) }{ n }}, \quad c_2 \le 3100. 
\end{equation}
Hence, for any 
\begin{equation*}
    n \ge c_3 \rho \exp(B^2/2) B ( B ^2 d + z ) \ge c_2^2 \rho \max _ k  \frac{d  \log(16e^4 / \bdP (g \in I _k )) + 2z }{ \bdP ( g \in I _ k ) }, 
\end{equation*}
with $c_3 \le 6.8 \times 10^7$, with probability at least $1 - \exp(-z)$ for any $k \in \discrete{0}{k _ 0 - 1}$,
\begin{align*}\label{eq:lower_bound_hess}
    \sup_ {u, v \in S ^ { d - 1} } \abs*{ \frac{1}{n} \sumn ( \bdOne _ { \set{ \dotp{u}{X_i} \in I _ k, \ \abs{ \dotp{v}{ X _ i }} > \eta / 4 }   } -  p _ {k , u , v } ) } &\le \frac{3 }{\sqrt \rho } \bdP ( g \in I _ k ) \\
    &\le \min(1, \theta) \bdP(g \in I _ k).
\end{align*}
To make this hold for all $k$'s, we take a union bound over $k_0 = \gamma B \le B$ terms which adds a term upper bounded by $\log (B)$, which we can upper bound crudely again by $B$ and can be absorbed by multiplying the constant by $2$ and factoring another $B$, yielding the final condition, 
\begin{equation*}
    n \ge c_4 \rho \exp(B ^2 / 2 ) B ^2 ( B d + z ), \quad c_4 \le 1.4 \times 10^8.
\end{equation*}

\end{proof}

\textbf{Final step: proof of \eqref{eq:lowerbound}}

In this final step, we combine the lemmas we needed to show that the second term is actually negligible compared to the first term.
Recall that from Lemma \ref{lemma:empirical_process_reduction} we have 
\begin{align*}
    \frac{1}{n} \sumn \exp( \dotp{\beta}{X _ i } ) \dotp{X_i}{v} ^2 &\ge \frac{1}{1600 e ^ 5 } \sum_{ k = 0} ^ { k _ 0 - 1 } \Big(
        \exp( B  ^ 2 - k ) \eta_ * ^2 p _ { k , u , v } \\
    &\quad + \exp ( B ^2 - k ) \frac{ \eta _ * ^2 }{ n }
        \sumn \bdOne _ { \set{ \dotp{u}{X_i} \in I _ k, \ \abs{ \dotp{v}{ X _ i }} > \eta / 4 }   }
        - p _ {k , u , v }
    \Big). 
\end{align*}
By Lemma \ref{lemma:bound1}, there exists a numerical constant $C_{15} > 0$ such that
\begin{equation*}
    \sum_{ k = 0} ^ { k _ 0 - 1 } \exp( B  ^ 2 - k ) \eta_ * ^2 p _ { k , u , v } \ge C_{15} v ^ \top H v,
\end{equation*}
for the second term, Lemma \ref{lemma:emp_bound} applies with $ \theta = C_{15} / 50$. Whenever $n \ge c_4 \rho \exp(B^2 / 2)B^2(Bd + z)$, with probability at least than $ 1 - \exp(- z)$ we have simultaneously for all $k \in \discrete{0}{k_0 - 1}$,
\begin{equation*}
    \exp(B ^2 - k) \frac{ \eta _ * ^2 }{n }  \sumn (\bdOne _ { \set{ \dotp{u}{X_i} \in I _ k, \ \abs{ \dotp{v}{ X _ i }} > \eta / 4 } } - p _ { k , u, v }) \ge - \frac{C_{15}}{50} \exp(B ^2 - k ) \eta _ * ^2 \bdP ( g \in I _k ).
\end{equation*}
Thus, summing over $k$ from $0$ to $k _ 0 - 1 $ and applying Lemma \ref{lemma:sumbound} with $A = B$ and $k_0$, we deduce that whenever $n \ge c_4 \rho \exp(B ^2 / 2 ) B^2 (Bd + z)$
\begin{equation}
    \sum_{ k = 0 } ^ {k _ 0 - 1 } \exp(B ^2 - k) \frac{ \eta _ * ^2 }{n }  \sumn (\bdOne _ { \set{ \dotp{u}{X_i} \in I _ k, \ \abs{ \dotp{v}{ X _ i }} > \eta / 4 } } - p _ { k , u, v }) \ge - \frac{C_{15}}{ 50 } \exp(B ^2 / 2) \eta _ * ^2.
\end{equation}
Using $\exp(B ^2 / 2) \eta _ * ^2 \le 25 v ^\top H v $, we obtain the result. 

\end{proof}

\subsubsection{Intermediate regime}

\begin{proof}[Proof of Lemma \ref{mainlem:hessian_lowerbound_small}]
To prove the result, we first need the following technical lemma proven by Chardon, Lerasle and Mourtada which we recall here for convenience, 
\begin{lemma}{\cite[Lemma 11]{chardon2026finitesampleperformancemaximumlikelihood}}\label{lemma:emp_processlb}
    Let $X_1, \dots , X _ n $ be i.i.d. random variables taking value in a space $\bcX$ with common distribution $P$, and let $\bcA$ be a collection of subsets of $\bcX$ with VC-dimension at most $d \ge 1$. Let $p \in ( 0 ,1 ) $, and assume that $P(A) \ge p$ for any $A \in \bcA$. If, 
    \begin{equation*}
        n \ge \frac{180 \log ( 480 / p) d}{ p },
    \end{equation*}
    then with probability at least $ 1 - 2 e ^{ - np / 64 }$, one has 
    \begin{equation*}
        \inf _ {A \in \bcA } \frac{1}{n } \sumn \ind{X _ i \in A } \ge \frac{p }{ 2 }. 
    \end{equation*}
\end{lemma}

    Now, let $v \in S ^{d - 1 }$, and $\beta \in \bdR^d$,
    \begin{equation*}
        v ^\top \hess \est L ( \beta ) v = \frac{1}{n} \sumn e ^{\dotp{ \beta }{ X _ i } } \dotp{ v }{ X _ i } ^2 \ge \frac{1}{ n } \sumn \ind{ \dotp{ u }{ X _ i }  \ge 1, \abs{ \dotp{ v} {X _ i } } \ge 1 }.
    \end{equation*}
    Writing the set,
    \begin{equation*}
        \set{ x \in \bdR ^d , \dotp{ u }{ x }  \ge 1, \abs{ \dotp{ v }{ x } } \ge 1 } 
    \end{equation*}
    as unions of intersections of half spaces of $\bdR^d$, each with VC-dimension $d + 1$. We apply the results from \cite{vanderVaart2009vcdimensions} again, we deduce that the VC dimension of the class is upper bounded by $52d$. Moreover, by Lemma \ref{lemma:emp_processlb} that we may apply with $p = \bdP(\bcN(0,1) \ge 1 ) ^2$ thanks to Lemma \ref{lemma:small_lemma_lbproba}, we obtain that whenever $n \ge p^{-1} 9360 \log(480/p ) d$,   
    \begin{equation*}
        \frac{1}{n} \sumn \ind{ \dotp{ u }{ X _ i } \ge 1, \abs{ \dotp{ v} {X _ i } } \ge 1 } \ge \frac{p}{2} = c_5 v ^\top I_d  v,
    \end{equation*}
    with probability at least $ 1 - 2\exp(-np/64)$ and with $c_5 = p /  2 \ge \frac{1}{100}$.
\end{proof}

\subsection{Proof of the upper bounds on the empirical Hessian}

\subsubsection{Large sample size regime}

\begin{proof}[Proof of Lemma \ref{mainlem:hessian_upperbound}]

The proof reduces to showing that the inequality holds both in the direction $u _ *$ and over the subspace $\set{u_*}^\perp$, we decompose the proof into two corresponding lemmas, starting with the orthogonal space $\set{ u _ * } ^\perp$. 

\begin{lemma}{Control over the orthogonal space}\label{lem:orthospace}. Let $\eps \in(0,1)$, there exists a numerical constant $C_{17}$ such that, with probability at least $1 - 2\exp(-z)$, for any $v \in S ^{d - 1 } \cap \set{u_*}^\perp$,
    \begin{equation}
        v ^\top \hess \est L (\beta_  * ) v  \le C_{17} v^\top H v,
    \end{equation}
    whenever $n \ge \exp \p{ \frac{B ^2 }{ 2 } (1 + \varepsilon ^{- 1 }) } \p{ (d + z )e^{B \sqrt{ 2z } } }^{1 + \varepsilon} \vee e^{B ^2 } ( d + z)$, with $C_{17} \le 58$.  
\end{lemma}

Before starting the proof, notice that due to the particular stucture of $H$, $\hess \est L (\beta _ *) \preceq H$ over $\set {u_*}^\perp$ reduces to showing that $v ^ \top \hess \est L ( \beta _ * ) v \le e^{B^2/2}$ for any unit vector $v$ orthogonal to $u_*$. 

\begin{proof}
    The proof relies on Theorem \ref{th:pac_bayes_mat}, the PAC-Bayesian inequality for sub-Gamma random matrices. Let $u, v \in S ^{d - 1 } \cap \set{ u _ * } ^\perp$ and define the centered empirical Hessian $\hat H _ n = \hess \est L (\beta_ * ) - \bdE \br{ \hess \est L (\beta_ * ) }$. We consider the following event 
    \begin{equation*}
        \Omega _ b = \p{ \max _ i \dotp{ u _ * }{ X _ i } \le \sqrt{ 2 \log n + 2 z} },
    \end{equation*}
    which happens with probability at least $1 - \exp(-z)$, finally set $M = \exp( B \sqrt{ 2 \log n + 2z })$. 
    Let $ k \ge 2$ be an integer and $g \sim \bcN(0,1)$, we have, by independence and applying Cauchy-Schwarz inequality,
    \begin{equation*}
        \bdE \br{ e ^{k B \dotp{u_*}{X} } \abs{ \dotp{ u}{X} \dotp{ v }{ X } } ^k } \le M ^{ k - 2 } \bdE \br{ e ^{2 B g} } \bdE \br{ g ^{ 2 k } } \le M ^{ k - 2 } e ^{2 B ^2 } \frac{(2k)!}{2^k k!} \le (2M ) ^{ k -  2 } 4 e ^{2 B ^2 } k !.
    \end{equation*}
    As, $u ^\top \hess \est L (\beta_  * ) v = \dotp{ \hess \est L (\beta_  * ) }{ u v ^\top }_F$, by control of the moments, Bernstein's inequality implies the following control of the centered MGF,
    \begin{equation}
        \bdE \br{ \exp \p{s \dotp{ \hat H _ n \ind { \Omega _ b } }{ u v ^\top } _ F }} \le \exp \p{ \frac{ s ^2 }{n} 2 e ^{2 B ^2 }}, \quad \text{for any $\abs{s} < \frac{n}{M}$}. 
    \end{equation}
    By upper bounding the expectation of the truncated variables by the true expectation, the PAC-Bayes inequality for sub-Gamma random matrices then implies on $\Omega _ b$,
    \begin{equation}
        \sup _ { u , v \in S ^{d - 1 } \cap \set{ u _ * } ^\perp } u ^ \top \hess \est L (\beta_ * ) v \le e^{ B ^2 / 2 } +  32 \sqrt{2 } e ^{ B ^2 } \sqrt{\frac{  d + z   }{n}}  + 8 \sqrt 2 M \frac{ d + z }{n},
    \end{equation}
    with probability at least $ 1 - \exp(-z)$. The second term is dominated by $e^{B ^2 / 2 }$ whenever $ n \ge (d + z) e^{B ^2 }$. For the last term, to know when it is negligible compared to $e^{B ^2 / 2 }$, notice that 
    \begin{align*}
        \exp( - B ^2 / 2 + B\sqrt{ 2 \log n + 2z } - \log n ) \le \exp \p{ - \p{ \frac{B }{ \sqrt{ 2 }} - \sqrt{ \log n } }^2   + B \sqrt{ 2 z } }.
    \end{align*}
    Since for any $\varepsilon > 0$, we have $(a - b) ^2 \ge \frac{a ^2 }{1 + \varepsilon} - \frac{b^2}{\varepsilon}$, we deduce that this is true whenever there exists $\varepsilon > 0$ such that 
    \begin{equation}
        n \ge \exp \p{ \frac{ B ^2 }{ 2 } (1 + \varepsilon ^{- 1 }) } \p{ (d + z ) e ^{ B \sqrt{ 2 z } } }^{1 + \varepsilon}.
    \end{equation}
    Implying we need,
    \begin{equation*}
        n \ge \exp \p{ \frac{ B ^2 }{ 2 } (1 + \varepsilon ^{- 1 }) } \p{ (d + z ) e ^{ B \sqrt{ 2 z } } }^{1 + \varepsilon} \vee e ^{B ^2 } (d+z),
    \end{equation*}
    over $\set{u_*}^\perp$. 
\end{proof}

\begin{lemma}{Control in the direction $u_*$}\label{lem:directioncontrol}. There exists a numerical constant $C_{18}$ such that, with probability at least $1 - 2 \exp(-z)$, 
    \begin{equation}
        u_* ^\top \hess \est L (\beta_  * ) u_* \le C_{18}u_*^\top H u_*,
    \end{equation}
    whenever $n \ge e ^{B ^2 + 4B \sqrt{ z } }$ with $C_{18} \le 12$.  
\end{lemma}

Before proving the result, notice that again, due to the structure of $H$, the results reduces to having 
\begin{equation}
    u _ * ^\top \hess \est L ( \beta _ * ) u _ * \le e^{B ^2 / 2 }( B ^2 + 1 ),
\end{equation}
with sufficiently high probability. 

\begin{proof}
    We proceed by truncation once again, we recall that $\Omega _ b$ is defined as 
    \begin{equation}
        \Omega _ b = \p{ \max _ i \dotp{ u _ * }{ X _ i } \le \sqrt{ 2 \log n + 2 z} },
    \end{equation}
    and notice that,
    \begin{equation*}
    u _ *^\top  \hess \est L (\beta _ * ) u _ * = \frac{1}{n} \sumn e ^{ B \dotp{ u _ * }{ X _ i } } \dotp{ u _ * }{ X _ i } ^2. 
    \end{equation*}
    On this event, Bernstein's inequality for sub-Gamma random variables yields  
    \begin{equation}
        \frac{1}{n} \sumn e ^{ B g _ i } g _ i ^2 \le B ^2 e^{ B ^2  / 2 } + \sqrt{ \frac{86 B ^4 e^{2B ^2 } z }{n} } + \frac{e ^{B \sqrt{ 2 \log n + 2 z } }z }{n}.
    \end{equation}
    with probability at least $1 - \exp(-z)$, where we upper bounded the truncated expectation and variance by their non-truncated version. As everything is dominated by $B^2 e^{B ^2 / 2 }$ whenever $n \ge ze^{B ^2 + 4B \sqrt{ z }  } $, we obtain that in this case,
    \begin{equation}
        u _* ^\top \hess \est L (\beta _ * ) u _ * \le 12 B ^2 e ^{B ^2 / 2 } \le 12 u _ * ^\top H u _ *,
    \end{equation}
    with probability at least $1 - \exp(-z)$ if $\Omega_b $ already holds or with probability at least $ 1 - 2 \exp(-z)$ for the standalone result. 
\end{proof}

    To conclude, notice that since we want to control the quantity 
    \begin{equation}
        w ^\top \hess \est L (\beta _ * ) w 
    \end{equation}
    uniformly over $S ^{d - 1 }$, because we can always write $w = \alpha u _ * + \beta v$ for $u _ * \perp v$ unit vectors, we have 
    \begin{equation}
        w ^\top \hess \est L (\beta _ * ) w  = \alpha ^ 2 u _ * ^\top \hess \est L (\beta _ * ) u _ * + \beta ^2 v ^\top \hess \est L (\beta _ * ) v + 2 u _ * ^\top \hess \est L (\beta _ * ) v.
    \end{equation}
    However due to the structure of the Hessian, 
    \begin{equation}
        2 u _ * ^\top \hess \est L (\beta _ * ) v = \frac{1}{n} \sumn e^{ B \dotp{ u _ * }{ X _ i} } \dotp{ u _ * }{ X _ i } \dotp{ v }{ X _ i } \le u _ * ^\top \hess \est L (\beta _ * ) u_* +  v ^\top \hess \est L (\beta _ * ) v,
    \end{equation}
    as $2ab \le a^2 + b ^2$ for any real numbers $a, b$. Hence, up to a constant $2$, it suffices to control over the two aforementioned spaces. We simply have to combine Lemmas \ref{lem:orthospace} and \ref{lem:directioncontrol}, this yields that there exists a constant $c_6 \le 116$ such that 
    \begin{equation}
        \hess \est L (\beta _ * ) \preceq c_6 H,
    \end{equation}
    with probability at least $1 - 3\exp(-z)$. 
\end{proof}

\subsubsection{Intermediate regime}

The proof is very similar in spirit to what was already done, we want to control the empirical Hessian above with high probability, however, since $n$ is much smaller, the upper bound cannot hold with the same level of precision with high probability. The approach is then to upper bound with high probability the Hessian by the matrix, 
\begin{equation*}
    e ^ {B ^2 / \delta} I _d, \quad \delta > 0,
\end{equation*}
for some value of $\delta$ that will be picked in a specific way. The proof still reduces to showing that the inequalities of the quadratic maps holds both in the direction $u _ *$ and over the subspace $\set{u_*}^\perp$. We begin with the direction $u _ * $.

\begin{lemma}{Control in the direction $u_*$}\label{lem:directioncontrol_smalln}.  Let $\eps \in (0,1)$ and $\delta = \frac{2 \eps}{ 2 + \eps} \le 1$, there exists numerical constant $C_{19}, C_\eps$ where $C_\eps$ depends only on $\eps$ such that, with probability at least $1 - 2 \exp(-z)$, 
    \begin{equation}
        u_* ^\top \hess \est L (\beta_  * ) u_* \le C_{19} e^{ B ^2 / \delta },
    \end{equation}
    whenever $n \ge 4 C _ \eps (\exp( \eps ( 1 + 2 \eps)z) z ^{1 + 2 \eps} ( 1 + (2z)^{ 1 + \eps}) +1)$, with $C_{19} \le 4$.  
\end{lemma}

\begin{proof}
    We proceed by truncation, recall that $\Omega _ b$ is defined as 
    \begin{equation}
        \Omega _ b = \p{ \max _ i \dotp{ u _ * }{ X _ i } \le \sqrt{ 2 \log n + 2 z} }.
    \end{equation}
    By similar arguments, on this event, Bernstein's inequality for sub-Gamma random variables yields  
    \begin{equation}\label{bernstein}
        \frac{1}{n} \sumn e ^{ B \dotp{ u _ * }{ X  _ i } } \dotp{ u _ * }{ X  _ i } ^2 \ind{\Omega_ b } \le e^{B ^2 / \delta} + e^{B ^2 / \delta } \sqrt{ \frac{2 z }{n} } + \frac{( 2 \log n + 2z) z e ^{B \sqrt{ 2 \log n + 2 z } } }{n}.
    \end{equation}
    with probability at least $1 - \exp(-z)$, where we upper bounded the truncated expectation and variance by their non-truncated version, and since $\varepsilon < 1, \delta \le 2/3$ so $\exp(B ^ 2 / \delta) \ge B ^2 \exp(B ^2)$. We need to determine when the last term in \eqref{bernstein} is negligible compared to $\exp(B ^2 / \delta)$.

    We denote $ L = \log n$, we want 
    \begin{equation*}
        (2L + 2 z ) z \exp \p{ B \sqrt{ 2 z } + B \sqrt{ 2 L } - L } \le \exp( B ^2 / \delta),
    \end{equation*}
    which is equivalent to 
    \begin{equation*}
        (2L + 2 z ) z \exp \p{ - \p{ \frac{ B }{ \sqrt 2 } - \sqrt L }^2 + B \sqrt{ 2 z } } \le \exp( B ^2 / \delta - B ^2 / 2 ).
    \end{equation*}
    Since 
    \begin{equation*}
        -(a - b) ^2 \le \frac{- a ^2 }{ 1 + \varepsilon } + \frac{b ^2}{ \varepsilon } \quad \text{and} \quad ab \le \frac{a^2 }{ 2 \varepsilon } + \frac{ \varepsilon b ^2 }{ 2 }, 
    \end{equation*}
    we have 
    \begin{equation*}
        - \p{ \frac{ B }{ \sqrt 2 } - \sqrt L }^2 \le \frac{- L }{ 1 + \varepsilon} + \frac{B ^2 }{ 2 \varepsilon } \quad \text{and} \quad B \sqrt{ 2 z } \le \frac{B ^2 }{ 2 \varepsilon } + \varepsilon z. 
    \end{equation*}
    So it is sufficient to have
    \begin{equation*}
        (2L + 2 z ) z \exp \p{ - \frac{L }{ 1 + \varepsilon} + \frac{B ^2 }{ \varepsilon } + \varepsilon z  } \le \exp \p{ \frac{B ^2 }{ \delta } - \frac{B ^2 }{ 2 } },
    \end{equation*}
    taking $\frac{1}{ \delta} = \frac{1}{2} + \frac{1}{\varepsilon}$ so $\delta = \frac{2 \varepsilon}{ 2 + \varepsilon}$, we need $n$ to satisfy 
    \begin{equation*}
         n^{1 / (1+\varepsilon)} \ge (2 L + 2 z)z \exp \p{ \varepsilon z } ,
    \end{equation*}
    i.e.
    \begin{equation*}
        n \ge (2 \log n + 2z ) ^{ 1 + \varepsilon } z ^{1 + \eps} \exp( \varepsilon ( 1 + \varepsilon ) z ). 
    \end{equation*}
    By convexity of $ x \mapsto x ^{1 + \varepsilon }$ we have $(a + b) ^{ 1 + \varepsilon} = 2 ^{ 1 + \eps} ( a/ 2 + b / 2 ) ^{ 1 + \eps} \le 2^{ \eps} (a ^{1 + \eps} + b ^{1 + \eps})$ so we have 
    \begin{equation*}
        (2 z + 2 \log n)^{ 1 + \eps} \le 2 ^\eps (2z) ^{ 1 + \eps } + 2 ^\eps (2 \log n)^{ 1 + \eps }
    \end{equation*}
    so it is sufficient to have 
    \begin{equation*}
        n \ge 2 ^\varepsilon (2z^2)^{ 1 + \varepsilon } \exp( \eps ( 1 + \eps ) z) + 2 ^\eps (2 \log n ) ^{ 1 + \eps} z^{ 1 + \eps} \exp(\eps ( 1 + \eps) z ).
    \end{equation*}
    Since for any fixed $\varepsilon$ we have $n / (\log n )^{ 1 + \eps} \ge n ^{ \frac{1 + \eps }{ 1 + 2 \eps } }$ for any $n \ge C _{ \eps }$, where 
    \begin{equation}\label{def:c_eps}
        C _ \eps = \p{ \frac{2 (1 + 2 \eps) ( 1 + \eps) }{ \eps } \log \p{ \frac{ (1 + 2 \eps) ( 1 + \eps) }{ \eps }} } ^{ \frac{(1 + 2 \eps) ( 1 + \eps) }{ \eps } },
    \end{equation}
    \begin{equation*}
        \frac{n }{ ( \log n ) ^{1 + \eps}} \ge n ^{ \frac{1 + \eps }{ 1 + 2 \eps } } \iff n ^\frac{ \eps }{ 1 + 2 \eps } \ge (\log n) ^{1 + \eps }.
    \end{equation*}
    We deduce that whenever, 
    \begin{equation*}
        n \ge 2 ^\eps \exp( \eps ( 1 + \eps) z ) (2z)^{ 1 + \eps } z ^{1 + \eps} + C _ \eps 2 ^{ \eps \frac{1 + 2 \eps }{ 1 + \eps }} \exp( \eps (1 + 2 \eps) z ) z ^{ 1 + 2 \eps} \quad \text{and} \quad n \ge C _ \eps,
    \end{equation*}
    i.e.
    \begin{equation}\label{cond:cond1}
        n \ge 4 C _ \eps (\exp( \eps ( 1 + 2 \eps)z) z ^{1 + 2 \eps} ( 1 + (2z)^{ 1 + \eps}) + 1), 
    \end{equation}
    we have
    \begin{equation*}
        u _ * ^\top \hess \est L (\beta _ * ) u _ * \le 4 e^{B ^2 / \delta },
    \end{equation*}
    with probability at least $1 - \exp(-z)$ if $\Omega_b $ already holds or with probability at least $ 1 - 2 \exp(-z)$ for the standalone result and with $\delta = \frac{2 \eps }{ 2 + \eps }$. 
\end{proof}

\begin{lemma}{Control over the orthogonal space}\label{lem:orthospace_smalln}. Let $\eps \in (0,1)$ and $\delta = \frac{2 \eps}{ 2 + \eps} \le 1$, there exists a numerical constant $C_{20}$ such that, with probability at least $1 - 2 \exp(-z)$, for any $v \in S ^{d - 1 } \cap \set{u_*}^\perp$,
    \begin{equation}
        v ^\top \hess \est L (\beta_  * ) v  \le C_{20} e^{B ^2 / \delta},
    \end{equation}
    whenever $n \ge (d+z)^{ 1  + \eps} z ^{ 1 + \eps} \exp( \eps ( 1 + \eps) z)$, with $C_{20} \le 58$.   
\end{lemma}

Before starting the proof, notice that $\hess \est L (\beta _ *) \preceq e^{B ^2 / \delta} I _d $ over $\set {u_*}^\perp$ reduces to showing that $v ^ \top \hess \est L ( \beta _ * ) v \le e^{B^2/\delta}$ for any unit vector $v$ orthogonal to $u_ *$. 

\begin{proof}
    We mimic the proof of \ref{mainlem:hessian_upperbound}, let $u, v \in S ^{d - 1 } \cap \set{ u _ * } ^\perp$ and define the centered empirical Hessian $\hat H _ n = \hess \est L (\beta_ * ) - \bdE \br{ \hess \est L (\beta_ * ) }$. We place ourselves the following event 
    \begin{equation*}
        \Omega _ b = \p{ \max _ i \dotp{ u _ * }{ X _ i } \le \sqrt{ 2 \log n + 2 z} }
    \end{equation*}
    which happens with probability at least $1 - \exp(-z)$, and we denote $M = \exp( B \sqrt{ 2 \log n + 2z })$. Let $ k \ge 2$ be an integer, by the same arguments as before, 
    \begin{equation}
        \bdE \br{ e ^{k B g } \abs{ \dotp{ u}{X} \dotp{ v }{ X } } ^k } \le (2M ) ^{ k -  2 } 4 e ^{2 B ^2 } k !.
    \end{equation}
    By upper bounding the expectation of the truncated variables by the true expectation, the PAC-Bayes inequality for sub-Gamma random matrices then implies on $\Omega _ b$,
    \begin{equation}
        \sup _ { u , v \in S ^{d - 1 } \cap \set{ u _ * } ^\perp } u ^ \top \hess \est L (\beta_ * ) v \le e^{ B ^2 / 2 } +  32 \sqrt{2 } e ^{ B ^2 } \sqrt{\frac{  d + z   }{n}}  + 8 \sqrt 2 M \frac{ d + z }{n}
    \end{equation}
    with probability at least $ 1 - \exp(-z)$. Since $\delta \le 1 $, the first term is smaller than $\exp(B ^2 / \delta)$ and so is the second term whenever $ n \ge (d + z)$. For the last term, to know when it is negligible compared to $e^{B ^2 / \delta }$, notice that, if we denote $ L = \log n$, we want 
    \begin{equation*}
        z \exp \p{ B \sqrt{ 2 z } + B \sqrt{ 2 L } - L + \log(d+z) } \le \exp( B ^2 / \delta),
    \end{equation*}
    which is equivalent to 
    \begin{equation*}
        z \exp \p{ - \p{ \frac{ B }{ \sqrt 2 } - \sqrt L }^2 + B \sqrt{ 2 z } + \log(d+z) } \le \exp( B ^2 / \delta - B ^2 / 2 ).
    \end{equation*}
    Since 
    \begin{equation*}
        -(a - b) ^2 \le \frac{- a ^2 }{ 1 + \varepsilon } + \frac{b ^2}{ \varepsilon } \quad \text{and} \quad ab \le \frac{a^2 }{ 2 \varepsilon } + \frac{ \varepsilon b ^2 }{ 2 }, 
    \end{equation*}
    we have 
    \begin{equation*}
        - \p{ \frac{ B }{ \sqrt 2 } - \sqrt L }^2 \le \frac{- L }{ 1 + \varepsilon} + \frac{B ^2 }{ 2 \varepsilon } \quad \text{and} \quad B \sqrt{ 2 z } \le \frac{B ^2 }{ 2 \varepsilon } + \varepsilon z. 
    \end{equation*}
    So it is sufficient to have
    \begin{equation*}
        z \exp \p{ - \frac{L }{ 1 + \varepsilon} + \frac{B ^2 }{ \varepsilon } + \varepsilon z + \log(d+z) } \le \exp \p{ \frac{B ^2 }{ \delta } - \frac{B ^2 }{ 2 } },
    \end{equation*}
    taking $\frac{1}{ \delta} = \frac{1}{2} + \frac{1}{\varepsilon}$ so $\delta = \frac{2 \varepsilon}{ 2 + \varepsilon}$, we need $n$ to satisfy 
    \begin{equation*}
         n^{1 / (1+\varepsilon)} \ge z \exp \p{ \varepsilon z + \log ( d + z ) } ,
    \end{equation*}
    i.e.
    \begin{equation}\label{cond:cond2}
        n \ge (d + z ) ^{ 1 + \varepsilon } z ^{1 + \eps} e^{ \varepsilon ( 1 + \varepsilon ) z }. 
    \end{equation}
    In this case, we deduce that we have 
    \begin{equation*}
        \sup _ { v \in S ^ { d - 1 } \cap \set{ u _ *} ^\perp } v ^\top \hess \est L (\beta _ * ) v \le 58e^{B ^2 / \delta },
    \end{equation*}
    with probability at least $ 1 - \exp(-z)$ if $\Omega_b$ already holds and with probability at least $ 1 - 2 \exp(-z)$ otherwise. 
\end{proof}

\begin{proof}[Proof of Lemma \ref{mainlem:hessian_upperbound_low}]\label{def:k_eps}
    We combine Lemma \ref{lem:directioncontrol_smalln} and Lemma \ref{lem:orthospace_smalln}. Moreover, the condition $n \ge K _ \eps d^{1 + \eps} \exp(8z)$ with $K _ \eps = 80 C _ \eps$ ($C _ \eps$ defined in \eqref{def:c_eps}) implies that conditions \eqref{cond:cond1} and \eqref{cond:cond2} hold. 
\end{proof}

\subsection{Proof of the upper bounds on the empirical gradient}

\subsubsection{Large sample size regime}

\textbf{Assumption :} 
\begin{itemize}
    \item $ \hess \est L (\beta _ * ) \preceq C_9H, C_9 \le 116$ holds with high probability. We place ourselves on this event which holds with probability at least $1 - 3\exp(-z)$ due to Lemma \ref{mainlem:hessian_upperbound}, we denote it $\Omega_{HB}$. 
\end{itemize}

\begin{proof}[Proof of Lemma \ref{mainlem:gradient_bound}]
    The proof relies on the PAC-Bayesian inequality for sub-Gamma random vectors. If $X \sim \bcP (\lambda)$, for any $\abs s \le 1$, $\bdE \br{ \exp(s (X - \bdE \br{ X })} \le \exp( \lambda s ^ 2)$, hence, if we denote $\bcX = (X _ 1, \dots, X_ n )$, for any $v \in \bdR ^d $, 
    \begin{align*}
        \bdE \br{ \exp( \dotp{v}{ H ^{-1/2} \grad \est L (\beta _ *) } ) \middle| \bcX} &= \bdE \br{ \exp \p{ \frac{1}{n} \sumn (e^{\dotp{\beta_*}{X_i} } - Y _ i ) \dotp{H ^{-1/2}v}{X_i}} \middle| \bcX} \\
        &= \prod _ { i = 1} ^n \bdE \br{ \exp \p{ \frac{1}{n} (e^{\dotp{\beta_*}{X_i} } - Y _ i ) \dotp{H ^{-1/2} v}{X_i} } \middle| \bcX } \\
        & \le \prod_{i = 1} ^n \exp \p{ \frac{1}{n^2} \dotp{H ^{-1/2} v}{X_i}^2 e^{\dotp{\beta _ * }{X_i }} } \\
        & = \exp \p{ \frac{ \norm{v}_{H ^{-1/2} \hess \est L (\beta _ * ) H ^{-1/2} }^2  }{n}},
    \end{align*}
    provided that $\frac{1}{n} \max _ i \abs{ \dotp{ H ^{ - 1 / 2 } v }{X _ i } } \le 1$. This implies that $H ^{- 1 / 2 } \grad \est L (\beta _ * ) $ is $(\star, \Sigma)$ sub-Gamma conditionally to $\bcX$ in the sense of definition \ref{def:subgammarv} with
    \begin{equation}
        \norm{ v } _ \star = \frac{1}{n} \max _ i \abs{ \dotp{ X _ i }{ H ^{ - 1 / 2 } v}}, \quad \Sigma = \frac{2}{n} H^{- 1 / 2 } \hess \est L (\beta _ * ) H ^{- 1 / 2 }.
    \end{equation}
    Theorem \ref{th:pac_bayes} then implies that, with probability at least $1 - \exp( - z)$,
    \begin{align*}
        \norm{ \grad \est L (\beta _ * ) } _ {H ^{ - 1 } } &\le 14 \sqrt{ \tr( \Sigma ) } + 13 \sqrt{ \norm{ \Sigma } (\log 2 + z ) } \\
        & + 19 w _ d \sqrt{ \log 2 + z } + 22 \Delta _d ( \log 2 + z ),
    \end{align*}
    where $w_d = \bdE \br{ \norm{ g } _ \star \mid \bcX }$ is the Gaussian width and $\Delta _ d = \sup _ { u \in S ^{ d - 1 } } \norm{ u } _ \star$ is the Gaussian diameter. The computation of the Gaussian width and Gaussian diameter terms is detailed in Appendix \ref{appendix:calculations}. The two key terms that are of leading order are the trace term and the operator norm term. On the event $\Omega_{HB}$, bounding them is straightforward as $\Sigma \preceq \frac{2C_9}{n} I _d $ thus $\tr(\Sigma) \le \frac{2C_9d}{n}$ and $\norm{ \Sigma } \le \frac{2C_9}{n}$. Thus, with probability at least $1 - 2 \exp(-z)$, and as $\norm{ H ^{- 1 / 2 } } = \exp(- B ^2 / 4 ) \le 1$ we have, 
    \begin{align*}
        \norm{ \grad \est L (\beta_ * ) } _ { H ^{- 1 } } &\le 22 \times \Big( \sqrt{ \frac{2C_9d}{n} } + \sqrt{ \frac{2C_9z}{n} } + \frac{32 \sqrt{d} \log (2n) }{n} \sqrt{ \log 2 + z } \\
        & +  \sqrt{ 2z( \log 2 + z ) } \frac{ \sqrt{2 \log (2n) }}{n} \\
        & + \frac{4 \sqrt d + \sqrt{ 2 \log (2 n)  } + \sqrt{2z} }{n} \Big). 
    \end{align*}
    For the third term to be negligible compared to $\sqrt{ (d+z)/n}$, we need $n \ge 16$ and $n \ge (\log 2 + z ) \br{ \log ( \log 2 + z )}^2$, for the fourth term we need $ n \ge e $ and $\sqrt{ 2 z ( \log 2 + z) } \frac{ \sqrt{  \log n } }{n } \le \sqrt{ \frac{d+z}{n}}$ and the last term is bounded above by the same quantity whenever $ n \ge 10$. Since we need $\Omega_{HB}$ to hold,  
    \begin{equation}
        n \ge c_7 \p{ e^{(\frac{B ^2 }{ 2 }  (1 + \varepsilon ^{ - 1 } ) )} \p{ (d + z )e^{B \sqrt{ 2z } } }^{1 + \varepsilon} } \vee \p{ e^{B ^2 + 4 B \sqrt{ z }}}, \quad c_7 \le 300,
    \end{equation}
    implies
    \begin{equation}
        \norm{ \grad \est L (\beta_  * ) } _ { H ^{ - 1 } } \le 550 \sqrt{ \frac{d + z}{n} },
    \end{equation}
    with probability at least $ 1 - 5\exp(-z)$. 
\end{proof}

\subsubsection{Intermediate sample size regime}

What changes here is the geometry we control the gradient in, instead of controlling its $H ^{- 1 }$ norm as we were doing before because we had controlled the deviations around $H$, we control it with the stronger $H_0^{-1}$ norm (in the sense $H _ 0 ^{ - 1 } \preceq H ^{-1}$) where we defined, 
\begin{equation*}
    H _ 0 = e ^{ B ^2 / \delta} I_d.
\end{equation*}

\textbf{Assumption :} 
\begin{itemize}
    \item $ \hess \est L (\beta _ * ) \preceq C_{10} e ^{B ^2 / \delta} I _d$ with $C_{10} \le 116$ holds with high probability. We place ourselves on this event which holds with probability at least $1 - 3\exp(-z)$ due to Lemma \ref{mainlem:hessian_upperbound_low}, we denote it $\Omega_{HBl}$. 
\end{itemize}

\begin{proof}[Proof of Lemma \ref{mainlem:gradient_bound_lown}]
    We apply the PAC-Bayes machinery in $H_0^{-1}$ norm. If $X \sim \bcP (\lambda)$, for any $\abs s \le 1$, $\bdE \br{ \exp(s (X - \bdE \br{ X })} \le \exp( \lambda s ^ 2)$, hence, if we denote $\bcX = (X _ 1, \dots, X_ n )$, for any $v \in \bdR ^d $, 
    \begin{align*}
        \bdE \br{ \exp( \dotp{v}{ H_0 ^{-1/2} \grad \est L (\beta _ *) } ) \middle| \bcX} &= \bdE \br{ \exp \p{ \frac{1}{n} \sumn (e^{\dotp{\beta_*}{X_i} } - Y _ i ) \dotp{H_0 ^{-1/2}v}{X_i}} \middle| \bcX} \\
        &= \prod _ { i = 1} ^n \bdE \br{ \exp \p{ \frac{1}{n} (e^{\dotp{\beta_*}{X_i} } - Y _ i ) \dotp{H_0 ^{-1/2} v}{X_i} } \middle| \bcX } \\
        & \le \prod_{i = 1} ^n \exp \p{ \frac{1}{n^2} \dotp{H_0 ^{-1/2} v}{X_i}^2 e^{\dotp{\beta _ * }{X_i }} } \\
        & = \exp \p{ \frac{ \norm{v}_{H_0 ^{-1/2} \hess \est L (\beta _ * ) H_0 ^{-1/2} }^2  }{n}}. 
    \end{align*}
    Provided that $\frac{1}{n} \max _ i \abs{ \dotp{ H_ 0  ^{ - 1 / 2 } v }{X _ i } } \le 1$. This implies that $H_0 ^{- 1 / 2 } \grad \est L (\beta _ * ) $ is $(\star, \Sigma)$ sub-Gamma conditionally to $\bcX$ in the sense of definition \ref{def:subgammarv} with
    \begin{equation}
        \norm{ v } _ \star = \frac{1}{n} \max _ i \abs{ \dotp{ X _ i }{ H_0 ^{ - 1 / 2 } v}}, \quad \Sigma = \frac{2}{n} H_0^{- 1 / 2 } \hess \est L (\beta _ * ) H_0 ^{- 1 / 2 }.
    \end{equation}
    Theorem \ref{th:pac_bayes} then implies that, with probability at least $1 - \exp( - z)$,
    \begin{align*}
        \norm{ \grad \est L (\beta _ * ) } _ {H_0 ^{ - 1 } } &\le 14 \sqrt{ \tr( \Sigma ) } + 13 \sqrt{ \norm{ \Sigma } (\log 2 + z ) } \\
        & + 19 w _ d \sqrt{ \log 2 + z } + 22 \Delta _d ( \log 2 + z ).  
    \end{align*}
    Once again, $w_d = \bdE \br{ \norm{ g } _ \star \mid \bcX }$ is the Gaussian width and $\Delta _ d = \sup _ { u \in S ^{ d - 1 } } \norm{ u } _ \star$ is the Gaussian diameter. Using results from Appendix \ref{appendix:calculations}, and as on the event $\Omega_{HBl}$, $\Sigma \preceq \frac{2C_{10}}{n} I _d $ thus $\tr(\Sigma) \le \frac{2 C_{10} d}{n} $ and $\norm{ \Sigma } \le \frac{2C_{10}}{n}$. The PAC-Bayes inequality yields that with probability at least $1 - 2 \exp(-z)$, as $\norm{ H_0 ^{- 1 / 2 } } \le 1$, we have, 
    \begin{align*}
        \norm{ \grad \est L (\beta_ * ) } _ { H_0 ^{- 1 } } &\le 22 \times \Big( \sqrt{ \frac{2C_{10}d}{n} } + \sqrt{\frac{2C_{10}z}{n} } + \frac{32 \sqrt{d} \log (2n) }{n} \sqrt{ \log 2 + z } \\
        &+ \sqrt{ 2z( \log 2 + z ) } \frac{ \sqrt{2 \log (2n) }}{n} \\
        & + \frac{4 \sqrt d + \sqrt{ 2 \log (2 n)  } + \sqrt{2z} }{n} \Big). 
    \end{align*}
    For the third term to be negligible compared to $\sqrt{ (d+z)/n}$, we need $n \ge 16$ and $n \ge (\log 2 + z ) \br{ \log ( \log 2 + z )}^2$, for the fourth term we need $ n \ge e $ and $\sqrt{ 2 z ( \log 2 + z) } \frac{ \sqrt{  \log n } }{n } \le \sqrt{ \frac{d+z}{n}}$ and the last term is bounded above by the same quantity whenever $ n \ge 10$. 
    Since we need $\Omega_{HBl}$ to hold in the first place. If we take $n$ such that 
    \begin{equation}
        n \ge K_\eps d ^{1 + \eps} \exp(8z),
    \end{equation}
    where $K _ \eps = 80C _\eps$ and the value of $C _ \eps$ can be found in the proof of Lemma $\ref{lem:directioncontrol_smalln}$, the aforementioned conditions on the sample size are satisfied and we have
    \begin{equation}
        \norm{ \grad \est L (\beta_  * ) } _ { H_0 ^{ - 1 } } \le 550 \sqrt{ \frac{d + z}{n} },
    \end{equation}
    with probability at least $ 1 - 5\exp(-z)$. 
\end{proof}

\section{Proof of the main results}
\label{app:theorems}

\subsection{Existence of the MLE}

In the proof of Theorem \ref{th:mainth_existence}, given i.i.d. observations $(Y_i, X_i)_{i \in \discrete{1}{n}}$ from the Poisson model with Gaussian covariates, we denote by $\bfX$ the $ n \times d$ matrix with rows $(X_1, \dots, X_n)$. 

\begin{proof}

\cite{koriyama2025phasetransitionsexistenceunregularized} establish a necessary and sufficient condition for the existence of M-estimators. Specializing their result to the Poisson loss, which satisfies the assumptions of their theorem, yields the following characterization:

\begin{lemma}[\cite{koriyama2025phasetransitionsexistenceunregularized}, Lemma A.2: Poisson version]
\label{thm:koriyama}
The MLE does not exist if and only if
\begin{equation*}
   \exists b_* \in \mathbb{R}^d \setminus \{0\} \quad \text{such that} \quad \forall i \in \discrete{1}{n}, \quad
   \begin{cases}
       x_i^\top b_* = 0 \quad \text{if} \quad y_i > 0, \\
       x_i^\top b_* \le 0 \quad \text{if} \quad y_i = 0 .
   \end{cases}
\end{equation*}
\end{lemma}

In the case where $n < d$ first, the data is deterministically separated by any $b_*$ orthogonal to $\operatorname{Span}(X_1, \dots, X_n)$, hence, the MLE does not exist almost surely in this case. In the case $n \ge d$, using the rotational invariance of the distribution of $X_i \sim \bcN(0,I_d)$ we may assume without loss of generality that $Y_i$ only depends on $X_{i1}$, the first component of the vector $X_i$, for any $i$
\begin{equation*}
    (Y_i, X_{i1}) \indep (X_{i2}, \dots, X_{id}).
\end{equation*}
For $p \in \bdR ^n$, we denote $p = (p _ 1, \dots, p _n)$ and define the cone $C(u,y) \subset \bdR^n$,
\begin{equation*}
    C(u,y) = \set{ p \in \bdR ^n : \exists t \in \bdR, \forall i \in \discrete{1}{n}, t u_i + p_i \begin{cases}
        = 0 \quad \text{if} \quad y_i > 0 &
        \\ \le 0 \quad \text{if} \quad y_i = 0
    \end{cases}}.
\end{equation*}
Moreover, let $(e _ 1, \dots, e _ d)$ denote the canonical basis of $\bdR^d$, an equivalent reformulation of \cite[Lemma A.2]{koriyama2025phasetransitionsexistenceunregularized} yields that the MLE does not exist if,
\begin{equation*}
    \operatorname{Span}(\bfX e_2, \dots, \bfX e_d) \cap C(\bfX e_1, y) \ne \set{0}. 
\end{equation*}
Note that $\operatorname{Span}(\bfX e_2, \dots, \bfX e_d)$ is a rotationally invariant random subspace of $\bdR^n$ with dimension $d - 1$ almost surely and that $C = C(\bfX e _ 1, y)$ is a random cone in $\bdR^n$. We want to obtain an upper bound on the probability of the intersection being nontrivial, to this end, we apply the approximate kinematic formula from Amelunxen et al., we restate the result here for the sake of clarity,

\begin{theorem}[Special case of \cite{amelunxen2014livingedgephasetransitions}, Theorem 7.1]
\label{thm:ALMT-special}
Let $C \subset \mathbb{R}^n$ be a convex cone and $\lambda > 0$. Define the statistical dimension of the cone $C$ by $\delta(C) = \bdE \br{ \norm{\Pi_C(g)} ^2}$ for $g \sim \mathcal{N}(0,I_n)$ where $\Pi_C$ the Euclidean projection onto $C$. For a random subspace $L_{n-m}$ uniformly drawn among $(n-m)$-dimensional subspaces, the following is true
\begin{equation*}
    m \ge \delta(C) + \lambda \implies \bdP \p{ L_{n-m} \cap C \neq  \set{0} } \le p_C(\lambda),
\end{equation*}
where $p_C(\lambda) = 4\exp\p{ - \lambda^2/8(\omega^2(C) + \lambda)}$ and $\omega(C) = \sqrt{ \delta(C) \wedge \delta(C^\circ)}$ is called the transition width, $C ^\circ$ denotes the polar cone of $C$.
\end{theorem}

As $\operatorname{Span}(\bfX e_2, \dots, \bfX e_d)$ is almost surely a $d-1$ dimensional subspace of $\bdR^n$, uniformly drawn among $d - 1$ dimensional subspaces of $\bdR^n$ due to the Gaussian structure, the result ensures that the maximum likelihood estimator exists with probability at least $ 1 - p_C(\lambda)$, provided that
\begin{equation}\label{cond:existence_condition}
    n - d + 1 \ge \delta(C) + \lambda,
\end{equation}
conditionally on $(\bfX e _ 1, y)$. Let us define the following zero and nonzero sets,
\begin{equation*}
    Z = \set{ i : y_ i = 0}, \ P = \set{ i : y_ i > 0}, \ u_Z = (u_i, i \in Z), \ u_P = (u_i, i \in P).
\end{equation*}

From the preceding definitions, $C$ only consists of components from $\bdR u_P$ and $\bdR u_Z + \bdR_-^{ \abs{Z} }$. Thus, up to a permutation $A$ which is an isometric transformation,
\begin{equation*}
    A C \subset \bdR u_P \times (\bdR u_Z + \bdR_-^{ \abs{Z} }),
\end{equation*}
it follows that
\begin{equation*}
    \delta(C) = \delta(AC) \le \delta( \bdR u_P ) + \delta ( \bdR u_Z + \bdR_-^{ \abs{Z} } ) = 1 + \delta ( \bdR u_Z + \bdR_-^{ \abs{Z} } ),
\end{equation*}
as the statistical dimension is invariant under isometric transformation, nondecreasing under inclusion, additive under product and coincides with the dimension for linear subspaces \cite[Proposition 3.1]{amelunxen2014livingedgephasetransitions}. The remaining task is then to derive an upper bound on $\delta( \bdR u _ Z + \bdR_-^{ \abs{ Z }})$. As, 
\begin{equation*}
    \bdR_-^{ \abs{Z} } \subset \bdR u_Z + \bdR_-^{ \abs{Z} } \subset \bdR^{ \abs{Z} },
\end{equation*}
we have conditionally on $(\bfX e_1, y)$
\begin{equation*}
    \frac{ \abs Z }{ 2} \le \delta(\bdR u_Z + \bdR_-^{ \abs{Z} }) \le \abs{Z}.
\end{equation*}
Moreover, $\abs{Z} = \sumn \ind{ Y _ i = 0}$ is binomial with parameters $(n,p)$, $p = \bdE \br{ \exp(-\exp(bg))} $ with $g \sim \bcN(0,1)$ and $b = \norm{\beta_*}$. As $\exp(x) \ge 1$ for any $x \ge 0$ and $ p \le 1$ otherwise, this implies 
\begin{equation*}
    p \le \bdE \br{ e^{-1} \bdOne_{\set{g \ge 0}} } + \bdE \br{ \bdOne_{\set{g < 0}} } \le \frac{e + 1}{2e},
\end{equation*}
also since for any $x < 0$, $\exp(x) < 1$ we also have 
\begin{equation*}
    p \ge \frac{1}{e } \bdP ( g < 0) = \frac{1}{2e}.
\end{equation*}
Hoeffding's inequality yields with probability at least $1 - 2\exp(-z)$ that
\begin{equation*}
    1 + \frac{n}{4e} - \sqrt{2nz} \le \delta(C) \le 1 + n \frac{e+1}{2e} + \sqrt{2nz}. 
\end{equation*}
To control the transition width $\omega(C)$, we also need to control $\delta(C ^\circ)$. Again, by \cite[Proposition 3.1]{amelunxen2014livingedgephasetransitions}, we have that $\delta (C) + \delta (C ^\circ ) = n$, so that
\begin{equation*}
    n \frac{(e - 1)}{2e} - \sqrt{2nz} - 1 \le \delta(C^\circ) \le n \frac{4e - 1}{4e} + \sqrt{2nz} - 1
\end{equation*}
hence
\begin{equation*}
    \frac{n}{4e} - \sqrt{2nz } - 1 \le \omega ^2 (C) \le n \frac{4e - 1}{4e} + \sqrt{ 2 nz }    
\end{equation*}
whenever $ n \ge 10$ with probability at least $1 - 2\exp(-z)$. Moreover, by \eqref{cond:existence_condition}, we have that the MLE exists with probability at least $1 - 2 \exp(-z) - p_C(\lambda)$ whenever, \begin{equation*}
    n \frac{e - 1 }{2e} - \sqrt{2nz} \ge d + \lambda.
\end{equation*}
Setting $ c _ 8 = \frac{e - 1}{4 e}, c _ 9 = \frac{4e - 1}{4e}$ and taking $\lambda = c_8 n$, it suffices that 
\begin{equation*}
    n \frac{e - 1 }{4e} - \sqrt{2nz} \ge d.
\end{equation*}
This is satisfied if 
\begin{equation*}
    \begin{cases}
        \sqrt{2nz} \le \frac{c_8}{2} n, \\
         d \le \frac{ c_8 }{ 2 } n,
    \end{cases}
\end{equation*}
which implies $n \ge \frac{8}{ c_8 ^2 } z \vee \frac{2}{ c_8 }d $, with $\frac{2}{c_8} \le 13$ and $\frac{8}{c_8^2} \le 321$, moreover, this also implies $z/n \le c_8^2/8$ thus, as, 
\begin{equation*}
    p_C(c_8 n ) = 4 \exp \p{ - \frac{ c _ 8 ^2 n ^ 2  }{ 8( \omega^2(C) + c_ 8 n )}} \le 4 \exp \p{ - \frac{c_ 8^2 n } { 8(c_8 + c _ 9 + \sqrt{2z/n})} }
\end{equation*}
$p_C(c_8n) \le 4 \exp \p{ - \frac{c_8^2 n }{ 8((3/2)c_8 + c_9)}}$. The expression simplifies by taking $n \ge \frac{8 ( \frac{3}{2} c _ 8 + c _ 9 ) }{c _ 8 ^2 } z = c _ {10} z $ so that $p_C(c _ 8 n) \le 4 \exp(-z)$, moreover, $c _ {10} = \frac{16 e ( 11 e - 5 ) }{(e-1)^2} \le 367$. 

\end{proof}

\subsection{Proof of the large sample size risk bound}

\begin{proof}[Proof of Theorem \ref{th:mainth_riskbound}]
    The proof is an application of Lemma \ref{lemma:localization}, let $\kappa > 0$ a numerical constant to be determined. Let $\beta \in S_H ( \beta _ *, \kappa \sqrt{ (d + z) / n } )$ and let $n \ge 4(d + z ) e ^{- B ^2 / 2} \kappa ^2$ so that $\sqrt{ \frac{ d + z }{ n } } \le \frac{1}{2} e ^{B ^2  / 4 }$. By Lemma \ref{mainlem:hessian_lowerbound}  and Lemmas \ref{lemma:deterministic_upperbound} and \ref{mainlem:gradient_bound}, we have that whenever there exists $\varepsilon \in (0,1)$ such that,
    \begin{equation*}
        n \ge C_2 \exp \p{ \frac{B ^2}{2} (1 + \varepsilon ^{- 1 }) } \p{ (d + z )e^{4 B \sqrt{z}}}^{1 + \varepsilon},
    \end{equation*}
    with $C_2 \le 31500000$ then the assumptions of the localization lemma all hold with $ \nu = C_{11} \sqrt{ \frac{ d + z } {n }}$, $C_{11} \le 550$, $c _ 0 = 3 \times 10 ^{ - 10 } $ and $c_ 1 = 45$. Moreover, $\nu < c _ 0 r_0 / 2$ reduces to $\kappa > \frac{2 C_{11}}{c_0}$. Thus, for such a choice of $\kappa$, 
    \begin{equation*}
        L ( \est \beta ) - L ( \beta _ * ) \le C_3 \frac{d + z }{ n }, 
    \end{equation*}
    with $C_3 = 2c_1C_{11}^2/c_0^2 \le 6.2 \times 10^{26}$. 
\end{proof}

\subsection{Proof of the intermediate sample size risk bound}\label{proof:smalln}
 
\begin{proof}[Proof of Theorem \ref{th:mainth_riskbound_smalln}]
    The proof is an application of Lemma \ref{lemma:localization} to localize the MLE, let $\kappa > 0$ be a numerical constant to be determined. Let $\eps \in (0,1)$ and define $\delta = \frac{2 \eps }{ 2 + \eps } \le 1$. By Lemma \ref{mainlem:hessian_lowerbound_small} and Lemma \ref{mainlem:gradient_bound_lown}, we have that whenever 
    \begin{equation}\label{cond:condition}
        n \ge K_\eps d^{ 1 + \eps} \exp(8z),
    \end{equation}
    with $C_\eps$ defined in equation \eqref{def:c_eps} and $K _ \eps = 80 C _\eps$ then we simultaneously have with high probability 
    \begin{gather*}
        \norm{ \grad \est  L(\beta_  * ) } _ {H _ 0 ^{ - 1 } } \le C_{13} \sqrt{ \frac{d + z}{n} } = \nu ; \\
        \forall \beta \in \bdR^d, \hess \est L ( \beta) \succeq \frac{1}{100} I_d = c _ 0 I_d,
    \end{gather*}
    with $C_{13} \le 550$ and $c_ 0 \ge 1/100$. From that, let $\beta \in S( \beta _ *, \kappa e^{B ^2 / 2 \delta} \sqrt{ (d + z) / n } )$, we have 
    \begin{align*}
        \est L (\beta ) - \est L (\beta _ *) & \ge - \nu \norm{ \beta - \beta _ * } _ { H _ 0 } + \frac{1}{2} \inf_{ \beta' \in B( \beta _ *, \kappa e^{B ^2 / 2 \delta } \sqrt{ \frac{d + z}{n} }) } \norm{ \beta - \beta _ * } _ { \hess \est L (\beta ')} ^2 \\
        & \ge - \nu \norm{ \beta - \beta _ * } _ { H _ 0 } + \frac{c _ 0 }{ 2 } \norm{ \beta - \beta _  *}^2.
    \end{align*}
    Thus, we have localization of the MLE whenever, 
    \begin{equation}\label{step:comparison_geometry}
        \nu < \frac{c _ 0 }{ 2 } \frac{ \norm{ \beta - \beta _ * } ^2 }{ \norm{ \beta - \beta _ * } _ { H _ 0 } } = \frac{ \kappa c _ 0 }{ 2 } \sqrt{ \frac{d+z}{n}},
    \end{equation}
    Taking $\kappa \ge \frac{2 C_{13}}{c _ 0}$ gives the result, hence with high probability we have, 
    \begin{equation*}
        \norm{ \est \beta - \beta _ * } \le \kappa e^{B ^2 / 2 \delta } \sqrt{ \frac{d + z }{ n }},
    \end{equation*}
    with $\kappa \le 1.1 \times 10 ^5$. The upper bound on the excess risk is obtained in the same way as the first proof, we need a deterministic upper bound on the Hessian of the population loss that holds uniformly over the localization set, in this case, the set $B(\beta _ *, \kappa e ^{B^2/2\delta} \sqrt{(d+z)/n})$. For $n \ge d + z$, which is implied by condition \eqref{cond:condition}, we have for any $\beta$ in the localization set, 
    \begin{equation*}
        \norm{\beta } ^2 \le 2(\norm{ \beta - \beta _ * } ^2  + \norm{ \beta _ * } ^2) \le 2( \kappa ^2 e ^{B ^2 / \delta } + B ^2) \le 4 \kappa ^2 e^{B ^2 / \delta} = R. 
    \end{equation*}
    As for any $\beta \in \bdR^d$, we have $\hess L (\beta ) = e^{\norm{ \beta } ^2 / 2 } ( \norm{ \beta } ^2 u u ^\top +I_d) \preceq 2R e^{R / 2} I _d = M_*$ . We deduce that, if the MLE is localized, 
    \begin{equation*}
        L(\est \beta) - L(\beta _ * ) \le \frac{1}{2} \norm{ \est \beta - \beta _ * } ^2 _ {M _ * } \le R e ^{R / 2} \norm{ \est \beta - \beta _ * } ^2 \le f(B, \delta) \frac{d+z}{n},
    \end{equation*}
    where we defined 
    \begin{equation}\label{expr:f_expression}
        f(B, \delta) = \frac{R^2}{4} e ^{R / 2 } = 4 \kappa ^4 e^{2B ^2 / \delta} e^{ 2\kappa ^2 e ^ {B^2/\delta} }.
    \end{equation}
\end{proof}

\section{Proofs of the PAC-Bayes inequalities}
\label{appendix:pacbayes}

\subsection{The generic PAC-Bayes inequality} \label{appendix:generic_pacbayes}

We first recall here the generic PAC-Bayes inequality, on which the proofs of Theorem \ref{th:pac_bayes} and Theorem \ref{th:pac_bayes_mat} rely.

In the following, we denote $X$ and $\omega$ two independent random variables taking values in measurable spaces $\mathscr{X}$ and $\Omega$ respectively. The distribution of $X$ is denoted $P_X$ while $\omega$ can have several distributions, one, denoted $\pi$ is called the \textit{prior} while the family $(\rho_ \theta)_{\theta \in \Theta}$ is called the set of \textit{posteriors}. All the measures $\rho _ \theta$ are absolutely continuous with respect to $\pi$ and $\gamma _ \theta$ denotes a density of $\rho _ \theta$ with respect to $\pi$. Finally, for any suitably integrable function $f : \mathscr{X} \times \Omega \to \bdR$, we denote,
    \begin{equation}
        \bdE _X \br{ f(X, \omega) } = \int _ \mathscr{X} f(x, \omega) d P_X(x),
    \end{equation}
    its expectation with respect to $P_X$ and for any $\rho \in \set{ \pi} \cup \set{\rho _ \theta, \theta \in \Theta}$,
    \begin{equation}
        \bdE _ \rho \br{ f(X, \omega) } = \int _ \Omega f(X, \omega) d \rho (\omega),
    \end{equation}
    its expectation with respect to $\rho$. The generic PAC-Bayes inequality states the following.

\begin{lemma}{Generic PAC-Bayesian inequality}\label{lemma:pac-bayes}
    For any $z > 0$ and $s > 0$, with probability at least $1 - \exp(-z)$, we have for any $\theta \in \Theta$,
    \begin{equation}
        \bdE _ {\rho _ \theta } \br{ f (X , \omega) } \le \frac{1}{s} \p{ \bdE _ {\rho _ \theta } \br{ \log \bdE _ X \br{ \exp(s f(X , \omega)) }} + \bdE _ {\rho _ \theta } \br{ \log \gamma _ \theta (\omega) } + z}.
    \end{equation}
\end{lemma}

\begin{proof}
    We first have that for any function $g : \Omega \to \bdR$,
    \begin{equation*}
        \forall \theta \in \Theta, \bdE_{\rho _ \theta } \br{ g(\omega) } \le \log \bdE _ \pi \br{ \exp (g (\omega )) } + \bdE _ { \rho _ \theta } \br { \log \gamma _ \theta ( \omega ) }.
    \end{equation*}
    This is due to Jensen's inequality, as for any $\theta \in \Theta$, 
    \begin{equation*}
        \log \bdE _ \pi \br{ \exp( g ( \omega ) )} \ge \log \bdE_{\rho _ \theta} \br{ \frac{\exp(g(\omega))}{ \gamma _ \theta ( \omega ) } \ind { \gamma _ \theta (\omega ) > 0 } }.
    \end{equation*}
    We apply this result to the function $g_X(\omega ) = s f (X , \omega) - \log \bdE _ X \br{ \exp ( s f(X, \omega) )}$, we obtain that for any $\theta \in \Theta$,
    \begin{align*}
        s \bdE _ {\rho _ \theta } \br{ f(X, \omega) } \le \bdE_{\rho _ \theta} \br{ \log \bdE _ X \br{ \exp(s f(X, \omega) )} } + \bdE _ {\rho _ \theta } \br{ \log \gamma _ \theta ( \omega ) } + \log \bdE _ \pi \br{ \exp(g_X (\omega)) }. 
    \end{align*}
    In particular, for any $z > 0$, 
    \begin{align*}
        \bdP ( \exists \theta \in \Theta, s \bdE _ {\rho _ \theta } \br{ f(X, \omega) } &- \bdE_{\rho _ \theta} \br{ \log \bdE _ X \br{ \exp(s f(X, \omega) )} } - \bdE _ {\rho _ \theta } \br{ \log \gamma _ \theta ( \omega ) } > z) \\
        & \le \bdP \p{ \log \bdE _ \pi \br{ \exp(g_X (\omega)) } > z}.
    \end{align*}
    Thus, by Markov's inequality, 
    \begin{align*}
        \bdP \p{ \log \bdE _ \pi \br{ \exp(g_X (\omega)) } > z} \le e^{-z} \bdE _ X \bdE _ \pi \br{ \exp ( g_X (\omega ) ) }.
    \end{align*}
    As $\bdE _ X \br{ \exp( g _ X (\omega) )} = 1$, the result is implied by Fubini's theorem. 
\end{proof}

As mentioned previously, what makes Lemma \ref{lemma:pac-bayes} extremely interesting is the uniformity in $\theta$. While it might seem impractical at first glance, the result allows one to specify freely the choice of prior and posteriors to compute the different quantities, which is where lies the richness of the approach. Moreover, the lemma concretely reduces the problem of obtaining a uniform bound on the desired quantity to controlling a Laplace transform term and a divergence term, in particular, the proof of Theorem \ref{th:pac_bayes} highlights the potential tradeoff between those two terms in the case where the Laplace transform of the random variable is well behaved only in a neighborhood of zero. Indeed, the posteriors must be chosen so that they put most of the mass near the origin so as to control the Laplace term, but at the same time, they have to be sufficiently spread out to avoid the divergence term with the fixed posterior blowing up.

\subsection{Proof of the PAC-Bayesian inequality for sub-Gamma random vectors}
\label{proof:pac_bayes_vec}

\begin{proof}[Proof of Theorem \ref{th:pac_bayes}]
    We apply the generic PAC-Bayes inequality given by Lemma \ref{lemma:pac-bayes}. For this, we set $f(X, \omega) = \dotp{X}{\omega}$ and $\Theta = \set{ \theta \in \bdR ^d , \norm{ \theta } \le 1}$. Let $\beta > 0$ be fixed later and define, for any $\theta$, $\varphi_\theta$ to be the density with respect to the Lebesgue measure of the Gaussian distribution $\bcN(\theta, \beta I _d)$. We denote $\rho _ \theta$ the new distribution whose density with respect to the Lebesgue measure $\lambda$ is given for any $\omega \in \bdR^d$ by
    \begin{equation}
        f _ \theta (\omega) = \frac{1}{C _ {\theta, r} } \varphi _ \theta (\omega) \bdOne_{ \set{\norm{ \omega - \theta}_ \star \le r} },
    \end{equation}
    for some $r > 0$ where $C _ {\theta, r } = \bdP ( \norm{ \omega - \theta } _ \star \le r)$ is the normalizing constant ensuring that $\int f _ \theta d \lambda = 1$. By symmetry of the constraining set, it follows that $\bdE _{\rho _ \theta } \br{ \omega } = \theta$, hence, 
    \begin{equation}
        \sup _ { \theta \in \Theta} \bdE _{\rho _ \theta } \br{ f(X, \omega ) } = \sup_{ \norm{ \theta } \le 1 } \dotp{ X }{ \bdE _ {\rho _ \theta } \br{ \omega } }  = \sup_{ \norm{ \theta } \le 1 } \dotp{ X }{ \theta } = \norm{ X }.
    \end{equation}
    We define the prior distribution as $\pi = \bcN(0 , \beta I_d)$, every $\rho_\theta$ is then absolutely continuous with respect to $\pi$ and 
    \begin{align*}
        \gamma _ \theta (\omega ) &= \frac{1}{ C _ {\theta , r}} \exp \p{ \frac{1}{2 \beta } ( \norm{ \omega } ^2 - \norm{ \omega - \theta } ^2 )} \bdOne_{ \set{\norm{ \omega - \theta}_ \star \le r} } \\
        &= \frac{1}{ C _ {\theta , r}} \exp \p{ \frac{1}{2 \beta } ( \norm{ \theta } ^ 2  + 2 \dotp{ \omega - \theta}{\theta})}\bdOne_{ \set{\norm{ \omega - \theta}_ \star \le r} }.
    \end{align*}
    Therefore, 
    \begin{equation}
        \bdE _{ \rho _ \theta} \br{ \log \gamma _ \theta (\omega) } = - \log C _ {\theta , r } + \frac{\norm{\theta} ^ 2 }{2 \beta } \le - \log C _ {\theta , r } + \frac{1}{2 \beta }.
    \end{equation}
    We have to bound from below $C_{\theta, r } = \bdP (\norm{ \omega - \theta } _ \star \le r) = \bdP ( \norm{ g } _ \star \le r / \sqrt{ \beta })$ where $g \sim \bcN(0,I_d)$. Markov gives $C _ {\theta ,r } \ge 1/2$ if $r = 2 w_d \sqrt{ \beta }$, in this case, 
    \begin{equation}
        \bdE _{ \rho _ \theta} \br{ \log \gamma _ \theta (\omega) } \le \log 2 + \frac{ 1  }{ 2 \beta } . \label{eq:bound1_pacbayes}
    \end{equation}
    Now, for $\omega$ sampled from $\rho _ \theta$, $\norm{ \omega } _ \star \le \norm{ \omega - \theta } _ \star + \norm{ \theta } _ \star \le 2 w_d \sqrt \beta + \Delta _ d$, hence, $\norm{s \omega } _ \star \le 1$ whenever $\abs{s } \le 1/b$ with $b = 2 w_d \sqrt \beta + \Delta _ d$. For such a value of $s$, 
    \begin{align}
        \bdE _{\rho _ \theta } \br{ \log \bdE _ X \br{ \exp (s f(X , \omega) ) } } &= \bdE _{\rho _ \theta } \br{ \log \bdE _ X \br{ \exp (s \dotp{\omega}{X} ) } } \notag \\
        & \le \frac{ s ^ 2 }{ 2 } \bdE _{\rho _ \theta } \br{ \norm{ \omega } _ \Sigma ^2 } \notag \\
        &= \frac{ s ^ 2 }{ 2 } \p{ \bdE _{\rho _ \theta } \br{ \norm{ \omega - \theta } _ \Sigma ^2  } + \bdE _{\rho _ \theta } \br{ \norm{ \theta } _ \Sigma ^2 }} \notag \\
        & \le \frac{ s ^ 2 }{ 2 } \p{ \beta \tr(\Sigma) + \norm{ \Sigma }}. \label{eq:bound2_pacbayes}
    \end{align}
    As $X$ is $(\star, \Sigma)$ sub-Gamma in the sense of \ref{def:subgammarv}. Gathering \eqref{eq:bound1_pacbayes} and \eqref{eq:bound2_pacbayes}, Lemma \ref{lemma:pac-bayes} yields for any $z > 0, \beta > 0$ and $\abs{ s } < 1/(2 \sqrt \beta w_ d + \Delta _ d)$, with probability at least $ 1 - \exp(-z)$, for any $\theta \in \Theta$, 
    \begin{align*}
        \dotp{ X }{ \theta } &\le \frac{1}{s} \p{ \frac{s^2}{2} \p{ \norm{ \Sigma } + \beta \tr(\Sigma ) } + \frac{1}{2 \beta } + \log 2 + z} \\
        &= s \beta \tr(\Sigma ) + \frac{1}{2 s \beta } + \frac{s }{ 2} \norm{ \Sigma } + \frac{ \log 2 + z }{ s }. 
    \end{align*}

    \textbf{Optimization :} We do not seek to optimize the constants, for each case and subcase, we verify that $s \le \frac{1}{2 \Delta _d }$ and $s^2 \beta < \frac{1}{16 w _ d ^2 }$ which implies $s (\Delta _ d + 2 \sqrt \beta w _ d ) \le \frac{1}{2} + \frac{1}{2}.$

\subsubsection*{Case $\norm{ \Sigma } > \Delta _ d ^2 ( \log 2 + z ) $:}
\subsubsection*{Sub-case $\min( \norm{ \Sigma }( \log 2 + z ) ,  \tr(\Sigma)) > w_d^2 (\log 2 + z )$:}
Taking,
\begin{equation*}
    s = \frac{1 } { 4 } \sqrt{ \frac{( \log 2 + z ) }{\norm{ \Sigma }} }, \quad \beta = \frac{1}{4s} \frac{1}{ \sqrt{ \tr ( \Sigma ) } },
\end{equation*}
in this subcase, we indeed have $s \le \frac{1}{4 \Delta _d } \le \frac{1}{2 \Delta _ d } $, also
\begin{align*}
    s^2 \beta = \frac{s}{4 \sqrt{ \tr( \Sigma ) } } = \frac{\sqrt{ \log 2 + z } }{16 \sqrt{ \norm{ \Sigma } \tr( \Sigma ) }} \le \frac{1}{16 w _ d ^2 }.
\end{align*}
As $\norm{ \Sigma } \tr( \Sigma ) > w_d ^ 4 ( \log 2 + z ) $, 
\begin{align*}
    s \beta \tr(\Sigma ) + \frac{1}{2 s \beta} + \frac{s}{2} \norm{ \Sigma } + \frac{ \log 2 + z }{ s } &= \frac{1}{ 4 } \sqrt{ \tr ( \Sigma ) } + 2 \sqrt{ \tr(\Sigma ) } + \frac{1}{4} \sqrt{ \norm{ \Sigma } ( \log 2 + z ) } \\
    & + 4 \sqrt{ \norm{ \Sigma } (\log 2 + z ) }. \\
    & \le 3 \sqrt{ \tr( \Sigma ) } + 5 \sqrt{ \norm{ \Sigma } ( \log 2 + z ) }. 
\end{align*} 

\subsubsection*{Sub-case $\tr (\Sigma ) > w_ d ^2 ( \log 2 + z ) > \norm{ \Sigma } (\log 2 + z )$:}

Taking, 
\begin{equation*}
    s = \frac{\sqrt{\log 2 + z }}{2 w _ d }, \quad \beta = \frac{1}{8 s \sqrt{ \tr( \Sigma ) } },
\end{equation*}
in this subcase, we indeed have $s \le \frac{\sqrt{ \log 2 + z } } { 2 \sqrt{ \norm{\Sigma } }} \le \frac{1}{2 \Delta _ d }$, also 
\begin{equation*}
    s^2 \beta = \frac{\sqrt{ \log 2 + z } }{16 w _ d \str } \le \frac{1}{16 w _d ^2 }.
\end{equation*}
Moreover 
\begin{align*}
    s \beta \tr(\Sigma ) + \frac{1}{2 s \beta} + \frac{s}{2} \norm{ \Sigma } + \frac{ \log 2 + z }{ s } &= \frac{1}{8} \str + 4 \str + \frac{\norm{ \Sigma} \sqrt{ \log 2 + z } }{4 w _ d } \\
    & + 2 w _ d \sqrt{ \log 2 + z }.
\end{align*} 
As in this subcase, $\nga \le w _ d ^2 $ this yields 
\begin{equation*}
    s \beta \tr(\Sigma ) + \frac{1}{2 s \beta} + \frac{s}{2} \norm{ \Sigma } + \frac{ \log 2 + z }{ s } \le 5 \str + 3 w _ d \sqrt{ \log 2 + z }.
\end{equation*}

\subsubsection*{Sub-case $ \norm{\Sigma } (\log 2 + z) > w _ d ^2 (\log 2 + z ) > \tr(\Sigma ) $:}

Taking, 
\begin{equation*}
    s = \frac{1}{2} \sqrt{ \frac{\log 2 + z}{ \norm{\Sigma } }}, \quad \beta = \frac{1}{4 ( \log 2 + z ) },
\end{equation*}
in this subcase, we indeed have, $s \le \frac{1}{2 \Delta _ d}$, also
\begin{equation*}
    s^2 \beta = \frac{1}{16} \frac{\log 2 + z }{  \norm{\Sigma } (\log 2 + z )} = \frac{1}{16 } \frac{1}{ \norm{\Sigma } } < \frac{1}{16 w _ d ^2 }.
\end{equation*}
Furthermore, 
\begin{align*}
    s \beta \tr(\Sigma ) + \frac{1}{2 s \beta} + \frac{s}{2} \norm{ \Sigma } + \frac{ \log 2 + z }{ s } &= \frac{1}{8} \frac{\tr(\Sigma ) }{ \ntr } \frac{1}{ \sqrt{ \log 2 + z } } + 4 \sqrt{ \norm{ \Sigma } (\log 2 + z ) } \\
    & + \frac{1}{2} \sqrt{ \nga (\log 2 + z ) } + 2 \sqrt{ \nga ( \log 2 + z ) }.
\end{align*} 
As in this subcase $\tga < w_d^2 (\log 2 + z)$,  
\begin{equation*}
    \frac{1}{8} \frac{\tr(\Sigma ) }{ \ntr } \frac{1}{ \sqrt{ \log 2 + z } } < w_d \sqrt{ (\log 2 + z ) }
\end{equation*}
so 
\begin{equation*}
    s \beta \tr(\Sigma ) + \frac{1}{2 s \beta} + \frac{s}{2} \norm{ \Sigma } + \frac{ \log 2 + z }{ s } \le w_d \sqrt{\log 2 + z } + 7\sqrt{ \nga (\log 2 + z) }.
\end{equation*}

\subsubsection*{Sub-case $w_d ^2 ( \log 2 + z ) > \max ( \nga ( \log 2 + z ) , \tga) $:}

Taking 
\begin{equation*}
    s = \frac{1}{4} \frac{ \sqrt{ \log 2 + z } }{ w _ d }, \quad \beta = \frac{1}{4s } \frac{1}{ w _ d \sqrt{ \log 2 + z } },
\end{equation*}
in this subcase,
\begin{equation*}
    w_d ^2 > \nga > \Delta_ d ^ 2 (\log 2 + z ) \implies \frac{1}{4} \frac{ \sqrt{ \log 2 + z }} { w _d } < \frac{1}{4 \Delta _ d } < \frac{1}{2 \Delta _ d }. 
\end{equation*}
Also,
\begin{equation*}
    s^2 \beta = \frac{1}{16} \frac{1 }{ w _ d ^2 },
\end{equation*}
and,
\begin{align*}
    s \beta \tr(\Sigma ) + \frac{1}{2 s \beta} + \frac{s}{2} \norm{ \Sigma } + \frac{ \log 2 + z }{ s } &= \frac{1}{4} \frac{\tga }{ w _ d \sqrt{ \log 2 + z } } + 2 w _ d \sqrt{ \log 2 + z } \\
    &+ \frac{1}{4} \frac{ \nga \sqrt{ \log 2 + z } }{ w _ d } + 4 \frac{w_d (\log 2 + z) }{ \sqrt{ \log 2 + z } } \\
    & \le 6 w _d \sqrt{ \log 2 + z } + \frac{1}{4} \frac{ \nga \sqrt{ \log 2 + z } }{ w _ d } \\
    &  + \frac{1}{4} \frac{\tga }{ w _ d \sqrt{ \log 2 + z } }.
\end{align*} 
As in this subcase
\begin{equation*}
    \tga < w_d ^2 (\log 2 + z ), \quad w_d \sqrt{ \log 2 + z } > \ntr \sqrt{ \log 2 + z },
\end{equation*}
we conclude that 
\begin{equation*}
    s \beta \tr(\Sigma ) + \frac{1}{2 s \beta} + \frac{s}{2} \norm{ \Sigma } + \frac{ \log 2 + z }{ s } \le 7 w _d \sqrt{ \log 2 + z } + \sqrt{ \nga (\log 2 + z ) }. 
\end{equation*}

\subsubsection*{Case $\nga \le \Delta_d ^2 (\log 2 + z ) $:}

\subsubsection*{Sub-case $\min( \nga (\log 2 + z ), \tga) > w _ d ^2 ( \log 2 + z ) $:}

Taking 
\begin{equation*}
    s = \frac{1}{4 \Delta _ d }, \quad \beta = \frac{1}{4s} \frac{ 1 }{ \str },
\end{equation*}
in this subcase, $s \le 1 / 2 \Delta _d $ and 
\begin{equation*}
    s^2 \beta = \frac{1}{ 16 } \frac{1}{ \Delta _ d \str } \le \frac{1}{16} \sqrt{ \frac{ \log 2 + z }{ \tga \nga  }} \le \frac{1}{16 w_ d ^2 },
\end{equation*}
as
\begin{equation*}
    \tga \nga > w_ d ^4 (\log 2 + z ).
\end{equation*}
We also have,
\begin{align*}
    s \beta \tr(\Sigma ) + \frac{1}{2 s \beta} + \frac{s}{2} \norm{ \Sigma } + \frac{ \log 2 + z }{ s } &= \frac{1}{4} \str + 2 \str + \frac{1}{8 \Delta _d } \nga \\
    & + 4 \Delta _d (\log 2 + z ) \\
    & \le 3 \str + 5 \Delta _ d (\log 2 + z ), 
\end{align*} 
as 
\begin{equation*}
    \nga / \Delta _ d \le \Delta _ d (\log 2 + z ),
\end{equation*}
in this subcase. 

\subsubsection*{Sub-case $ \tga > w _ d ^2 ( \log 2 + z ) > \nga ( \log 2 + z ) $:}
Taking,
\begin{equation*}
    s = \min \p{ \frac{1}{4 \Delta _ d }, \frac{\str}{ 4 w_  d ^2 } }, \quad \beta = \frac{1}{ 4 s } \frac{1 }{ \str },
\end{equation*}
in this subcase, $s < \frac{1 }{ 2 \Delta _ d}$ and 
\begin{equation*}
    s ^2 \beta = \min \p{ \frac{1}{16 \Delta _ d \str}, \frac{1}{16 w _ d ^2 }} \le \frac{1}{16 w _ d ^2 }.
\end{equation*}
Furthermore,
\begin{equation*}
    s \beta = \frac{1}{ 4 \str}, \quad \frac{1}{2 s \beta} = 2 \str, 
\end{equation*}
and,
\begin{equation*}
    \frac{s}{2} \nga = \min \p{ \frac{\nga }{ 4 \Delta _ d }, \frac{\nga \str }{ 4 w _d ^2 } } \le \frac{ \nga \str }{ 4 w _d ^2 } \le \frac{1}{4} \str.
\end{equation*}
Finally,
\begin{align*}
    \frac{ \log 2 + z }{ s } &\le \max \p{ (\log 2 + z ) 4 \Delta _ d, (\log 2 + z ) \frac{4 w _ d ^2 }{\str } } \\
    &\le 4 \max \p{ \Delta _ d (\log 2 + z ), w_d \sqrt{ \log 2 + z } },
\end{align*}
thus,
\begin{equation*}
    s \beta \tr(\Sigma ) + \frac{1}{2 s \beta} + \frac{s}{2} \norm{ \Sigma } + \frac{ \log 2 + z }{ s } \le 3 \str + 4 \Delta _ d(\log 2 + z ) + 4 w_d \sqrt{ \log 2 + z}.
\end{equation*}

\subsubsection*{Case $\nga ( \log 2 + z ) > w_ d ^2 ( \log 2 + z ) > \tga$:}
Taking,
\begin{equation*}
    s = \frac{1}{2 \Delta _ d }, \quad \beta = \frac{1}{4 (\log 2 + z )},
\end{equation*}
\begin{align*}
    s^2 \beta = \frac{1}{ 16 } \frac{1}{ \Delta _ d ^2 ( \log 2 + z ) } \le \frac{1}{16} \frac{1}{ \nga } \le \frac{1}{16} \frac{1}{ w _ d ^2}. 
\end{align*}
In this subcase, $s \beta = \frac{1}{8 \Delta _ d (\log 2 + z ) }$ hence 
\begin{align*}
    s \beta \tga + \frac{1}{2 s \beta } + \frac{s }{ 2 } \nga + \frac{ \log 2  +z }{ s } &= \frac{\tga }{ 8 \Delta _ d (\log 2 + z ) } + 4 \Delta _ d (\log 2 + z ) \\
    & + \frac{\nga}{4 \Delta _ d } + 2 \Delta _ d (\log 2 + z ).  
\end{align*}
Also, 
\begin{equation*}
    \frac{\tga}{ \Delta _ d (\log 2 + z ) } \le \frac{ \nga }{ \Delta _ d } \le \Delta _ d ( \log 2 + z).
\end{equation*}
As $\nga \le \Delta _ d ^2 ( \log 2 + z ) $. Thus 
\begin{equation*}
    s \beta \tr(\Sigma ) + \frac{1}{2 s \beta} + \frac{s}{2} \norm{ \Sigma } + \frac{ \log 2 + z }{ s } \le 8 \Delta _ d (\log 2 + z ). 
\end{equation*}

\subsubsection*{Sub-case $ w _ d ^2 ( \log 2 + z ) > \max ( \nga (\log 2 + z ) , \tga )$:}

Taking,
\begin{equation*}
    s = \min \p{ \frac{1}{4\Delta _ d }, \frac{ \sqrt{ \log 2 + z }}{ w _ d }}, \quad \beta = \frac{1}{4 s } \frac{ 1 }{ w _ d \sqrt{ \log 2 + z } }
\end{equation*}
\begin{equation*}
    s ^2 \beta = \min \p{ \frac{1}{ 16 } \frac{ 1 }{ \Delta _ d w _ d \sqrt{ \log 2 + z } }, 1}
\end{equation*}
In this subcase $\nga < w _ d ^2 $ and
\begin{equation*}
    \Delta _ d \sqrt{ \log 2 + z } \ge \nga \implies \frac{1}{ \Delta _ d \sqrt{ \log 2 + z } } \le \frac{ 1 }{ \ntr } \le \frac{1}{ w _ d }
\end{equation*}
which implies $s^2 \beta \le 1/16w_d ^2$. Also,
\begin{align*}
    s \beta \tga + \frac{1}{2 s \beta } + \frac{s }{ 2 } \nga + \frac{ \log 2  +z }{ s } &= \frac{1}{4} \frac{ \tga }{ w _ d \sqrt{ \log 2 + z } } + 2 w_ d \sqrt{ \log 2 + z } \\
    & + \min \p{ \frac{ \nga}{4 \Delta _ d }, \frac{\nga \sqrt{ \log 2 + z }}{ w_d }} \\
    & + (\log 2 + z ) \max \p{ 4 \Delta _ d, \frac{w_d }{ \sqrt{ \log 2 + z } }}.
\end{align*}
Notice that
\begin{equation*}
    \frac{ \tga }{ w _ d \sqrt{ \log 2 + z } } \le w_d \sqrt{ \log 2 + z },
\end{equation*}
and also
\begin{equation*}
    \frac{ \nga }{ \Delta _ d } \le \Delta _ d (\log 2 + z),
\end{equation*}
hence 
\begin{equation*}
    s \beta \tr(\Sigma ) + \frac{1}{2 s \beta} + \frac{s}{2} \norm{ \Sigma } + \frac{ \log 2 + z }{ s } \le 4 w _ d \sqrt{ \log 2 + z } + 5 \Delta _ d (\log 2 + z ).
\end{equation*}

This concludes the proof for all the different cases. 
\end{proof}

\subsection{Proof of the PAC-Bayesian inequality for sub-Gamma random matrices}
\label{proof:pac_bayes_mat}

\begin{proof}[Proof of Theorem \ref{th:pac_bayes_mat}]
    The proof relies on an application of the generic PAC-Bayesian inequality (Lemma \ref{lemma:pac-bayes}) to the function
    \begin{equation*}
        f(x, \omega ) = \dotp{ x }{ \omega } _ F,
    \end{equation*}
    where $\omega = U V ^\top$ is a random matrix defined from two independent random vectors $U,V$. The PAC-Bayes prior $\pi$ and the posteriors $(\rho_\theta)_{\theta}$ of $(U,V)$ are respectively $\pi = \pi_U \otimes \pi_V$ and $(\rho _ \theta) _ \theta = (\rho _ u \otimes \rho _v)_{(u,v)}$ that we define as follows,

    \textbf{Distributions of $U$:} For any $u \in \bcE _ U $, we define the distribution $\rho_u$ as the conditional Gaussian restricted to $B (u, r_U)$ for $r_U > 0$ with covariance matrix $\Gamma _ U/\beta _ U $. $\rho _ u $ has a density $f_u$ with respect to the Lebesgue measure given by, 
    \begin{equation*}
        f_u(x) = \frac{1}{C _ U\sqrt{(2 \pi / \beta_U)^n\det \Gamma _ U} } \exp \p{ - \frac{\beta _ U }{2} ( x - u)^\top \Gamma _ U ^{- 1 } (x - u) } \ind { \norm{ x - u } \le r _ U }, 
    \end{equation*}
    where $C_ U = \bdP ( \bcN(0 , \Gamma _ U / \beta _ U) \in B (0 , r _ U))$ is the normalizing constant. By Markov's inequality, we see that if we pick $r _ U = 2 \sqrt{ \tr \Gamma _ U / \beta _ U }$ we have $C _ U \ge 1/ 2$. Notice that by symmetry of the domain, $\bdE _ { \rho _ u} \br{ U }= u$. We define the prior for $U$ by, 
    \begin{equation*}
        \pi_U = \bcN (0, \Gamma_ U / \beta _ U ). 
    \end{equation*}
    
    \textbf{Distributions of $V$:} We follow the same logic, for any $v \in \bcE _ V $, we define $\rho_v$ as the conditional Gaussian restricted to $B (v, r_V)$ for $r_V > 0$ with covariance matrix $\Gamma _ V/\beta _ V$. Similarly,  
    \begin{equation*}
        f_v(x) = \frac{1}{C _ V\sqrt{(2 \pi / \beta_V)^n\det \Gamma _ V} } \exp \p{ - \frac{\beta _ V }{2} ( x - v)^\top \Gamma _ V ^{- 1 } (x - v) } \ind { \norm{ x - v } \le r _ V }, 
    \end{equation*}
    We pick $r _ V = 2 \sqrt{ \tr \Gamma _ V / \beta _ V }$ so that $C _ V \ge 1/ 2$. We also have $\bdE _ { \rho _ v} \br{ V }= v$. We define the prior for $V$ by, 
    \begin{equation*}
        \pi_V = \bcN (0, \Gamma_ V / \beta _ V ). 
    \end{equation*}
    From those two definitions, the calculations are straightforward, 
    \begin{equation*}
        \bdE _{ \rho _ u } \br{ \norm{ U } ^ 2 } = \norm{ u } ^ 2 + \bdE \br{ \norm{ \bcN ( 0 , \Gamma_ U / \beta _ U ) } ^2 } \le \norm{ \Gamma _ U } + \frac{ \tr \Gamma _ U }{ \beta _ U }.
    \end{equation*}
    For the divergence term that we denote $K (\rho _ u, \pi)$, writing $\gamma _ u (x) = \frac{ d \rho _ u }{d \pi } (x) $, 
    \begin{equation*}
        \gamma _ u ( x ) = \frac{1}{C _ U } \exp \p{ \frac{\beta _ U}{2} ( 2 u ^\top \Gamma _ U ^{- 1 } x - u ^\top \Gamma _ U ^{- 1 } u)},
    \end{equation*}
    this yields,
    \begin{equation*}
        K (\rho _ u, \pi) = - \log C _ U + \frac{\beta _ U }{2} \norm{ u } _ { \Gamma _ U ^{- 1 } } ^2 \le \log 2 + \frac{ \beta _ U }{2}. 
    \end{equation*}
    Same thing for $V$, 
    \begin{equation*}
        \bdE _{ \rho _ v } \br{ \norm{ V } ^ 2 } \le \norm{ \Gamma _ V } + \frac{ \tr \Gamma _ V }{ \beta _ V }, \quad K(\rho _ v, \pi_V) \le \log 2 + \frac{\beta _ V }{ 2 }.
    \end{equation*}
    Moreover,
    \begin{equation*}
        K (\rho _ \theta, \pi) = K ( \rho _ u \otimes \rho _ v, \pi_U \otimes \pi _V) = K ( \rho _ u , \pi_U ) + K( \rho _ v, \pi _ V),
    \end{equation*}
    and by independence (assuming $s$ sufficiently small as we will show),
    \begin{equation*}
        \bdE _{\rho _ \theta } \br{ \log \bdE _{\mathbf X} \br{ \exp( s\dotp{ \mathbf X }{ UV^\top } _ F )}} \le s^2 \sigma ^2 \bdE_{\rho _ u } \br{ \norm{ U } ^2 } \bdE_{\rho _ v} \br{ \norm{V} ^2 }.  
    \end{equation*}
    If $u,v$ are sampled from some $\rho _ \theta$, then clearly
    \begin{equation*}
        \norm{u} \le \sqrt { \norm{ \Gamma _ U } } + 2 \sqrt{\frac{\tr \Gamma _ U }{ \beta _ U }}, \quad \norm{v} \le \sqrt { \norm{ \Gamma _ V } } + 2 \sqrt{\frac{\tr \Gamma _ V }{ \beta _ V }}.
    \end{equation*}
    As, $\sqrt { \norm{ \Gamma _ U } } + 2 \sqrt{\frac{\tr \Gamma _ U }{ \beta _ U }} \le \sqrt{ 2 \norm{ \Gamma _ U } + 8 \frac{ \tr ( \Gamma _ U)}{ \beta _ U }}$. The generic PAC-Bayesian inequality \ref{lemma:pac-bayes} thus gives, that for any $s$ such that 
    \begin{equation*}
        \abs{s} \le \p{ b \sqrt{ 2 \norm{ \Gamma _ U } +  8\frac{\tr \Gamma _ U}{ \beta _ U } } \sqrt{ 2\norm{ \Gamma _ V } + 8\frac{\tr \Gamma _ V }{ \beta _ V }}}^{-1},
    \end{equation*}
    for any $z > 0$, with probability at least $ 1 - \exp(-z)$, for any $(u ,v) \in \bcE _ U \times \bcE _ V $,
    \begin{equation*}
        \dotp{ \mathbf X }{ u v ^\top } _ F \le s \sigma ^2 \p{ \norm{ \Gamma _ U } + \frac{\tr \Gamma _ U }{ \beta _ U } }\p{ \norm{ \Gamma _ V } + \frac{\tr \Gamma _ V }{ \beta _ V } } + \frac{ \beta _ U + \beta _ V }{ 2 s } + \frac{ 2 \log 2 + z }{s}. 
    \end{equation*}
    It remains to optimize this bound in $s$. 
    
    \textbf{Optimization:} Let $\alpha \ge b$ and set, 
    \begin{equation*}
        s = \p{ \alpha \sqrt{ 2 \norm{ \Gamma _ U } +  8\frac{\tr \Gamma _ U}{ \beta _ U } } \sqrt{ 2\norm{ \Gamma _ V } + 8\frac{\tr \Gamma _ V }{ \beta _ V }}}^{-1}.
    \end{equation*}
    In this case, 
    \begin{align*}
        \dotp{ \bfX }{ u v ^\top } _ F &\le \frac{\sigma ^2}{ \sqrt 2 \alpha} \sqrt{ \norm{ \Gamma _ U } +  \frac{\tr \Gamma _ U}{ \beta _ U } } \sqrt{ \norm{ \Gamma _ V } +  \frac{\tr \Gamma _ V}{ \beta _ V } }  \\
        & + \alpha \sqrt2 \sqrt{ \p{ \norm{ \Gamma _ U } +  \frac{\tr \Gamma _ U}{ \beta _ U } } \p{ \norm{ \Gamma _ V } +  \frac{\tr \Gamma _ V}{ \beta _ V } } } \p{ ( \beta _ U + \beta _ V) + 4 \log 2 + 2 z }. 
    \end{align*} 
    We now set $\beta _ U = r (\Gamma _ U ) = \frac{ \tr(\Gamma _ U) }{ \norm{ \Gamma _ U }}$, $\beta _ V = r (\Gamma _ V ) = \frac{ \tr(\Gamma _ V) }{ \norm{ \Gamma _ V }}$ the effective ranks. 
This implies that, 
\begin{equation*}
    \dotp{ \bfX }{ u v ^\top } _ F \le \frac{\sigma ^2 }{ \alpha } \sqrt{ 2 \norm{ \Gamma _ U } \norm{ \Gamma _ V }} + \alpha 2 \sqrt{ 2 } \sqrt{ \norm{ \Gamma _ U } \norm{ \Gamma _ V } } (r (\Gamma _ U ) + r ( \Gamma _ V) + 2 (2 \log 2 + z)). 
\end{equation*}

\begin{equation*}
    \dotp{ \bfX }{ u v ^\top } _ F \le \sqrt{ 2 \norm{ \Gamma _ U } \norm{ \Gamma _ V }} \p{ \frac{\sigma ^2 }{ \alpha } + 2 \alpha  (r (\Gamma _ U ) + r ( \Gamma _ V) + 2 (2 \log 2 + z))}. 
\end{equation*}

We distinguish the two possible cases for $\alpha$, 

\textbf{Case 1:}
\begin{equation*}
    r(\Gamma _ U ) + r ( \Gamma _ V ) + 2 (2 \log 2 + z) < \frac{2 \sigma ^2 }{ b ^2 }. 
\end{equation*}
In this case we set, 
\begin{equation*}
    \alpha = \frac{ \sqrt{ 2 } \sigma } {\sqrt{ r(\Gamma _ U ) + r ( \Gamma _ V ) + 2 (2 \log 2 + z) }} > b.
\end{equation*}
With this choice, 
\begin{equation*}
    \dotp{ \bfX }{ u v ^\top } _ F \le 4\sigma\sqrt{ 2 \norm{ \Gamma _ U } \norm{ \Gamma _ V } } \bcC(U,V,z).
\end{equation*}

\textbf{Case 2:}
\begin{equation*}
    r(\Gamma _ U ) + r ( \Gamma _ V ) + 2 (2 \log 2 + z) \ge \frac{2 \sigma ^2 }{ b ^2 }. 
\end{equation*}
In this case we set, 
\begin{equation*}
    \alpha = b.
\end{equation*}
With this choice, 
\begin{equation*}
    \dotp{ \bfX }{ u v ^\top } _ F \le 3 b \sqrt{ 2 \norm{ \Gamma _ U } \norm{ \Gamma _ V } } \bcC(U,V,z) ^ 2.
\end{equation*}
which concludes the proof. 
\end{proof}

\section{Technical lemmas and calculations}
\label{appendix:technical}

\subsection{Calculations of the terms involved in the PAC-Bayes gradient bounds}
\label{appendix:calculations}

The approach is the same for the proofs of both Lemma \ref{mainlem:gradient_bound} and Lemma \ref{mainlem:gradient_bound_lown}, hence throughout the calculations $M \in \set{H, H _ 0}$ denotes the corresponding matrix used to control the gradient. In both cases $\norm{M ^{-1/2}} \le 1$. 

\textbf{Gaussian width :} Recall that $w_d = \bdE \br{ \norm{ g } _ \star \mid \bcX } = \frac{1}{n} \bdE \br{\max _ i \abs{ \dotp{g}{M^{-1/2}X _ i}} \mid \bcX}$ for $g \sim \bcN ( 0, I _d )$. Conditionally to $\bcX$, we then have $\norm{g} _ \star = n^{-1} \max _ i \abs{ Z _ i }$ where $Z_i \sim \bcN(0, \sigma_i^2)$ and $\sigma_i^2 = \norm{X_i}_{M^{-1}}^2$. This yields, 
\begin{align*}
    \bdE \br{ n \norm{g}_* \middle| \bcX } &= \bdE \br{ \max _ i \abs{ Z_ i } } = \frac{1}{\lambda} \bdE \br{ \log \max _ i \exp(\lambda  \abs{Z _ i } ) } \le \frac{1}{\lambda} \log \bdE \br{ \max _ i \exp (\lambda \abs{ Z _ i } )} \\
    & \le \frac{1}{\lambda} \log \sumn \bdE \br{ \exp (\lambda \abs{ Z _ i } )} \le \frac{ 1 }{ \lambda } \log \sumn 2 \exp \p{ \frac{ \lambda ^2 \sigma_i ^2 }{2} } \\
    & \le \frac{1}{\lambda } \log ( 2 n ) + \frac{\lambda}{2} \max _ i \sigma_i^2 
\end{align*}   
Thus
\begin{equation}
    \bdE \br{ \norm{g} _ * \mid \bcX } \le \sigma _ * \frac{ \sqrt{2 \log(2n) }}{n}, \quad \sigma_* = \max _ i \norm{ X _ i } _ {M^{-1}}.
\end{equation}
Now,
\begin{equation}
    \sigma_* \le \norm{M ^{- 1 / 2 } } \max _ i \norm{ X _ i },
\end{equation}
and Gaussian concentration gives (as $\max$ and $\norm{ \cdot} $ are 1-lipschitz),
\begin{equation}
    \sigma _ * \le \bdE \br{ \max _ i \norm{ X _ i } } + \sqrt{ 2 z }, \quad \text{with probability at least $1 - \exp(-z)$}. 
\end{equation}
Still using Gaussian concentration, $\bdP \p{ \norm{X_i} > \sqrt{ d } + t } \le \exp(-t^2/2)$, thus 
\begin{align*}
    \bdE \br{ \max _ i \norm{ X _ i } } &= \int _ 0 ^{ + \infty } \bdP \p{ \max _ i \norm{ X _ i } > t } d t \le t _ 0 + n \int _{ t _ 0 } ^{+ \infty } \bdP \p{ \norm{ X } > t } dt \\
    & \le t _ 0 + n \int _ { t _ 0 } ^{+ \infty } \exp \p{ - (t - \sqrt d)^2/2} dt, \quad \text{ if $t_0 \ge \sqrt{ d }$} \\
    & \le t _ 0 + n \int _ { t _ 0 } ^{+ \infty } t \exp \p{ - (t - \sqrt d)^2/2} dt \\
    & = t _ 0 + n \int _ { t _ 0 } ^{+ \infty } (t - \sqrt d ) \exp \p{ - (t - \sqrt d)^2/2} dt \\
    &+ n \sqrt d \int _ { t _ 0 } ^{+ \infty } \exp \p{ - (t - \sqrt d)^2/2} dt \\
    & = t _ 0 + n \exp \p{ - \frac{(t _ 0 - \sqrt{ d }) ^2 }{2}} + n \sqrt{2 \pi d} \bdP ( Z \ge t _ 0), \quad Z \sim \bcN( \sqrt d , 1) \\
    & \le t _ 0 + n \exp (- (t _ 0 - \sqrt{ d })^2/2) + n \sqrt{2 \pi d } \frac{1}{2} \exp (- (t _ 0 - \sqrt d ) ^2 / 2 ) .
\end{align*}
As for $Z \sim \bcN(0, 1)$ and $t \ge 0$, we have $\bdP ( Z \ge t ) \le \frac{1}{2} \exp(-t^2/2)$. 
Picking $t _ 0 = \sqrt{ d} + \sqrt{ 2 \log n }$ gives 
\begin{equation}
    \bdE \br{ \max _ i \norm{ X _ i } } \le 4 \sqrt d + \sqrt{ 2 \log n },
\end{equation}
which is the desired bound. For this term, we deduce, as $\norm{ M ^ {-1/2} } \le 1$
\begin{align*}
    w _d \sqrt{\log 2 + z } &\le \frac{ \sqrt {2 \log ( 2 n ) } }{n} \p{ 4 \sqrt d + \sqrt{2 \log n } + \sqrt{2z} } \sqrt{\log 2 + z } \\
    &\le \frac{32 \sqrt{d} \log (2n) }{n} \sqrt{ \log 2 + z } \\
    & + \sqrt{ 2z( \log 2 + z ) } \frac{ \sqrt{2 \log (2n) }}{n},
\end{align*}
with probability at least $ 1 - \exp(-z)$. 

\textbf{Gaussian diameter term :} Recall the definition, 
\begin{equation}
    \Delta _ d = \sup _ { u \in S ^{ d - 1 } } \norm{ u } _ \star = \sup _ {u \in S ^{ d - 1 } } n ^{- 1 } \max _ i \abs{ \dotp{ M ^{- 1 / 2 }  X _ i}{u}}.
\end{equation}
Cauchy-Schwarz gives,
\begin{equation}
    \Delta _ d \le \frac{ \max _ i \norm{ M ^{ - 1 / 2 } X _ i }}{n} \le \norm{ M ^{ - 1 / 2 } } \frac{ \max _ i \norm{ X _ i } }{n} \le \norm{ M ^{ - 1 / 2 } } \frac{ 4 \sqrt{d} + \sqrt{ 2 \log ( 2n ) } + \sqrt{ 2z} }{n}
\end{equation}
with probability at least $1 - \exp(-z)$. 

\subsection{Technical lemmas for the Hessian}

\subsubsection{Geometric and deterministic lemmas}

In this section we prove Lemma \ref{lemma:form_of_hessian}, Lemma \ref{lemma:geometric} and Lemma \ref{lemma:deterministic_upperbound} which correspond to 3 different deterministic results used in Appendix \ref{app:lemmas} and Appendix \ref{app:theorems} to bound the Hessian empirical process and to prove the excess risk bounds using localization. The first lemma gives the expression for the Hessian of the Poisson loss under a Gaussian design, the second lemma gives elementary inequalities that we use throughout the proof of the high probability lower bound on the empirical Hessian to control various quantities related to the localization set. Finally, the last lemma gives a deterministic and uniform upper bound on the Hessian of the population loss over the localization set. This lemma is used in the last step of the convex localization lemma to convert localization of the MLE into an upper bound on the excess risk. 

\begin{lemma}\label{lemma:form_of_hessian}
    Let $\beta \in \bdR^d$, we have 
    \begin{equation}
        \bdE \br{ \hess \est L (\beta) } = \hess L (\beta ) = \bdE \br{ \exp\p{ \dotp{\beta}{X}} X X ^ \top }
    \end{equation}
\end{lemma}

\begin{proof}[Proof of Lemma \ref{lemma:form_of_hessian}]
\label{proof:hessianform}
    Setting $ u = \beta / \norm{\beta}$, by rotational invariance of the distribution on $u ^\perp$, we deduce that $\hess L (\beta ) $ has the form
    \begin{equation}
         \bdE [ \exp\p{ \dotp{\beta}{X}} X X ^ \top ] = \lambda_1 u u ^\top + \lambda_ 2 (I_d - uu^\top),
    \end{equation}
    with 
    \begin{equation}
        \lambda_1 = \exp \p{ \frac{ \norm{ \beta } ^2 }{2 }} \p{ 1 + \norm{ \beta } ^ 2 }, \qquad \lambda_2 = \exp \p{ \frac{ \norm{ \beta } ^2 }{2 }}.
    \end{equation}
    This yields
    \begin{equation}
        \hess L ( \beta ) = \exp \p{ \frac{\norm{ \beta} ^2 }{ 2 } } \p{ I_d + \norm{ \beta } ^2 u u ^\top  }.
    \end{equation}
\end{proof}

\begin{lemma}\label{lemma:geometric}
    Let $r \in [0,1]$ and $\beta \in B_H \p{ \beta_* , r \exp \p{ \frac{B^2}{4} } }$, then 
    \begin{gather}
        \norm{ \beta - \beta_* }^2 \le r ^2, \quad \dotp{ \beta - \beta _ * }{ u_* }^2 \le \frac{ r ^2 }{ B ^ 2 }; \label{eq:geo1} \\
        \norm{ \beta } ^2 \le B^ 2 + 2 r + r^2; \label{eq:geo2} \\
        \text{If} \quad B = \norm{ \beta _ * }, \quad \norm{ \beta } ^ 2 \ge B^2 - 2r, \quad \text{if, in addition,} \quad r \le \frac{1}{2}, \quad \text{then} \quad \abs{ \norm{ \beta } - B } \le \frac{ 2 }{ B }. \label{eq:geo3}
    \end{gather}
\end{lemma}

\begin{proof}[Proof of Lemma \ref{lemma:geometric}]
\label{proof:geometric}
    Since $\beta \in B_H \p{ \beta _ * , r \exp \p{ \frac{B^2}{4} } }$, we have 
    \begin{equation}
        \norm{ \beta - \beta _ * } ^2 + B ^2 \dotp{ \beta - \beta _ * }{ u _ *} ^2 \le r ^ 2 ,
    \end{equation}

    which implies \eqref{eq:geo1}. To prove \eqref{eq:geo2} we develop the square, as $ B \ge \norm{ \beta _ * }$
    \begin{align*}
        \norm{ \beta } ^ 2 &= \norm{ \beta _ * } ^ 2 + 2 \dotp{ \beta - \beta _ * }{ \beta _ * } + \norm{ \beta - \beta _ * } ^ 2 \\
        & \le \norm{ \beta _ * } ^ 2 + 2 B \abs {\dotp{ \beta - \beta _ * }{ u _ * }} + \norm{ \beta - \beta _ * } ^ 2 \\ 
        & \le B^2 + 2r + r^2,
    \end{align*}
    By \eqref{eq:geo1}, equation 2. Finally if $ \norm{ \beta _ * } = B $, 
    \begin{align}
        \norm{ \beta } ^2 &= \norm{ \beta _ * } ^ 2 + 2 \dotp{ \beta - \beta _ * }{\beta _ * } + \norm{ \beta - \beta _ * } ^ 2 \\
        & \ge B ^ 2 - 2 B \abs{ \dotp{ \beta - \beta _ * }{ u _ * } } \ge B ^2 - 2r.
    \end{align}
    Applying this and taking $r \le \frac{1}{2}$,  
    \begin{equation}
        \frac{ B }{ 2 } \le B - 1 \underset{\eqref{eq:geo3}}{\le} \norm{ \beta } \underset{\eqref{eq:geo2}}{\le} B + 2 \le 2 B, \label{eq:intermediate_step}
    \end{equation}
    as $B \ge 2$. Giving 
    \begin{equation}
        \frac{ 3}{ 2 } B \abs{ \norm{ \beta } - B } \le (\norm{ \beta } + B ) \abs{ \norm{ \beta } - B } \le 3,
    \end{equation}
    where the first inequality is implied by line \eqref{eq:intermediate_step} and the second one by the upper and lower bounds on $\norm{ \beta } ^2  - B ^2 $ from \eqref{eq:geo2} and \eqref{eq:geo3} equation 1.
\end{proof}

\begin{lemma}\label{lemma:deterministic_upperbound}
    For any $\beta \in B_ H \p{ \beta _ *,  \frac{1}{2} \exp \p{ \frac{ B ^2 }{ 4 } } }$, 
    \begin{equation}
        \hess L (\beta ) \preceq 45 H. 
    \end{equation}
\end{lemma}

\begin{proof}[Proof of Lemma \ref{lemma:deterministic_upperbound}]
    If $B = 2$, then by \eqref{eq:geo2} for $r = 1$, for any $ \beta \in B_H( \beta _ *, \exp(B^2/4)), \norm{ \beta } ^2 \le 7$, hence
    \begin{equation}
        \hess L (\beta) \preceq 8 e ^{3 / 2 } H, \quad \text{for any $\beta \in B_H( \beta _ *, \exp(B^2/4))$}.
    \end{equation}
    If $B > 2$, for any $\beta \in B _ H ( \beta _ *, \exp(B ^2 / 4 ))$ and $v \in S ^{ d - 1 }$,
    \begin{align}
        v^\top \hess L (\beta) v& = e^{ \frac{ \norm{ \beta } ^2 }{ 2 } } \p{ \norm{\beta } ^2 \dotp{u}{v}^2 + 1 } \le e^{ \frac{B ^2 + 3 }{ 2 }}\p{ \norm{\beta } ^2 \dotp{u}{v}^2 + 1 } \\
        & = e^{ 3/2 } e^{ \frac{B ^2 }{ 2 }}\p{ \frac{\norm{\beta } ^2}{B ^2 } B ^2 \dotp{u}{v}^2 + 1 } \quad \text{as $\norm{ \beta } ^2 \le B ^2 + 2 \le 2 B ^ 2 $ \eqref{eq:geo2}} \\
        & \le e^{ 3/2 } e^{ \frac{B ^2 }{ 2 }}\p{ \frac{\norm{\beta } ^2}{B ^2 } B ^2 \dotp{u}{v}^2 + 1 } \le e^{3 / 2 } e^{ \frac{B ^2 }{ 2 }}\p{ 2 B ^2 \dotp{u}{v}^2 + 1 } \\
        & = e^{ 3/2 } e^{ \frac{B ^2 }{ 2 }}\p{  \frac{2 B ^2}{ \norm{ \beta } ^2 } \p{ \dotp{\beta - \beta _ *}{v} + \dotp{\beta _ * }{v} }^2 + 1 } \\
        &\le e^{ 3/2 } e^{ \frac{B ^2 }{ 2 }} \p{  \frac{ 4 B ^2}{ \norm{ \beta } ^2 } \p{ \dotp{\beta - \beta _ *}{v}^2 + \dotp{\beta _ * }{v}^2 } + 1 },
    \end{align}
    as $(a + b) ^2 \le 2 (a ^2 + b^2)$, finally by Cauchy-Schwarz, $\dotp{\beta - \beta _ * }{ v } ^2 \le 1$ by \eqref{eq:geo1} and due to \eqref{eq:geo3} we have $\norm{ \beta } ^2 \ge B ^2 - 2 \ge B ^2 / 2$ which implies $B ^2 / \norm{ \beta } ^2 \le 2$. From that we deduce, 
    \begin{equation}
        v ^\top \hess L (\beta ) v \le e^{3/2} e^{B ^2 / 2 } \p{ 8 B ^2 \dotp{u _ * }{v} ^2 + 9 } \le 10 e ^{ 3 /2 } v^\top H v \le 45 v^\top H v. 
    \end{equation}
\end{proof}

\subsubsection{Lower bounds}

This section gathers the proof of the technical lemmas used during the proofs of lower bounds on the empirical Hessian. Lemma \ref{lemma:empirical_process_reduction} is the lemma used at the beginning of the proof for the lower bound on the empirical Hessian in the large sample size regime, it allows us to control uniformly the Hessian below by reducing the problem to obtaining deviation bound on an empirical process. Lemma \ref{lemma:bound1} is the lemma that controls the deterministic part that emerges when we center the empirical process, it's proven here as a standalone result relying on two sub-lemmas that we also prove in this section. 

\begin{proof}[Proof of Lemma \ref{lemma:empirical_process_reduction}]
\label{proof:emp_process_reduction}
Slicing the right hand side of \eqref{eq:lowerbound},
\begin{align*}
    \frac{1}{n } \sumn \exp \p{ \dotp{\beta }{ X _ i } } \dotp{X_i }{ v } ^ 2 &\ge \frac{\eta ^2 }{ 16 n} \sumn \exp \p{ \dotp{ \beta }{ X _ i } } \bdOne _ { \set{ \dotp{u}{X_i } \in [B - \gamma, B] , \ \abs{ \dotp{ v}{ X _ i } > \eta / 4 } } } \\
    & \ge \frac{ \eta ^2 }{ 16 n } \sum_{k = 0 } ^ {k _ 0 - 1 } \sumn \exp \p{ \dotp{ \beta }{ X _ i } } \bdOne _ { \set{ \dotp{u}{X_i } \in I_k , \ \abs{ \dotp{ v}{ X _ i } > \eta / 4 } } }.
\end{align*}
When $\norm{ \beta _ * } = B$, by \eqref{eq:geo3}, as $r \le 1/2$, we have $\norm{\beta} \ge B - 2/B \ge B - 4 / B$ (this is also trivially true when $\norm{\beta_*} < 2 $ with a zero lower bound). Thus when $\dotp{u}{X_i} \in I_k$, 
\begin{align*}
    \dotp{\beta}{X_i} &= \norm{ \beta } \dotp{u}{X _ i } \\
    & \ge (B - 4 / B ) ( B - (k + 1 ) / B ) \\
    & = B ^2 - k - 5 + 4 (k+1)/B^2 \\
    & \ge B ^2 - k - 5.
\end{align*}
From that, we deduce, 
\begin{equation}
    \exp (\dotp{\beta }{ X _ i } ) \bdOne _ {\set{ \dotp{u}{X_i} \in I_k }} \ge \frac{\exp (B ^2 - k )}{e^5} \bdOne _ { \set{ \dotp{u}{X_i} \in I _ k  }}. 
\end{equation}

Now, if we set $\eta_* = \max ( 5 , B \abs{ \dotp{u_*}{v} } )$, if $\eta_* = 5 $, $\eta \ge \eta_* / 5$. Otherwise, if $\eta _ * = B \abs{ \dotp{ u_* }{ v } } \ge 5$, notice that this implies $B > 2$ thus $B = \norm{ \beta _ * }$ and
\begin{equation*}
    B \abs{ \dotp{u}{v} } \ge \frac{B^2}{ \norm{ \beta } } \abs{ \dotp{u_*}{v} } - \frac{B}{ \norm{ \beta } } \abs{ \dotp{\beta - \beta _ * }{ v }} \ge \frac{B \abs{ \dotp{u_*}{v} }}{ 2 } - 2 \ge \frac{\eta_*}{10},
\end{equation*}
as $\norm{ \beta } \le 2B$ by \eqref{eq:geo2} combined with $\abs{ \norm{ \beta } - B }  \le \frac{2}{B}$ by \eqref{eq:geo3} which gives $\norm{ \beta } \ge B / 2$. This gives, 
\begin{equation*}
    \frac{1}{n} \sumn \exp ( \dotp{ \beta }{X _i } ) \dotp{ v}{ X _ i } ^2 \ge \frac{\eta_*^2 }{ 1600 e^5 n } \sum _ { k = 0 } ^{k _ 0 - 1 } \sumn \exp ( B ^ 2 - k ) \bdOne _ { \set{ \dotp{u}{X_i } \in I _ k , \ \abs {\dotp{v}{X _ i } } > \eta/ 4  }}.
\end{equation*}
\end{proof}

\begin{proof}[Proof of Lemma \ref{lemma:bound1}] 
\label{proof:proof_det_lb_hess}
We first need the two following helpers lemmas. 
\begin{lemma}\label{lemma:probabound}
    Let $g \sim \bcN ( 0, 1 ) $ and $ X \sim \bcN (0 , I _d )$, let $p_{k , u , v} $ be as previously defined. There exists a numerical constant $c_ 1 \ge \frac{1}{400}$ such that
    \begin{equation*}
        p_{k , u , v } \ge c_1 \bdP \p{ g \in I _ k }.
    \end{equation*}
\end{lemma}

\begin{proof}
    Write 
    \begin{equation*}
        v = \dotp{u}{v} u + \alpha w, \quad u \perp w, \quad w \in S ^ { d - 1 }, \quad \alpha = \sqrt{ 1 - \dotp{u}{v} ^ 2 }. 
    \end{equation*}
    By properties of Gaussian vectors $ \dotp{u}{X}$ and $\dotp{w}{X}$ are independent standard normal variables. Recall that $\eta = \max(1, B \abs{ \dotp{u}{v} })$ as defined in \eqref{def:slicings}. We distinguish two cases for the value of $\abs{ \dotp{u}{v} }$:
    
    \textbf{Case 1:} $ \abs{ \dotp{u}{v} } > 1 / B $. In this case $ \eta = B \abs{ \dotp{u}{v } }$. Then if $ \abs{ \dotp{w }{ X }} \le 1 / 4 $ and $ \dotp{u}{X} \in I _ k $, for any $ k \le k _ 0 - 1 $
    \begin{align*}
        \abs{ \dotp{v}{ X } } &\ge \abs{ \dotp{u}{v} } \abs{ \dotp{u}{ X } } - \alpha \abs{ \dotp{w}{X} } \\
        & \ge \abs{ \dotp{u}{v} } B - \abs{ \dotp{u}{v} } \frac{k + 1 }{ B } - \frac{1}{ 4 } \\
        & \ge \frac{B \abs{ \dotp{u}{v}} }{4} = \frac{ \eta }{ 4 },
    \end{align*}
    as $B \abs{ \dotp{ u }{v } } \ge 1$ and $(k + 1)/ B \le 1$ for any $k$. Thus, 
    \begin{equation*}
        p_{k , u , v } \ge \bdP \p{ \dotp{u}{X } \in I _ k , \ \abs{ \dotp{w }{X} } < 1 / 4 } = \bdP \p{g \in I _ k } \bdP \p{ \abs g < 1 / 4 },
    \end{equation*}
    and $ \bdP \p{ \abs g < 1 / 4 } > \bdP \p{ \abs g > 3 }$.
    
    \textbf{Case 2:} $\abs{ \dotp{u}{v} } \le 1 / B $. In this case $ \eta = 1 $, $ \alpha \ge \sqrt{ 1 - 1/B^2 } \ge 1 / 2$. Then if $\abs{ \dotp{w}{X} } > 3 $ and $\dotp{u}{X} \in I _ k $, we have 
    \begin{align*}
        \abs{ \dotp{v}{X} } &\ge \alpha \abs{ \dotp{w}{X} } - \frac{1}{B} \abs{ \dotp{u}{X} } \\
        & \ge \frac{3}{2} - 1 = \frac{1}{2} = \frac{\eta }{ 2 },
    \end{align*}
    so, 
    \begin{equation*}
        p _ {k , u ,v} \ge \bdP \p{ \dotp{u}{X} \in I_ k, \ \abs{\dotp{w}{X}} \ge 3 } = \bdP \p{ \dotp{u}{X} \in I_ k} \bdP \p{ \abs{\dotp{w}{X}} \ge 3 }.
    \end{equation*}
    Concluding the proof of the first lemma. 
\end{proof}
The second lemma we need is the following Gaussian result: 
\begin{lemma}\label{lemma:sumbound}
    Let $ g \sim \bcN (0,1) $ and $A > 1 $. For any integer $k > 0$, define 

    \begin{equation}\label{object:interval}
        I_k^A = \br{ A - \frac{k + 1 }{A}, A - \frac{k}{A} }.
    \end{equation}

    Then, for any integer $k_0 \ge 1$,
    \begin{equation*}
        c_{0,A} \exp (A^2 / 2 ) \le \sum _ { k = 0 } ^{ k _ 0 - 1 } \exp (A^2 - k) \bdP \p{ g \in I _k^A } \le e c_{0,A} \exp (A^2 / 2 ) 
    \end{equation*}
    with $c_{0,A} = \bdP(g \in [0 , k_0 / A])$. In particular, when $k_0 = \gamma A$ for $\gamma \in [1/2, 1]$, $ 0.19 \le c_{0,A} \le 0.35$.
\end{lemma}

\begin{proof}
    Let $x \in I _ k ^ A$, we have $A ^2 - k - 1 \le A x \le A ^2 - k $, so 
    \begin{equation*}
        \exp ( A ^2 - k - 1 ) \bdP ( g \in I _ k ^ A ) \le \int_{I_k^A } \exp \p{A x - \frac{x^2 }{ 2} } \frac{\brd x}{ \sqrt{2 \pi }} \le \exp (A ^2 - k ) \bdP ( g \in I _ k ^ A ).
    \end{equation*}
    Therefore, 
    \begin{align*}
        \sum _ { k = 0 } ^ {k _ 0 - 1 } \exp (A ^2 - k ) \bdP ( g \in I _ k ^A) &\ge \sum _ { k = 0 } ^ {k _ 0 - 1 } \int_{I_k^A } \exp \p{A x - \frac{x^2 }{ 2} } \frac{\brd x}{ \sqrt{2 \pi }}\\
        & = e^{A ^2 / 2 } \int _ {A - k _ 0 / A } ^ { A } \exp \p{ - \frac{(x- A ) ^2 }{ 2 }} \frac{ \brd x}{ \sqrt{ 2 \pi }} \\
        & = c_{0,A} \exp \p{ \frac{ A ^2 }{ 2 } }
    \end{align*}
    This argument can be mimicked to get the upper bound.     
\end{proof}

We are now in position to prove Lemma \ref{lemma:bound1}: 
    \begin{align*}
        \sum_{ k = 0 } ^ { k _ 0 - 1 } \exp ( B ^ 2 - k ) \eta _ * ^ 2 p _ {k , u ,v } & \ge c _ 1 \eta _ * ^ 2 \sum _ {k = 0 } ^ { k _ 0 - 1 } \exp ( B ^2 - k ) \bdP ( g \in I _ k  ) \quad \text{(Lemma \ref{lemma:probabound})} \\
        & \ge c_ 1 c_{0, B} \eta _ * ^ 2 e^{B ^2 / 2 } \quad \text{(Lemma \ref{lemma:sumbound})}. 
    \end{align*} 
    Finally, recalling that $\eta _ * ^2 = \max(25, B ^2 \dotp{ u _ * }{ v } ^2 ) \ge \frac{B^2 \dotp{ u _ * }{ v } ^2 }{ 2 } + \frac{25}{2}$, we have $\exp(B^2/2) \eta _ * ^ 2 \ge \frac{1}{2} v ^ \top H v $ and since $k_0 / B \ge 1 / 2 $,  
    \begin{equation*}
        \sum_{ k = 0 } ^ { k _ 0 - 1 } \exp ( B ^ 2 - k ) \eta _ * ^ 2 p _ {k , u ,v } \ge \frac{c_1 c_{0 , B}}{2} v ^\top H v,  
    \end{equation*}
    with $c_1 c_{0,B} / 2 \ge \frac{1}{2} \times \frac{1}{400} \times 0.19 \ge \frac{3}{20000} = C_{15}$. 
\end{proof}

\begin{lemma}\label{lemma:small_lemma_lbproba}
    Let $ X \sim \bcN ( 0 , I _d )$, we have 
    \begin{equation*}
        \inf _ { u , v \in S ^{d - 1 } } \bdP ( \dotp{ u }{ X }  \ge 1, \abs{ \dotp{ v }{ X } } \ge 1 ) \ge p _ 0,
    \end{equation*}
    where $p _ 0 = \bdP ( \bcN(0,1) \ge 1)^2 \ge 0.025$. 
\end{lemma}

\begin{proof}[Proof of Lemma \ref{lemma:small_lemma_lbproba}]
    Let $u,v$ be unit vectors in $\bdR^d$ and $X$ a standard gaussian vector, to simplify the notations, denote $Y_1 = \dotp{ u }{ X } $, $Y_2 = \dotp{ v }{ X }$ which are both standard Gaussian. Finally, denote $Z \sim \bcN(0,1)$ such that $Z \indep Y _ 1$, with those notations in hand, due to the Gaussian structure, 
    \begin{equation*}
        Y _ 2 = \rho Y _ 1 + \sqrt{ 1 - \rho ^2 } Z, 
    \end{equation*}
    if we denote $\rho = \dotp{u}{v}$ the covariance between $Y_1$ and $Y _ 2$. Notice that $\rho \in [-1, 1]$, we distinguish two cases, 

    \textbf{Case 1:} $\rho \ge 0$, in this case, if $Y_1 \ge 1$ and $Z \ge 1$, then
    \begin{equation*}
        Y _ 2 = \rho Y _ 1 + \sqrt{ 1 - \rho ^2 } Z \ge \rho + \sqrt{ 1 - \rho ^2 } \ge 1,
    \end{equation*}
    as the function $x \mapsto x + \sqrt{ 1 - x ^2} $ is lower bounded by $1$ on $[0,1]$. We deduce, 
    \begin{equation*}
        \p{Y _ 1 \ge 1, Z \ge 1 } \subset \p{ Y_ 1 \ge 1 , \abs{  Y_ 2 } \ge 1}. 
    \end{equation*}
    Thus, in the case $\rho \ge 0$, $\bdP \p{ Y_ 1 \ge 1 , \abs{  Y_ 2 } \ge 1} \ge \bdP ( \bcN(0,1) \ge 1)^2$. 

    \textbf{Case 2:} $\rho < 0$, in this case, if $Y _ 1 \ge 1$ and $Z \le - 1$, then 
    \begin{equation*}
        - Y _ 2 = - \rho Y _ 1 - \sqrt{ 1 - \rho ^2 } Z \ge \abs{\rho} + \sqrt{ 1 - \rho ^2 } \ge 1
    \end{equation*}
    by the same argument on the function $x \mapsto x + \sqrt{ 1 - x ^2}$, so that $\abs{Y_2} \ge 1$. We deduce, 
    \begin{equation*}
        \p{Y _ 1 \ge 1, Z \le -1 } \subset \p{ Y_ 1 \ge 1 , \abs{  Y_ 2 } \ge 1}.
    \end{equation*}
    Which concludes similarly, as $\bdP(Z \le -1) = \bdP(Z \ge 1)$. 
\end{proof}

\end{document}